\documentclass[a4paper]{article}
\usepackage[utf8]{inputenc}
\usepackage[T1]{fontenc}

\usepackage{
amsmath,
amsthm,
amscd,
amssymb,
stmaryrd,
mathtools
}
\usepackage{wasysym}
\usepackage{nicematrix}
\usepackage{comment}
\usepackage{tikz}
\usepackage{tikz-cd}
\usetikzlibrary{decorations.pathreplacing,calligraphy}
\usetikzlibrary{positioning,arrows,calc}
\usepackage{xspace}

\usepackage{thmtools}
\usepackage{thm-restate}

\usepackage{graphicx}
\usepackage{mdframed}
\usepackage{caption}
\usepackage{enumitem}

\usepackage{url}
\usepackage[pdfencoding=unicode,hidelinks]{hyperref}

\theoremstyle{plain}
\newtheorem{theorem}{Theorem}[section]
\newtheorem{lemma}[theorem]{Lemma} 
\newtheorem{definition}[theorem]{Definition} 
\newtheorem{corollary}[theorem]{Corollary} 
\newtheorem{proposition}[theorem]{Proposition} 
\newtheorem{claim}[theorem]{Claim} 
\newtheorem{question}[theorem]{Question} 
\newtheorem{remark}[theorem]{Remark} 

\newtheorem{itheorem}{Theorem}

\newtheoremstyle{derp}
{3pt}
{3pt}
{}
{}
{\upshape}
{:}
{.5em}
{}
\theoremstyle{derp}
\newtheorem{example}{Example}

\newcommand{\Z}{\mathbb{Z}}
\newcommand{\N}{\mathbb{N}}

\newcommand{\ID}{\mathrm{id}}

\newcommand{\val}{\mathrm{v}}

\newcommand\xqed[1]{%
  \leavevmode\unskip\penalty9999 \hbox{}\nobreak\hfill
  \quad\hbox{#1}}
\newcommand\qee{\xqed{$\fullmoon$}}

\newcommand{\Aut}{\mathrm{Aut}}

\newcommand{\card}{\#}
\newcommand{\Sym}{\mathrm{Sym}}
\newcommand{\Alt}{\mathrm{Alt}}

\newcommand{\smart}{\mathcal{S}}
\newcommand{\smb}{{\blacktriangleright}}
\newcommand{\smd}{{\blacktriangleleft}}
\newcommand{\smp}{{\rhd}}
\newcommand{\smq}{{\lhd}}
\newcommand{\dec}{{\mathrm{dec}}}

\newcommand{\final}{{\mathrm{final}}}
\newcommand{\init}{{\mathrm{init}}}
\newcommand{\aup}{{\mathrm{up}}}
\newcommand{\adown}{{\mathrm{down}}}
\newcommand{\sign}{{\mathrm{sign}}}

\newcommand{\len}{{\mathrm{len}}}
\newcommand{\encode}{{\mathrm{encode}}}

\newcommand{\NC}{{\mathrm{NC}}}
\newcommand{\RTM}{{\mathrm{RTM}}}

\newcommand{\dright}{{d_{0}}}
\newcommand{\dleft}{{d_{1}}}

\newcommand{\dforw}{{d_{\rightarrow}}}
\newcommand{\dback}{{d_{\leftarrow}}}

\newenvironment{turing}[1][]{\arraycolsep=1.4pt\def\arraystretch{0.9}\Big[\NiceMatrix[l,#1,columns-width=4mm]}{\endNiceMatrix\Big]}
\newenvironment{turingc}[1][]{\arraycolsep=1.4pt\def\arraystretch{0.9}\Big[\NiceMatrix[c,#1,columns-width=4mm]}{\endNiceMatrix\Big]}
\newenvironment{tpattern}[1][]{\arraycolsep=1.4pt\def\arraystretch{0.9}\Big(\NiceMatrix[l,#1,columns-width=4mm]}{\endNiceMatrix\Big)}
\newenvironment{tpatternc}[1][]{\arraycolsep=1.4pt\def\arraystretch{0.9}\Big(\NiceMatrix[c,#1,columns-width=4mm]}{\endNiceMatrix\Big)}
\newcommand{\midwid}{2.6cm}
\newcommand{\sidwid}{1.3cm}
\newcommand{\didwid}{0.7cm}
\colorlet{nicegreen}{black!20!green}
\colorlet{niceblue}{black!25!blue}
\colorlet{nicered}{black!10!red}

\newcommand{\widt}[2]{\makebox[0pt][l]{\ensuremath{#2}}\phantom{\ensuremath{\hspace{#1}}}}

\newcommand{\TMtrans}{\rightarrow}
\newcommand{\bitape}{\mathcal{X}}
\newcommand{\ctape}{\mathcal{C}}
\newcommand{\encoding}{E}
\newcommand{\pencoding}{F}
\newcommand{\GRencoding}{\tilde{\encoding}}
\newcommand{\GRpencoding}{\tilde{\pencoding}}
\newcommand{\origsmart}[1]{#1^{(\textrm{o})}}
\newcommand{\tmgroup}{\mathcal{G}}
\newcommand{\duckset}{D}
\newcommand{\ghostset}{G}
\newcommand{\rfun}{\hspace*{-0.05em}/\hspace*{-0.2em}}
\newcommand{\rduck}[2]{#1\rfun_{#2}}
\newcommand{\plusorbit}[1]{+#1}
\newcommand{\wgh}[1]{#1'}
\newcommand{\wghintext}[1]{#1}
\newcommand{\addc}[1]{\mathrm{add}_{#1}}
\newcommand{\duckdomain}{U}
\newcommand{\duckimage}{V}
\newcommand{\cbstructcond}{\mathcal{CB}}
\newcommand{\rcbdist}{l_\mathrm{r}}
\newcommand{\lcbdist}{l_\mathrm{l}}

\title{Distortion element in the automorphism group of a full shift}

\author{%
  Antonin Callard \\
  Université de Caen Normandie, GREYC, Caen, France \\
  \texttt{contact@acallard.net} \\
  \url{https://orcid.org/0000-0002-4673-4881}
  \and
  Ville Salo \\
  University of Turku, Finland \\
  \texttt{vosalo@utu.fi} \\
  \url{https://orcid.org/0000-0002-2059-194X}
}

\begin{document}
\maketitle

\begin{abstract}
  We show that there is a distortion element in a finitely-generated subgroup $G$ of the automorphism group of the full shift, namely an element of infinite order whose word norm grows polylogarithmically. As a corollary, we obtain a lower bound on the entropy dimension of any subshift containing a copy of $G$, and that a sofic shift's automorphism group contains a distortion element if and only if the sofic shift is uncountable. We obtain also that groups of Turing machines and the higher-dimensional Brin-Thompson groups $mV$ admit distortion elements; in particular, $2V$ (unlike $V$) does not admit a proper action on a CAT$(0)$ cube complex. In each case, the distortion element roughly corresponds to the SMART machine of Cassaigne, Ollinger and Torres-Avil\'es.
\end{abstract}

\section{Introduction}
\label{sec:Intro}

We begin with an introduction to automorphism groups and the topic of distortion in Section~\ref{sec:AutomorphismsAndDistortionIntro}, as this is the motivation and context for our main results listed in Section~\ref{sec:Results}.

The proofs of the main results are based on rather different ideas, namely conveyor belts, ``ducking'', dynamics of Turing machines, permutation groups, and also some ideas from computer science, namely reversible computation and logical gates. Some background for these ideas is given in Section~\ref{sec:TMGatesBackground}.

\subsection{Automorphism groups and distortion}
\label{sec:AutomorphismsAndDistortionIntro}

A recent trend in symbolic dynamics is the study of automorphism groups of subshifts. Typical activities include the study of restrictions that dynamical properties of the subshift put on these groups, and in turn constructing complicated automorphism groups or subgroups thereof.

The former activity has been most successful in the low-complexity setting, see~\cite{PaSc22} for a recent account of the state of the art. For example, minimal subshifts with upper entropy dimension less than $1/2$ have amenable automorphism groups \cite{2016-CK-minimal-exponential}, and (as discussed in more depth below) zero-entropy subshifts do not admit elements with exponential distortion \cite{2018-CFKP}.

The latter activity has been most successful on sofic shifts. In particular, a lot is known about the finitely-generated subgroups of automorphism groups of full shifts: see~\cite{Sa20} for a listing of properties that have been exhibited. For instance, let us mention that these groups $G$, while not finitely-generated, contain finitely-generated ``f.g.-universal subgroups'', namely ones that contain isomorphic copies of all finitely-generated subgroups of $G$. The class of subgroups of automorphism groups is also quite robust, being closed under graph products \cite{Sa18a,Sa20}. A classical reference for the study of automorphism groups of transitive SFTs (an important subclass of sofic shifts) is \cite{1988-BLR}.

In this paper, we study the group-theoretic notion of distortion, introduced by Gromov~\cite{Gr93}, in the context of automorphism groups of subshifts. If $G$ is a finitely-generated group, we say $g \in G$ is a \emph{distortion element}, or \emph{distorted}, if $g$ is of infinite order and the word norm $|g^n|$ grows sublinearly (with respect to some, or equivalently any, finite generating set). For groups that are not finitely-generated, we say that an element is distorted if it is distorted in some finitely-generated subgroup. While distortion elements are usually allowed to have finite order, in this paper we focus on distorted elements of infinite order.

Two basic examples of groups with distortion elements are the Heisenberg group with presentation $\langle a, b \;|\; [[a,b], a], [[a,b], b] \rangle$, where the element $[a,b]$ has quadratic distortion, meaning we can represent an element of the form $[a,b]^{\Omega(n^2)}$ by composing $n$ generators; and the Baumslag-Solitar group $\mathrm{BS}(1,2)$ with presentation $\langle a, b \;|\; a^b = a^2 \rangle$ where $a$ is easily seen to be exponentially distorted, meaning the word norm of $a^n$ grows logarithmically.

The previous examples show that distortion elements can appear in nilpotent and metabelian linear groups. It is known that they cannot appear in biautomatic groups~\cite{GeSh91}, certain types of mapping class groups \cite{FaLuMi01}, and the outer automorphism group of the free group~\cite{Al02}. See~\cite{SwJa11,RoMa18,CaFr06,GuLi19,CaCo20,FrHa06,Pe20,Na21} for other distortion-related works.

\smallskip
Getting back to automorphism groups, it is an open problem (that we solve in the present paper) whether the automorphism group of any subshift can contain a distortion element~\cite{2018-CFKP}. It is not known whether the Heisenberg group~\cite{1990-KR} or the Baumslag-Solitar group $\mathrm{BS}(1,2)$ embed in $\Aut(A^\Z)$, or indeed in the automorphism group of any subshift, and these problems stay open. (It is also open whether the additive group of dyadic rationals $\Z[\frac12] \leq \mathrm{BS}(1,2)$ embeds in $\Aut(A^\Z)$~\cite{1988-BLR}.) 

Besides being an interesting group-theoretic notion, the quest for distortion elements in automorphism groups of subshifts is motivated by several purely symbolic dynamical considerations. First,~\cite[Theorem~1.2]{2016-CK-minimal-exponential} shows that finitely-generated torsion-free subgroups of the automorphism group of a subshift of polynomial complexity are virtually nilpotent. See~\cite[Theorem 5.5]{2016-DDMP} for a similar conclusion for inverse limits of bounded step nilsystems. If we could rule out distortion in such examples, we could conclude virtual abelianness.

Second, it is known that the Baumslag-Solitar group, more generally any group with an exponentially distorted element, does not embed in the automorphism group of a zero-entropy subshift~\cite{2018-CFKP}. More precisely, it was observed there that the Morse-Hedlund theorem allows one to translate a distortion element into a lower bound on the complexity of a subshift. This is notable, as this is the only known restriction for automorphism groups of general zero-entropy subshifts. Thus, distortion looks like a natural candidate for restrictions on automorphism groups of general subshifts (as far as the authors know, no restrictions are known on countable subgroups of automorphism groups of general subshifts).

Third, distortion is tied to an intrinsic notion in automorphism group theory, namely the growth of the \emph{radius} (a.k.a.\ range) of the automorphism, when seen as a cellular automaton. Namely, distortion in the group sense implies sublinear growth of the radius~\cite{2019-CFK}. It is not immediately obvious that even sublinear radius growth is possible (indeed this was left open in~\cite{2019-CFK}), but several examples of sublinear radius growth have been constructed. The most relevant for us is the observation from~\cite{2017-GS} that one can even obtain sublinear radius growth in the automorphism group of a full shift: the so-called \emph{SMART machine}, when simulated by an automorphism, gives rise to such growth.

While distortion elements have not previously been exhibited in automorphism groups of subshifts, some facts are known about their dynamics (mostly related the notion of radius). Links to expansive directions and Lyapunov exponents are shown in~\cite{2019-CFK}. A related result is shown in~\cite{BiDoMa22}, namely distortion elements of automorphism groups of general expansive systems can not themselves be expansive. Links to the dimension group action and inertness are discussed in~\cite{2019-CFK,Sc20}.

\subsection{Results}
\label{sec:Results}

The main result of the present paper is that the automorphism group of some full shift (thus any full shift by standard embedding theorems~\cite{1990-KR}), contains a distortion element with ``quasi-exponential'' distortion, in the sense that the distortion function grows like $\exp(\sqrt[4]{\Omega(n)})$. It is more convenient to work directly with word norms than with the distortion function, so we take this approach in the paper. Note that for well-behaved functions, the word norm growth is just the inverse of the distortion function.

\begin{restatable}{itheorem}{distortioneveryfullshift}
\label{thm:distortion-every-full-shift}
For any non-trivial alphabet $A$, the group $\Aut(A^\Z)$ has an element $g$ of infinite order such that $|g^n|_F = O(\log^4 n)$ for some finite set $F$.
\end{restatable}

Here, by a \emph{non-trivial alphabet} we mean a finite set $A$ with $2 \leq |A| < \infty$; we also use the standard shorthand $\log^4 n = (\log n)^4$.

A simple counting argument shows that the word norms of $n$th powers of a group element cannot be $o(\log n)$ with respect to a fixed finite generating set. For our specific automorphism, one can strengthen this: the radius of $g^n$ as a cellular automaton is $\Theta(\log n)$, so the true growth of word norms of powers of our automorphism is between $\Omega(\log n)$ and $O(\log^4 n)$.

Our theorem solves the second subquestion of~\cite[Question~5.1]{2018-CFKP} in the affirmative. Most of the present paper deals with the proof of this theorem. The element $g$ in this theorem is essentially the SMART machine~\cite{2017-COT}, so morally this also confirms a conjecture of~\cite{2017-GS}, although the embedding we use is slightly more involved than the specific one considered in~\cite{2017-GS}. The group we use in the proof is given in Lemma~\ref{lem:better-bounds-for-smart}.

\smallskip
The generators of our group are relatively simple, but we have little idea what kind of group they generate. The finitely-generated group $\langle F \rangle$ can of course be taken to be larger (the distortion function can only become faster-growing this way), so one can take a more canonical choice of generators, say, all reversible cellular automata with biradius $1$ (on the huge alphabet we use).

One can perform some further massage to get a simpler-sounding example: it is known that the automorphism group of a full shift contains so-called finitely-generated f.g.-universal subgroups, namely ones containing copies of all f.g.\ groups of reversible cellular automata~\cite{Sa22}. Any such group can be used in the result (although the element $g$ will be more complicated). In particular, one can pick as $F$ the symbol permutations and the partial shift $\sigma \times \mathrm{id}$ on the product full shift $\{0,1\}^\Z \times \{0,1,2\}^\Z$.

\medskip
From the main theorem, we obtain several corollaries of interest, which are proved in Section~\ref{sec:Corollaries}. First, we obtain the characterization of the class of sofic shifts whose automorphism groups have distortion elements.

\begin{restatable}{itheorem}{Sofic}
\label{thm:Sofic}
Let $X$ be a sofic shift. Then $\Aut(X)$ contains a distortion element if and only if $X$ is uncountable.
\end{restatable}

It is well-known that for sofic shifts,  uncountability is equivalent to having positive entropy.
\smallskip

As another immediate consequence, using the argument of~\cite{2018-CFKP} we obtain that the automorphism group of a full shift cannot be embedded in the automorphism group of a low-complexity subshift. Recall that the \emph{lower entropy dimension}~\cite{2011-Meyerovitch} of a subshift is defined by the formula
\[ \underline{D}(X) = \liminf_{k \to \infty} \frac{\log(\log N_k(X))}{\log k}, \]
where $N_k(X)$ is the number of words of length $k$ that appear in $X$. The lower entropy dimension of a (one-dimensional) subshift with positive entropy is of course $1$. The \emph{upper entropy dimension} is defined analogously, with $\limsup$ in place of $\liminf$.

\begin{restatable}{lemma}{LowComplexity}
\label{lem:LowComplexity}
Let $X$ be a subshift with lower entropy dimension less than $1/d$. If $f \in \Aut(X)$ satisfies $|f^n| = O(\log^d n)$, then $f$ is periodic.
\end{restatable}

\begin{itheorem}
\label{thm:LowerEntropyDim}
The group $\Aut(A^\Z)$ has a finitely-generated subgroup $G$ such that every subshift $X$ with $G \leq \Aut(X)$ has lower entropy dimension at least $1/4$.
\end{itheorem}

Theorem~\ref{thm:LowerEntropyDim} is of course an immediate corollary of Lemma~\ref{lem:LowComplexity}. It states a \emph{low-complexity restriction} on the automorphism group, i.e.\ it states that automorphism groups of subshifts with low enough complexity (growth of the number of admissible words) cannot have some property. The above theorem seems to be the first low-complexity restriction on automorphism groups where
\begin{enumerate}
\item the complexity bound is superpolynomial,
\item there are no additional dynamical restrictions, and
\item the prevented behavior can be exhibited in the automorphism group of another subshift.
\end{enumerate}
There are previously known restrictions satisfying any two of these items. For 1.\&2., zero entropy prevents exponential distortion \cite{2018-CFKP}; for 1.\&3., \cite{2016-CK-minimal-exponential} shows that if $X$ is minimal and has upper entropy dimension less than $1/2$, then it is amenable (while $\Aut(A^\Z)$ is not); for 2.\&3.\ (very low complexity restrictions) there are many results, see~\cite{PaSc22}.

\medskip
The subgroup where our distortion element lies can itself be seen as a group of Turing machines, indeed restricting its action to a certain sofic subshift directly gives rise to a subgroup of the group $\RTM(n, k)$ studied in~\cite{BaKaSa16}, leading to the following theorem.

\begin{restatable}{itheorem}{TuringMachines}
\label{thm:TuringMachines}
Let $n \geq 2, k \geq 1$. Then the group of Turing machines $\RTM(n, k)$ contains a distortion element; indeed there is a finitely-generated subgroup $G = \langle F \rangle$ and an element $f$ such that $|f^n|_F = O(\log^4 n)$.
\end{restatable}

All groups of Turing machines in turn embed in the higher-dimensional Brin-Thompson $mV$ for $m \geq 2$ introduced by Brin~\cite{Br04}, and we obtain the following.

\begin{restatable}{itheorem}{TwoV}
\label{thm:TwoV}
The Brin-Thompson group $mV$ contains a distortion element; indeed there is an element $f$ such that $|f^n| = O(\log^4 n)$.
\end{restatable}

This theorem provides a new restriction for geometries of $2V$. Namely, it is known that Thompson's group $V$ admits a proper action by isometries on a CAT$(0)$ cube complex~\cite{Fa05}. By~\cite[Theorem~1.5]{Ha21}, a group with distortion elements does not admit such an action, thus:

\begin{corollary}
\label{cor:2Vnocubeaction}
The Brin-Thompson group $mV$ does not act properly on a CAT$(0)$ cube complex for $m \geq 2$.
\end{corollary}

Of course, a similar fact is true for the other groups where we exhibit distortion elements.

\medskip
We conclude with previously known (but possibly not well-known) related distortion facts that are easy to prove. First, the fact the automorphism group of a full shift contains finitely-generated \emph{subgroups} that are distorted is essentially classical, namely $F_2 \times F_2$ embeds in $\Aut(A^\Z)$~\cite{1990-KR} and has subgroups with arbitrarily bad (recursive) distortion essentially by~\cite{Mi68}. To give a more down-to-earth example, $\Z_2 \wr \Z^2$, which embeds in $\Aut(A^\Z)$ by \cite{Sa20}, contains a polynomially distorted copy of itself, by a nice geometric argument~\cite{DaOl11}. One can also construct distorted subgroups directly by more intrinsic automorphism group techniques.

Second, in the setting of general expansive homeomorphisms, finding distortion elements is very easy. Namely, if $\mathbb{S}$ is the invertible natural extension of the $\times 2$-map on the circle, $\mathbb{S}^n$ contains a natural copy of $\mathrm{GL}(n, \Z)$ by simply summing tracks to each other~\cite{KoSa21}. For $n = 3$, the group $\mathrm{GL}(n, \Z)$ contains the Heisenberg group, thus has distortion elements.

\subsection{Turing machines and gates}
\label{sec:TMGatesBackground}

While our results in the previous section are stated fully in terms of homeomorphism groups, our \emph{proof methods} rather belong to the theories of dynamical Turing machines and of reversible gates. In this section, we outline some history of these ideas.

\subsubsection{Turing machines}
\label{sec:TMTI}

As mentioned in Section~\ref{sec:AutomorphismsAndDistortionIntro}, our automorphism group element simulates a ``Turing machine'', i.e.\ a dynamical system where a single head moves over an infinite tape of arbitrary data (over a fixed finite alphabet), and all the action happens near the head (which may move around the tape, such movement depending on the content of the tape; or modify said content). The dynamics of Turing machines, also known as one-head machines, is an important branch of symbolic dynamics. This can be seen as initiated in the 1997 paper of K\r{u}rka~\cite{1997-Kurka}, which explicitly defined the moving-head and moving-tape dynamics of Turing machines (although many relevant dynamical ideas appeared in the literature before this~\cite{Ho66,Ro75,Mo91}).

One of the most-studied behaviors of Turing machines is aperiodicity, meaning that the action of the Turing machine has no periodic points. This property is particularly interesting in the moving tape model, where the head is seen as fixed and only the tape moves.  K\r{u}rka originally conjectured that Turing machines cannot be aperiodic, but an explicit aperiodic Turing machine was exhibited in 2002 by Blondel, Cassaige and Nichitiu~\cite{BlCaNi02} (inspired by a technique of Hooper from 1966~\cite{Ho66}). Later, reversible aperiodic Turing machines (ones whose action is a homeomorphism) were found, the first by Kari and Ollinger~\cite{2008-KO}. This culminated in the discovery of the SMART machine $\smart$ by Cassaigne, Ollinger and Torres-Avilés~\cite{2017-COT}, a machine with only four states and three tape-letters, which is reversible and aperiodic, and whose moving-tape dynamics is a minimal homeomorphism on the Cantor space, see also~\cite{2018-Ollinger-talk}.

Turing machines, in the moving-head dynamics where the tape is not shifted and the head moves over it, can be directly seen as automorphisms of a sofic shift~\cite{BaKaSa16}. In fact, it is well-known that Turing machines can be ``embedded'' into automorphism groups of full shifts $\Aut(A^\Z)$. There are multiple ways of doing so; in this paper, we use the conveyor belt technique similar to the one used in~\cite{2017-GS}.

\smallskip
For the purpose of establishing distortion, the first important consideration, already discussed in Section~\ref{sec:AutomorphismsAndDistortionIntro}, is the ``speed'' of a Turing machine: a Turing machine with positive speed, meaning the existence of tape contents such that the head moves to infinity at a positive rate, could not possibly give rise to a distortion element. This is because the linear movement of the head (even on a single configuration) means that the radius of powers of the corresponding automorphism must grow at a linear rate as well, which prevents distortion.

It was shown in~\cite{Je14} that all aperiodic Turing machines have zero speed, and in~\cite{2017-GS} this was strengthened by proving that the maximal offset by which such a machine can move in $t$ time steps is $O(t/\log t)$. For the SMART machine $\smart$, more is known: in $t$ steps, it can only move by an offset of at most $O(\log t)$. This makes $\smart$ a perfect candidate for a distortion element of a subshift automorphism group, and indeed it was conjectured in~\cite{2017-GS} that it is one.

\subsubsection{Gates}
\label{sec:GatesTI}

The next ideas come from the study of reversible gates. By this, we refer to the study of permutation groups acting on (a sublanguage of) $A^n$, where $A$ is a finite alphabet, that are generated by ``reversible gates'': i.e.\ permutations that only consider a bounded subset of coordinates at a time. More precisely, if $k \leq n$ and $\pi \in \Sym(A^k)$, then we can apply $\pi$ to the subword starting at $i$ by the formula $\hat \pi(u \cdot v \cdot w) = u \cdot \pi(v) \cdot w$ where $u \in A^i, v \in A^k, w \in A^{n - k - i}$. From now on, we use the term ``gate'' for reversible gates, and ``classical gate'' to refer to the usual not necessarily reversible gates (in the few places where they are needed).

A fundamental lemma in this topic is that $\Alt(A^n)$ admits a generating set with bounded $k$, namely it is generated by the even permutations of $A^2$ if the cardinality $\card A$ is at least $3$ (when we consider them as gates, and allow their applications at any position $i = 0,  \ldots, n - 2$). A more complete statement appears in~\cite{Sa18}, while earlier proofs are given in~\cite{Se16,BoKaSa17,Bo19}.

\smallskip
The connection between gates and Turing machines is as follows. Let us consider generalized Turing machines in the sense of~\cite{BaKaSa16}, meaning the machine can look at and modify multiple cells at once, although only at a bounded distance from the head. Now, walking on a cyclic tape containing an element of $A^n$, we can apply permutations of $A^k$ at different relative positions $i$: simply move by $i$ steps, apply the permutation locally, and then move back by $-i$ steps. The above paragraph translates to the fact that there is a finite set of generalized Turing machines that can perform any even permutation of the tape content (relative to the head position). Actually, it turns out that since Turing machines carry a state, $k = 1$ suffices, i.e.\ the generating Turing machines need not be of a generalized type.

\subsection*{Acknowledgements}

We would like to thank Anthony Genevois for pointing out Corollary~\ref{cor:2Vnocubeaction}. We thank Pierre Guillon for helpful discussions.

\section{Definitions}\label{sec:defs}

\subsection{General notions}

We have $\N = \{0,1,2,\ldots\}$, $\Z_+ = \N \setminus \{0\}$, and $\Z_{\ell} = \Z/\ell\Z$ is integers modulo $\ell$. For $S$ a finite set, we denote by $\card S$ the cardinality of $S$. For $i,j \in \N$, denote $\llbracket i,j \rrbracket = \{ n \in \N : i \leq n \leq j \}$ and $\llbracket n \rrbracket = \llbracket 0, n-1 \rrbracket$. If $w \in \{0,1,\ldots,k-1\}^*$, write $\val_k(w)$ for the value $w$ represents in base $k$ (the leftmost digit having the highest significance by default), i.e.\ $\val_k(w) = \sum_{i=0}^{|w|-1} k^{|w|-1-i} w_{i}$; and we write $n_{(k)} \in \{0, \ldots, k-1\}^*$ for the number $n \in \N$ written in base $k$ (with length determined from context or specified in text), i.e.\ $\val_k(n_{(k)}) = n$.

For $\Sigma$ a finite set, called an \emph{alphabet}, denote $\Sigma^* = \bigcup_{n = 0}^{\infty} \Sigma^n$ the set of finite words over $\Sigma$. For $w \in \Sigma^*$, denote $\len(w)$ the length of $w$, i.e.\ the integer $n$ such that $w \in \Sigma^n$. For a word $w \in \Sigma^*$, denote $\overline{w}$ the reverse (or ``mirror image'') of $w$, i.e.\ if $w = w_0 \cdot w_1 \cdots w_{n-1}$, then $\overline{w} = w_{n-1} \cdot w_{n-2} \cdots w_0$. For $w \in \Sigma^n$ and $J \subseteq \llbracket 0, n-1 \rrbracket$, define $w|_{J} = w_{j_0} \cdot w_{j_1} \cdots w_{j_k}$ the restriction of $w$ to $J$, for $J = \{j_0,\dots,j_k\}$ and $j_0 \leq \dots \leq j_k$. Given $a \in \Sigma$ and $j \in \llbracket 0,n-1 \rrbracket$, the cylinder $[a]_j$ denotes the set of words $\{w \in \Sigma^n \mid w_j = a\}$. Usually our alphabets are \emph{non-trivial}, by which we mean $|\Sigma| \geq 2$.

\medskip

In Lemma~\ref{lem:nc1-conditioning}, we denote $\NC^1$ for ``Nick's Class'' of complexity of level $1$, i.e.\ the class of languages $L \subseteq \Sigma^*$ such that $L$ is decidable by Boolean circuits with a polynomial number of gates, with at most two inputs and depth $O(\log n)$ (see for example~\cite{ArBa09}). The reader need not be familiar with this class to follow our argument. The main technical result we need is Barrington's theorem from~\cite{1989-Barrington} (but this is also proved from scratch in our context).

\medskip

For $a,b$ elements of a group, the commutator of $a$ and $b$ is $[a, b] = a^{-1}b^{-1}ab$. The conjugation convention is $a^b = b^{-1} \circ a \circ b$. If $\pi \in \Sym(A)$ is a permutation, we may regard it as a permutation of $A \times B$ by $\pi((a, b)) = (\pi(a), b)$. Our groups always act from the left. If $g_1, \ldots, g_n$ are commuting elements of a group, we write $\prod_{i=1}^n g_i$ for their ordered product $g_n \cdots g_1$. In groups of bijections on a set (which almost all our groups are), we denote composition by $\circ$.

Given a finitely generated group $G$ generated by the finite set $S$, a \emph{presentation} of $g \in G$ is a word $w = s_n \dots s_1 \in (S \cup S^{-1})^*$ such that $g = s_n \cdots s_1$, and we write $w \equiv g$. The \emph{word norm $\lVert g \rVert_S$ of $g \in G$ relative to $S$} is then the length of a shortest presentation of $g$, i.e. $\lVert g \rVert_S = \min \{ n \in \N : \exists w \in (S \cup S^{-1})^n, w \equiv g \}$. This word norm is also the distance in the Cayley graph of $G$ between $1_G$ and $g$. In this context, an element $g \in G$ is said to be \emph{distorted} if $\lVert g^n \rVert_S = o(n)$.

\medskip
For sets $X,Y$ and $Z$, we say that a map $f : X \to Y$ \emph{lifts} into $\tilde{f} : X \times Z \to Y \times Z$ (or that $\tilde{f}$ is the \emph{lift} of $f$) if $\tilde{f}(x,z) = (f(x),z)$. For $S \subseteq X$ a subset of $X$, and $f : X \to X$, we call \emph{extended restriction} of $f$ to $S$ the map $f\rfun_S : X \to X$ defined as:
\[ f\rfun_S(x) = \begin{cases}
  f(x) & \text{if } x \in S \\
  x & \text{otherwise}
\end{cases}
\]
i.e.\ $f\rfun_S$ is the extension of the restriction $f|_S$ back to the full domain $X$, by fixing elements outside $S$.

\subsection{Subshifts and cellular automata}

Let $\Sigma$ be a finite alphabet. An element $x \in \Sigma^\Z$ is called a \emph{configuration}. An element $w \in \Sigma^*$ is called a \emph{word} or a \emph{pattern}, and a pattern $w \in \Sigma^*$ is said to \emph{appear} in a configuration $x \in \Sigma^\Z$, denoted $w \sqsubseteq x$, if there exists some $i \in \Z$ such that $x_{i+j} = w_j$ for every $j \in \llbracket 0,\len(w)-1 \rrbracket$.

We endow $\Sigma^\Z$ with the product topology. This topology is generated by the cylinders $[a]_j = \{ x \in \Sigma^\Z : x_j = a \}$ for $a \in \Sigma$ and $j \in \Z$. The \emph{left shift} $\sigma : \Sigma^\Z \to \Sigma^\Z$ defined by $\sigma(x)_i = x_{i+1}$ is a $\Z$ action on $\Sigma^\Z$. Closed and shift-invariant subsets $X$ of $\Sigma^\Z$ are called \emph{subshifts}. For $X$ a subshift and $n \in \N$, we denote $\mathcal{L}_n(X)$ the set of finite words of length $n$ that appear in $X$, and $\mathcal{L}(X) = \bigcup_{n \in \N} \mathcal{L}_n(X)$ its \emph{language}.

We say that a subshift $X$ is \emph{sofic} if $\mathcal{L}(X)$ is a regular language.

\medskip
If $X$ and $Y$ are subshifts, a continuous and shift-equivariant map $f : X \to Y$ is called a \emph{morphism}. It is an \emph{endomorphism} if $X = Y$ and an \emph{automorphism} if, in addition, it is bijective (in which case $f^{-1}$ is also an endomorphism). Endomorphisms are sometimes called \emph{cellular automata}, and automorphisms \emph{reversible cellular automata}. For $f : X \to Y$ a morphism between two subshifts, its \emph{radius} (as a cellular automaton) is the minimal $r$ such that $f(x)_i$ is a function of $x_{\llbracket i-r, i+r \rrbracket}$. The \emph{biradius} of an automorphism is the maximum of the radii of $f$ and $f^{-1}$.

\subsection{Turing machines}

In this article, we use Turing machines as a specific kind of action on subshifts. We note that despite the terminology, it is not necessarily helpful to think of them as computational devices.\footnote{We will later perform computation in a group of Turing machines, but this computation is not related to the usual type of Turing machine computation, in that the iteration of a single machine is \emph{not} going going to be used to perform computation.}

\medskip
Let $Q$ be a finite set called the \emph{state set}, and $\Gamma$ be a finite set called the \emph{tape alphabet}. In the model of~\cite{2008-KO},\footnote{This model is equivalent to the usual definition of Turing machines, but handles reversibility better.} a \emph{Turing machine} is a triple $\mathcal{M} = (\Gamma, Q, \Delta)$, where $\Delta \subseteq (Q \times \{+1,-1\} \times Q) \cup (Q \times \Gamma \times Q \times \Gamma)$ is the \emph{transition table}. A transition $(q,\delta,q') \in Q \times \{+1,-1\} \times Q$ is called a \emph{move transition}, and a transition $(q,a,q',b) \in Q \times \Gamma \times Q \times \Gamma$ is called a \emph{matching transition}.

In the rest of this paper, we focus on the action of Turing machines on two families of objects: bi-infinite tapes, and finite cyclic tapes.

\paragraph*{Bi-infinite tapes}~\\
In the alphabet $\Gamma \cup (Q \times \Gamma)$, elements of $H = Q \times \Gamma$ are called \emph{heads}. Denote
\[ \bitape = \{x \in {(\Gamma \cup (Q \times \Gamma))}^\Z \mid \forall i,j \in \Z: i \neq j \implies x_i \in \Gamma \vee x_j \in \Gamma \} \]
the set of bi-infinite tapes with at most one head somewhere. We can associate to $\mathcal{M}$ its so-called \emph{moving-head model}~\cite{1997-Kurka}, i.e.\ the binary relation $\rightarrow_\mathcal{M}$ on $\bitape$ defined by $x \rightarrow_\mathcal{M} x$ if $x \in \bitape$ contains no head (i.e.\ $x \in \Gamma^\Z$); and if $x \in \bitape$ contains a head at position, say $i_0 \in \Z$ with $x_{i_0} = (q,a)$ for some $q \in Q$ and $a \in \Gamma$, then $x \rightarrow_\mathcal{M} x'$ if there exists $t \in \Delta$ such that:
\begin{align*}
  \text{If $t = (q,a,q',b) \in \Delta$}: & \qquad
  x'_i = \begin{cases}
    (q',b) & \text{if } i=i_0 \\
    x_j & \text{otherwise}
  \end{cases} \\
  \text{If $t = (q,\delta,q') \in \Delta$}: & \qquad
  x'_i = \begin{cases}
    a & \text{if } i=i_0 \\
    (q',x_{i}) & \text{if } i=i_0+\delta\\
    x_i & \text{otherwise}
  \end{cases}
\end{align*}

The binary relation $\TMtrans_\mathcal{M}$ on $\bitape$ (denoted $\TMtrans$ for short if the context is clear) is the \emph{reachability} relation. We write $\TMtrans_\mathcal{M}^k$ its $k$-th power, and $\TMtrans_\mathcal{M}^*$ its transitive closure. We say that $\mathcal{M}$ reaches the configuration $x'$ from $x$ in $k \in \N$ steps if $x \TMtrans_{\mathcal{M}}^k x'$. A transition $x \TMtrans_\mathcal{M}^* x'$ is called a \emph{move}.

\smallskip
The machine $\mathcal{M}$ is \emph{deterministic} if $\rightarrow_\mathcal{M}$ defines a partial function, \emph{complete deterministic} if it defines a total function (which is then continuous and, obviously, shift-commuting), and \emph{complete reversible} (or \emph{reversible} for short) if it defines a bijection (which is then a homeomorphism). When $\mathcal{M}$ is complete deterministic (which all our machines are), when using the relation $\rightarrow_\mathcal{M}$ as a function we write it as $T_\mathcal{M} : \bitape \to \bitape$, which is an endomorphism of the subshift~$\bitape$. Similarly, when the machine $\mathcal{M}$ is reversible, it is an automorphism of~$\bitape$.

\paragraph*{Finite cyclic tapes}~\\
The set of \emph{cyclic configurations of length $\ell \in \Z_+$} 
is the set 
\[ \ctape_\ell = \{ x \in {(\Gamma \cup (Q \times \Gamma))}^{\Z/\ell\Z}\mid \forall i,j \in \Z/\ell\Z, i \neq j \implies x_i \in \Gamma \vee x_j \in \Gamma \} \]
of finite configurations containing at most one head. We always assume $\ell \geq 2$ in what follows (the case $\ell = 1$ makes sense, but requires notational modifications and is the least interesting case anyway).

The machine $\mathcal{M}$ defines a binary relation $\rightarrow_\mathcal{M}$ on $\ctape_\ell$ by considering these finite tapes as cyclic, i.e.\ we define $x \rightarrow_\mathcal{M} x$ if $x \in \ctape_\ell$ contains no head (i.e.\ $x \in \Gamma^{\Z/\ell\Z}$); and if $x \in \ctape_\ell$ contains a head at position, say $i_0 \in \Z/\ell\Z$ with $x_{i_0} = (q,a)$ for some $q \in Q$ and $a \in \Gamma$, then $x \rightarrow_\mathcal{M} x'$ if there exists $t \in \Delta$ such that:
\begin{align*}
  \text{If $t = (q,a,q',b) \in \Delta$}: & \qquad
  x'_i = \begin{cases}
    (q',b) & \text{if } i=i_0 \\
    x_i & \text{otherwise}
  \end{cases} \\
  \text{If $t = (q,\delta,q') \in \Delta$}: & \qquad
  x'_i = \begin{cases}
    a & \text{if } i=i_0 \\
    (q', x_i) & \text{if } i=i_0+\delta \bmod \ell \\
    x_i & \text{otherwise}
  \end{cases}
\end{align*}
As above, the relation $\TMtrans_\mathcal{M}$ is called the \emph{reachability} relation. If $\mathcal{M}$ is complete deterministic, the function $\rightarrow_\mathcal{M}$ will be denoted by $T_{\ell,\mathcal{M}} : \ctape_\ell \to \ctape_\ell$. Note that it is an endomorphism of the shift action of $\Z$ (or $\Z_{\ell}$) which translates the cyclic tape around.

\medskip
Determinism, completeness and reversibility are characterized by obvious combinatorial properties. In particular, $\mathcal{M} = (\Gamma, Q, \Delta)$
is complete deterministic if and only if exactly one transition applies at any time:
\[ \forall (q,a) \in (Q \times \Gamma): \card\{ t \in \Delta \mid t = (q,a,\cdot,\cdot) \text{ or } t = (q,\cdot,\cdot) \} = 1. \] Defining the reverse of a transition by $(q,\delta,q')^{-1} = (q',-\delta,q)$ and $(q,a,q',b)^{-1} = (q',b,q,a)$, this reverse relation extends to transition tables with $\Delta^{-1} = \{t^{-1} \mid t \in \Delta\}$: the \emph{reverse} of  $\mathcal{M}$ is then defined by $\mathcal{M}^{-1} = (\Gamma,Q,\Delta^{-1})$, and $\mathcal{M}$ is reversible if both $\mathcal{M}$ and $\mathcal{M}^{-1}$ are complete deterministic.

\smallskip
Finally, for any machine $\mathcal{M} = (Q,\Gamma,\Delta)$, denote by $m: \N \to \N$ its \emph{movement function}, i.e.\ $m(n)$ is the maximal number of cells the machine can visit in $n$ steps. More precisely, $m(n)$ is the length $r-l+1$ of the largest interval $\llbracket l, r \rrbracket \subseteq \Z$ such that there exists a sequence of $n$ steps of computation $x_0 \rightarrow_\mathcal{M} x_1 \rightarrow_\mathcal{M} \dots \rightarrow_\mathcal{M} x_n$ (with $x_i \in \bitape $) such that for every position $i \in \llbracket l, r \rrbracket$, at least one of the tapes $x_k$ ($0 \leq k \leq n)$ has its head at position~$i$.

\section{The SMART machine on cyclic tapes}\label{sec:smart}
Let SMART be the Turing machine $(Q,\Gamma,\Delta)$, where $Q = \{\smb_1,\smd_1,\smp_1,\smq_1\} \cup \{\smb_2,\smd_2,\smp_2,\smq_2\}$, $\Gamma = \{0,1,2\}$ and $\Delta$ is the following transition table:

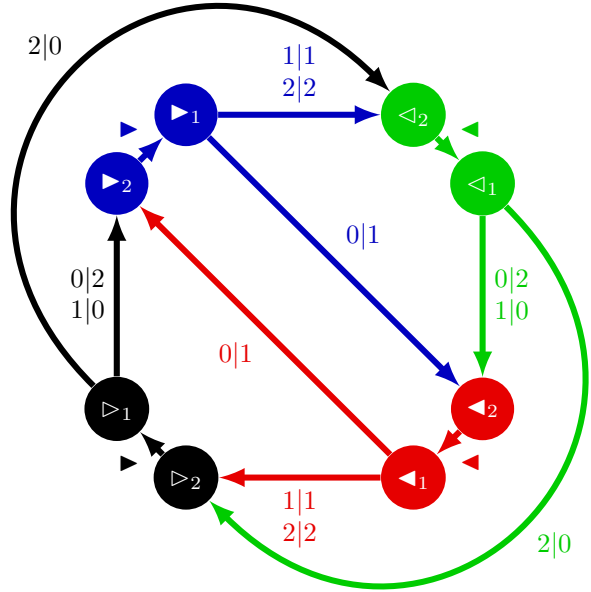
\begin{center}
  \begin{tikzpicture}[scale=1.3]
    \clip (-3,-3) rectangle (3,3);
    \node[white,fill=niceblue,circle,text depth=2pt] at (-1.85,1.15) (b2) {$\smb_2$};
    \node[white,fill=niceblue,circle,text depth=2pt] at (-1.15,1.85) (b1) {$\smb_1$};
    \node[white,fill=nicegreen,circle,text height=7pt] at (1.15,1.85) (q2) {$\smq_2$};
    \node[white,fill=nicegreen,circle,text height=7pt] at (1.85,1.15) (q1) {$\smq_1$};
    \node[white,fill=black,circle,text height=7pt] at (-1.85,-1.15) (p1) {$\smp_1$};
    \node[white,fill=black,circle,text height=7pt] at (-1.15,-1.85) (p2) {$\smp_2$};
    \node[white,fill=nicered,circle,text depth=2pt] at (1.85,-1.15) (d2) {$\smd_2$};
    \node[white,fill=nicered,circle,text height=7pt] at (1.15,-1.85) (d1) {$\smd_1$};

    \draw[-latex,line width=0.8mm,niceblue] (b2.north east) to node[auto,niceblue] {$\blacktriangleright$} (b1.south west);
    \draw[-latex,line width=0.8mm,niceblue] (b1.south east) to node[auto,niceblue] {$0|1$} (d2.north west);
    \draw[-latex,line width=0.8mm,niceblue] (b1.east) to node[auto,niceblue,text width=0.45cm] {$1|1$\\$2|2$} (q2.west);
    \draw[-latex,line width=0.8mm,nicered] (d2.south west) to node[auto,nicered] {$\blacktriangleleft$} (d1.north east);
    \draw[-latex,line width=0.8mm,nicered] (d1.north west) to node[auto,nicered] {$0|1$} (b2.south east);
    \draw[-latex,line width=0.8mm,nicered] (d1.west) to node[auto,nicered,text width=0.45cm] {$1|1$\\$2|2$} (p2.east);
    \draw[-latex,line width=0.8mm,black] (p2.north west) to node[auto,black] {$\blacktriangleright$} (p1.south east);
    \draw[-latex,line width=0.8mm,black] (p1.north) to node[auto,black,text width=0.45cm] {$0|2$\\$1|0$} (b2.south);
    \draw[-latex,line width=0.8mm,black,distance=3.2cm,out=135,in=135] (p1.north west) to node[auto,black] {$2|0$} (q2.north west);
    \draw[-latex,line width=0.8mm,nicegreen] (q2.south east) to node[auto,nicegreen] {$\blacktriangleleft$} (q1.north west);
    \draw[-latex,line width=0.8mm,nicegreen] (q1.south) to node[auto,nicegreen,text width=0.45cm] {$0|2$\\$1|0$} (d2.north);
    \draw[-latex,line width=0.8mm,nicegreen,distance=3.2cm,out=-45,in=-45] (q1.south east) to node[auto,nicegreen] {$2|0$} (p2.south east);
  \end{tikzpicture}

  \small An arrow from $q$ to $q'$ labeled $\blacktriangleright$ (resp. $\blacktriangleleft$) denotes a transition $(q,+1,q')$ (resp.~$(q,-1,q')$). An arrow from $q$ to $q'$ labeled $a|b$ denotes a transition $(q,a,q',b)$.
\end{center}

We refer to $\smb_1,\smb_2,\smd_1,\smd_2$ (resp.\ $\smp_1,\smp_2,\smq_1,\smq_2$) as \emph{filled} and \emph{hollow triangles}.

\begin{remark}
The SMART machine was introduced with a slightly different formalism in~\cite{2017-COT}, and slightly revised in~\cite{2018-Ollinger-talk} (states were renamed and permuted). The machine above adapts the latter in the model of~\cite{2008-KO} for Turing machines: in other words, we duplicate the states. We kindly advise readers already familiar with the SMART machine to read these definitions and propositions carefully. 

Namely, while our SMART machine is in a sense completely equivalent, in the formulas in Proposition~\ref{prop:smart-moves} describing traversals of SMART over zeroes, the patterns corresponding to filled and hollow initial states are of the same length (unlike the corresponding ones in~\cite{2017-COT}). This will be helpful later, when we encode the position in the sweep into the corresponding area on the tape without any extra space.
\end{remark}

In this section, we consider the action of this machine on finite patterns (denoted with rounds brackets) like
\[ \begin{tpatternc}
  1 & 1 & 0^k & 2 \\
  \smb_2 & & &
\end{tpatternc} \]
The argument applies whether or not these are finite subpatterns of a finite cyclic tape, or of an infinite configuration. When specifying a move (with some number of transition steps) between two patterns, it is implicit that the initial and final patterns have the same domain, and the machine does not exit this domain during the intermediate steps. Complete cyclic configurations (where the notation specifies the contents of all $\ell$ cells) will be denoted similarly, but with square brackets.

\begin{proposition}\label{prop:smart-moves}
  (Adapted from~\cite[Lemma 1]{2017-COT}) Let $f(k) = 3^{k+1}-2$. For all~$k$, $s_* \in \{0,1,2\}$ and $s_+ \in \{1,2\}$, the following moves hold:
  \footnotesize
  \begin{align*}
      M_\smb(k):&\ \begin{tpattern}
          s_+ & 0^k & s_* \\
          \smb_2 & &
      \end{tpattern}
      \TMtrans^{\scriptscriptstyle{f(k)}}
      \begin{tpattern}
          s_+ & 0^k & s_* \\
          & & \smb_1
      \end{tpattern}
      \qquad &
      M_\smd(k):&\ \begin{tpattern}
          s_* & 0^k & s_+ \\
          & & \smd_2
      \end{tpattern} 
      \TMtrans^{\scriptscriptstyle{f(k)}}
      \begin{tpattern}
          s_* & 0^k & s_+ \\
          \smd_1 & &
      \end{tpattern} \\
      M_\smp(k):&\ \begin{tpattern}
          s_* & 0^k & s_+ \\
          \smp_2 & &
      \end{tpattern}
      \TMtrans^{\scriptscriptstyle{f(k)}}
      \begin{tpattern}
          s_* & 0^k & s_+\\
          & & \smp_1
      \end{tpattern}
      \quad &
      M_\smq(k):&\ \begin{tpattern}
          s_+ & 0^k & s_* \\
          & & \smq_2
      \end{tpattern}
      \TMtrans^{\scriptscriptstyle{f(k)}}
      \begin{tpattern}
          s_+ & 0^k & s_*\\
          \smq_1 & &
      \end{tpattern}
  \end{align*}
  \normalsize  
  Additionally, the cell containing $s_*$ is only visited at the last (resp. first) step of the sequences of transitions $M_\smb$ and $M_\smd$ (resp. $M_\smp$ and $M_\smq$). And the cell containing $s_+$ is never modified.
\end{proposition}

\begin{proof}
  This proof adapts the proof of~\cite[Lemma 1]{2017-COT}, and highlights the recursive/nested aspects of these moves. In the case $k = 0$ one can check that indeed the formula describes a single transition. We reason by induction, and assume $M_\smb(k)$, $M_\smd(k)$, $M_\smp(k)$ and $M_\smq(k)$ hold. We only prove $M_\smb(k+1)$ and $M_\smp(k+1)$, by symmetry between $\smb$ and $\smd$ (resp.\ $\smp$ and $\smq$). Since $f(k+1) = 3f(k) + 4$ we should find $3$ recursions, and $4$ extra steps. This is what happens:

  \bigskip
  \noindent\begin{minipage}{0.49\textwidth}
    \centering
    $M_\smb(k+1)$
    \smallskip
    \footnotesize
    \begin{align*}
      && \begin{tpattern}
        s_+ & 0^k & 0 & s_* \\
        \smb_2 & & &
      \end{tpattern} \\
      && \text{Apply } M_\smb(k) \\
      & \TMtrans^{f(k)} &
      \begin{tpattern}
        s_+ & 0^k & 0 & s_* \\
        & & \smb_1 &
      \end{tpattern} \\
      && \text{Apply one step} \\
      & \TMtrans &
      \begin{tpattern}
        s_+ & 0^k & 1 & s_* \\
        & & \smd_2 &
      \end{tpattern} \\
      && \text{Apply } M_\smd(k) \\
      & \TMtrans^{f(k)} &
      \begin{tpattern}
        s_+ & 0^k & 1 & s_* \\
        \smd_1 & & &
      \end{tpattern} \\
      && \text{Apply one step} \\
      &\TMtrans &
      \begin{tpattern}
        s_+ & 0^k & 1 & s_* \\
        \smp_2 & & &
      \end{tpattern} \\
      && \text{Apply } M_\smp(k) \\
      & \TMtrans^{f(k)} &
      \begin{tpattern}
        s_+ & 0^k & 1 & s_* \\
        & & \smp_1 &
      \end{tpattern} \\
      && \text{Apply one step} \\
      &\TMtrans &
      \begin{tpattern}
        s_+ & 0^k & 0 & s_* \\
        & & \smb_2 &
      \end{tpattern} \\
      && \text{Apply one step} \\
      &\TMtrans &
      \begin{tpattern}
        s_+ & 0^k & 0 & s_* \\
        & & & \smb_1
      \end{tpattern}
    \end{align*}
    \normalsize
  \end{minipage}\hfill
  \begin{minipage}{0.49\textwidth}
      \centering
      $M_\smp(k+1)$
      \smallskip
      \footnotesize
      \begin{align*}
        && \begin{tpattern}
          s_* & 0 & 0^k & s_+ \\
          \smp_2 & & &
        \end{tpattern} \\
        && \text{Apply one step} \\
        & \TMtrans &
        \begin{tpattern}
          s_* & 0 & 0^k & s_+ \\
          & \smp_1 & &
        \end{tpattern} \\
        && \text{Apply one step} \\
        & \TMtrans &
        \begin{tpattern}
          s_* & 2 & 0^k & s_+ \\
          & \smb_2 & &
        \end{tpattern} \\
        && \text{Apply } M_\smb(k) \\
        & \TMtrans^{f(k)} &
        \begin{tpattern}
          s_* & 2 & 0^k & s_+ \\
          & & & \smb_1
        \end{tpattern} \\
        && \text{Apply one step} \\
        & \TMtrans &
        \begin{tpattern}
          s_* & 2 & 0^k & s_+ \\
          & & & \smq_2
        \end{tpattern} \\
        && \text{Apply } M_\smq(k) \\
        & \TMtrans^{f(k)} &
        \begin{tpattern}
          s_* & 2 & 0^k & s_+ \\
          & \smq_1 & &
        \end{tpattern} \\
        && \text{Apply one step} \\
        &\TMtrans &
        \begin{tpattern}
          s_* & 0 & 0^k & s_+ \\
          & \smp_2 & &
        \end{tpattern} \\
        && \text{Apply } M_\smp(k) \\
        & \TMtrans^{f(k)} &
        \begin{tpattern}
          s_* & 0 & 0^k & s_+ \\
          & & & \smp_1
        \end{tpattern}
      \end{align*}
      \normalsize
  \end{minipage}
  \qedhere
\end{proof}

\subsection{Action of SMART on cyclic tapes}

This section studies the action of SMART on cyclic tapes of length $\ell \geq 2$. We call \emph{initial configurations} the following four cyclic configurations:

\footnotesize
\begin{align*}
    C_\smb = \begin{turing}
        0 & 0^{\ell-1} \\
        \smb_1 &
    \end{turing}
    \qquad & 
    C_\smd = \begin{turing}
        0 & 0^{\ell-1} \\
        \smd_1 &
    \end{turing} \\
    C_\smp = \begin{turing}
        0 & 0^{\ell-1} \\
        \smp_1 &
    \end{turing}
    \qquad  & 
    C_\smq = \begin{turing}
        0 & 0^{\ell-1}\\
        \smq_1 &
    \end{turing}
\end{align*}
\normalsize

\begin{proposition}\label{prop:smart-eventually-shifts}
  Let $\ell \geq 2$. The action of the $(2\cdot3^\ell)\text{-th}$ power of SMART on $C_\smb$ and $C_\smp$ (resp. $C_\smd$ and $C_\smq$) is a right-shift (resp. left-shift). Furthermore, the intermediate configurations are all distinct even up to a shift.
\end{proposition}

\begin{proof}
  By symmetries between $\smb$ and $\smd$ (resp. $\smp$ and $\smq$), we prove the result for $C_\smb$ and $C_\smp$.

  \footnotesize
  \noindent\begin{minipage}{0.49\textwidth}
      \begin{align*}
          && \begin{turing}
            0 & 0 & 0^{\ell-2} \\
            \smb_1 &
          \end{turing} \\
          && \text{ Apply one step} \\
          &\TMtrans &
          \begin{turing}
            1 & 0 & 0^{\ell-2} \\
            \smd_2 &
          \end{turing} \\
          && \text{Apply } M_\smd(\ell-1) \\
          &\TMtrans^{\scriptscriptstyle{f(\ell-1)}} &
          \begin{turing}
            1 & 0 & 0^{\ell-2} \\
            \smd_1 &
          \end{turing} \\
          && \text{ Apply one step} \\
          &\TMtrans &
          \begin{turing}
            1 & 0 & 0^{\ell-2} \\
            \smp_2 &
          \end{turing} \\
          && \text{Apply } M_\smp(\ell-1) \\
          &\TMtrans^{\scriptscriptstyle{f(\ell-1)}} &
          \begin{turing}
            1 & 0 & 0^{\ell-2} \\
            \smp_1 &
          \end{turing} \\
          && \text{ Apply one step} \\
          &\TMtrans &
          \begin{turing}
            0 & 0 & 0^{\ell-2} \\
            \smb_2 &
          \end{turing} \\
          && \text{ Apply one step} \\
          &\TMtrans &
          \begin{turing}
            0 & 0 & 0^{\ell-2} \\
            & \smb_1
          \end{turing}
      \end{align*}
  \end{minipage}\hfill
  \begin{minipage}{0.49\textwidth}
      \begin{align*}
        && \begin{turing}
          0 & 0 & 0^{\ell-2} \\
          \smp_1 & &
        \end{turing} \\
        && \text{ Apply one step} \\
        &\TMtrans &
        \begin{turing}
          2 & 0 & 0^{\ell-2} \\
          \smb_2 & &
        \end{turing} \\
        && \text{Apply } M_\smb(\ell-1) \\
        &\TMtrans^{\scriptscriptstyle{f(\ell-1)}} &
        \begin{turing}
          2 & 0 & 0^{\ell-2} \\
          \smb_1 & &
        \end{turing} \\
        && \text{ Apply one step} \\
        &\TMtrans &
        \begin{turing}
          2 & 0 & 0^{\ell-2} \\
          \smq_2 & &
        \end{turing} \\
        && \text{Apply } M_\smq(\ell-1) \\
        &\TMtrans^{\scriptscriptstyle{f(\ell-1)}} &
        \begin{turing}
          2 & 0 & 0^{\ell-2} \\
          \smq_1 & &
        \end{turing} \\
        && \text{ Apply one step} \\
        &\TMtrans &
        \begin{turing}
          0 & 0 & 0^{\ell-2} \\
          \smp_2 & &
        \end{turing} \\
        && \text{ Apply one step} \\
        &\TMtrans &
        \begin{turing}
          0 & 0 & 0^{\ell-2} \\
          & \smp_1 &
        \end{turing}
      \end{align*}
  \end{minipage}\normalsize

  \medskip
  We used moves $M_\smb(\ell-1),M_\smd(\ell-1),M_\smp(\ell-1)$ and $M_\smq(\ell-1)$ in patterns that overlap themselves on their first and last letters in the cyclic tape. This is valid, because the cell containing $s_*$ is only visited at the last (resp. first) step of $M_\smb$ and $M_\smd$ (resp. $M_\smp$ and $M_\smq$).
  
  For the last claim, by shift-commutation and bijectivity of the action, it is enough to show that a shifted copy of the initial configuration does not appear before the last step. This is clear from looking at the first columns, which have positive values on all but the first step and the two last steps.
\end{proof}

\begin{lemma}\label{lem:smart-action-as-cycles}
  For $\ell \geq 1$, the action of SMART on cyclic tapes of length $\ell$ is composed of four disjoint cycles of length $2\ell \cdot 3^\ell$, which are the orbits of the four initial configurations.
  Additionally, the action of the $(2\cdot3^\ell)\text{-th}$ power of SMART on a cyclic tape is a right-shift (resp.\ left-shift) on the orbits of $C_\smb$ and $C_\smp$ (resp. $C_\smd$ and $C_\smq$).
\end{lemma}

\begin{proof}
  This is an immediate consequence of Proposition~\ref{prop:smart-eventually-shifts}: the orbits are each of length $2\ell \cdot 3^\ell$ (number of shifts $\times$ number of steps for each shift), and are disjoint (by looking at the first column in the previous proof). As there are $8\ell \cdot 3^\ell$ different cyclic configurations containing a head (eight different states with $\ell$ possible positions, and a ternary tape of length $\ell$), any configuration belongs to one of these four orbits: this concludes the proof.
\end{proof}

\subsection{Encoding SMART configurations}\label{sec:smart-encoding-overall}

Recall that for the SMART machine, $\Gamma = \{0,1,2\}$ and $Q \simeq \{ \smb,\smd,\smp,\smq\} \times \{1,2\} $. Lemma~\ref{lem:smart-action-as-cycles} implies that the set of cyclic SMART configurations of length $\ell$ that contain a head
\[ \{ x \in (\Gamma \cup (Q \times \Gamma))^{\Z/\ell\Z} \mid \exists! i \in \Z/\ell\Z, x_i \in Q \times \Gamma \} \]
is conjugate (as a finite dynamical system, or as a permutation) to a disjoint union of four (depending on whether the head is in state $\smb, \smd, \smp$ or $\smq$) disjoint systems of counters ranging in $\llbracket 0, \ell-1 \rrbracket \times \{1,2\} \times \{0,1,2\}^\ell$ (respectively for the position of the head, the second component $\{1,2\}$ of $Q$, and the tape of alphabet $\Gamma)$. Each of these counters encodes $2\ell \cdot 3^\ell$ different values, which is exactly the length of any SMART cycle by Lemma~\ref{lem:smart-action-as-cycles}.

\smallskip
We pick a natural conjugacy $\encoding_\ell : \ctape_\ell \to \ctape_\ell$, a shift-invariant bijection that encodes a SMART configuration into its orbit position in base $2\ell \cdot 3^\ell$. We refer to the conjugacy $\encoding_\ell$ as the \emph{encoding map}.

More precisely, if $w \in \ctape_\ell$ contains a head, then by Lemma~\ref{lem:smart-action-as-cycles} there exists some $0 \leq n < 2\ell \cdot 3^\ell$ such that $w = (T_{\ell,\smart})^n(C_q)$ for some initial configuration $C_q$ ($q \in \{\smb,\smd,\smp,\smq\}$), and the map $\encoding_\ell$ encodes the tuple $(q,n)$ in $\ctape_\ell$ as
\[ \encoding_\ell(w) = \sigma^{-\varepsilon \cdot a} \left( \begin{turing}
  c_1 & c_2 & \dots & c_\ell \\
  q_b & & &
\end{turing} \right) \]
where:
\begin{itemize}
  \item $q$ is stored in the first component of the state $\{\smb,\smd,\smp,\smq\}$;
  \item $b \cdot c \in \{1,2\} \cdot \{0,1,2\}^\ell$ encodes $n \bmod 2\cdot3^\ell$, i.e.\ $c$ is a ternary word satisfying $v_{(3)}(c) = n \bmod 3^\ell$, and $b$ stores $(\lfloor n/3^\ell \rfloor \bmod 2)+1$ in the second component of the state;
  \item $a$ is the quotient of $n$ by $2 \cdot 3^\ell$, and is encoded in how much is the cyclic configuration is shifted.
  \item $\varepsilon = +1$ if $q \in \{\smb,\smp\}$ (resp. $\varepsilon = -1$ if $q \in \{\smd,\smq\}$) shifts to the right (resp. left) if $q \in \{\smb,\smp\}$ (resp. $q \in \{\smd,\smq\}$).
\end{itemize}
and if $w \in \ctape_\ell$ contains no head (i.e. $w \in \Gamma^{\Z/\ell\Z})$, then we set $\encoding_\ell(w) = w$.

\smallskip
In other words, given a configuration $w = (T_{\ell,\smart})^n(C_q)$ for some initial configuration $C_q$ ($q \in \{\smb,\smd,\smp,\smq\}$) and $0 \leq n < 2\ell \cdot 3^\ell$, the map $\encoding_\ell$ encodes the tuple $(q,n)$ as plainly (and humanly readable) as possible. 

\bigskip
In the next two sections (Sections~\ref{sec:smart-analysis} and~\ref{sec:smart-encoding}), we detail how this bijection can be computed inductively, i.e.\ we define piecewise-defined bijective maps $\pencoding_{\mathrm{init}}$, $\pencoding_{k \to k+1}$ (for $0 \leq k \leq \ell-2$) and $\pencoding_{\ell,\mathrm{final}}$ acting on $\ctape_\ell$ such that
\[ \encoding_{\ell} = \pencoding_{\ell,\mathrm{final}} \circ \Big( \prod_{k=0}^{\ell-2} \pencoding_{k \to k+1} \Big) \circ \pencoding_{\mathrm{init}}. \] 

The precise definition of these maps is not strictly necessary to understand the rest of this paper (it is only used in lemmas~\ref{lem:encoding-case-in-group} and \ref{lem:encoding-steps-in-group}, where we need the piecewise defined functions to satisfy the requirements described in Section~\ref{sec:engineering-ducking-trick}). On a first reading, we recommend the reader simply remembers the idea of encoding configurations into a counter, and goes directly go to Section~\ref{sec:Finitary} about the finitary distortion of the SMART machine.

\subsection{Analysis of SMART configurations}\label{sec:smart-analysis}

We now explain how, given a cyclic SMART configuration $w$ of length $\ell$, we determine which orbit it belongs in and its position in this orbit, i.e.\ the number of steps required to obtain it from its corresponding initial configuration $C_\smb,C_\smd,C_\smp$ or $C_\smq$.

\medskip
We say that a cyclic configuration $w \in \ctape_\ell$ is performing the $j\text{-th}$ step of computation of $M_\smb(k)$ (resp.\ $M_\smd(k), M_\smp(k), M_\smq(k)$), for $0 \leq k \leq \ell-1$ and $0 \leq j \leq f(k)$, if it contains the $j\text{-th}$ pattern of the sequence of transitions $M_\smb(k)$ (resp. $M_\smd(k), M_\smp(k), M_\smq(k)$) of Proposition~\ref{prop:smart-moves}. At this point, it may not be clear that this is unique, but this will follow from our argument.

If a configuration is performing some step of computation from one of the moves $M_\smb(k)$, $M_\smd(k)$, $M_\smp(k)$ or $M_\smq(k)$, we refer to this move as its \emph{computation of level $k$}.

\paragraph*{Initialization}
We call the following patterns \emph{special patterns of level $k$ (for $1 \leq k \leq \ell-1$)}:
\footnotesize
\begin{align*}
  s(\smb_2,k) = & \begin{tpattern}
    s_+ & \widt{\didwid}{0^{k-1}} & 0 \\
    & & \smb_2
  \end{tpattern}
  & s(\smb_1,k) = & \begin{tpattern}
    s_+ & \widt{\didwid}{0^{k}} & s_* \\
    & & \smb_1
  \end{tpattern} \\
  s(\smd_2,k) = & \begin{tpattern}
    0 & \widt{\didwid}{0^{k-1}} & s_+ \\
    \smd_2 & &
  \end{tpattern}
  & s(\smd_1,k) = & \begin{tpattern}
    s_* & \widt{\didwid}{0^{k}} & s_+ \\
    \smd_1 & &
  \end{tpattern} \\
  s(\smp_2,k) = & \begin{tpattern}
    s_* & \widt{\didwid}{0^{k}} & s_+ \\
    \smp_2 & &
  \end{tpattern}
  & s(\smp_1,k) = & \begin{tpattern}
    0 & \widt{\didwid}{0^{k-1}} & s_+ \\
    \smp_1 & &
  \end{tpattern} \\
  s(\smq_2,k) = & \begin{tpattern}
    s_+ & \widt{\didwid}{0^{k}} & s_* \\
    & & \smq_2
  \end{tpattern}
  & s(\smq_1,k) = & \begin{tpattern}
    s_+ & \widt{\didwid}{0^{k-1}} & 0 \\
    & & \smq_1
  \end{tpattern}
\end{align*}
\normalsize
and the following are \emph{special patterns of level $\ell$}:
\footnotesize
\begin{align*}
  \begin{tpattern}
    0 & 0^{\ell-1} \\
    \smb_2
  \end{tpattern} \qquad
  \begin{tpattern}
    0 & 0^{\ell-1} \\
    \smd_2
  \end{tpattern} \\
  \begin{tpattern}
    0 & 0^{\ell-1} \\
    \smp_2
  \end{tpattern} \qquad
  \begin{tpattern}
    0 & 0^{\ell-1} \\
    \smq_2
  \end{tpattern}
\end{align*}
\normalsize
The latter appear exactly in the shifts of the configurations $\smart^{-1}(C_q)$, for $C_q$ the initial configurations ($q \in \{\smb,\smd,\smp,\smq\}$).

\medskip
By the proof of Proposition~\ref{prop:smart-moves}, we see that if a cyclic configuration contains a special pattern $s(\smb_2,k)$, $s(\smb_1,k)$, $s(\smd_2,k)$ or $s(\smd_1,k)$ (resp. $s(\smp_2,k)$, $s(\smp_1,k)$, $s(\smq_2,k)$ or $s(\smq_1,k)$), then it performs the last two steps of $M_\smb(k)$ or $M_\smd(k)$ respectively (resp. the first two steps of $M_\smp(k)$ or $M_\smq(k)$).

\begin{claim}
Given a cyclic configuration $w$ of length $\ell$ containing a head, exactly one of the following holds:
\begin{itemize}
  \item $w$ is the shift of an initial configuration.
  \item $w$ performs some step of computation of level $0$ from either $M_\smb(0)$, $M_\smd(0)$, $M_\smp(0)$ or $M_\smq(0)$.
  \item $w$ contains a special pattern of level $1 \leq k \leq \ell$.
\end{itemize}
\end{claim}
\begin{proof}
The patterns of $M_\smb(0)$, $M_\smd(0)$, $M_\smp(0)$ and $M_\smq(0)$ (eight in total), along with the special patterns of every level, disjointly cover all the non-initial configurations with a head (the level is determined by the distance to the nearest nonzero symbol in an appropriate direction).
\end{proof}

\begin{figure}
  \begin{mdframed}[backgroundcolor=black!5!white]
    $M_\smb(k)$ is performed in:
    \footnotesize
    \[
      \begin{array}{lrcl}
        M_\smb(k+1) \text{ (at step $0$):}
        & \begin{tpatternc}
          s_+ & 0^k & 0 & \textcolor{gray}{s_*} \\
          \smb_2 & & &
          \CodeAfter \tikz \draw (1-3) circle (1.5mm);
        \end{tpatternc}
        & \TMtrans^* &
        \begin{tpatternc}
          s_+ & 0^k & 0 & \textcolor{gray}{s_*} \\
          & & \smb_1 &
          \CodeAfter \tikz \draw (1-3) circle (1.5mm);
        \end{tpatternc} \\
        M_\smd(k+1) \text{ (at step $f(k)+1$):}
        & \begin{tpatternc}
          \textcolor{gray}{s_*} & 1 & 0^k & s_+ \\
          & \smb_2 & &
          \CodeAfter \tikz \draw (1-4) circle (1.8mm); \tikz \draw (1-2) circle (1.5mm);
        \end{tpatternc}
        & \TMtrans^* &
        \begin{tpatternc}
          \textcolor{gray}{s_*} & 1 & 0^k & s_+ \\
          & & & \smb_1
          \CodeAfter \tikz \draw (1-4) circle (1.8mm); \tikz \draw (1-2) circle (1.5mm);
        \end{tpatternc} \\
        M_\smp(k+1) \text{ (at step $2$):}
        & \begin{tpatternc}
          \textcolor{gray}{s_*} & 2 & 0^k & s_+ \\
          & \smb_2 & &
          \CodeAfter \tikz \draw (1-4) circle (1.8mm); \tikz \draw (1-2) circle (1.5mm);
        \end{tpatternc}
        & \TMtrans^* &
        \begin{tpatternc}
          \textcolor{gray}{s_*} & 2 & 0^k & s_+ \\
          & & & \smb_1
          \CodeAfter \tikz \draw (1-4) circle (1.8mm); \tikz \draw (1-2) circle (1.5mm);
        \end{tpatternc}
      \end{array}
    \]
    \normalsize

    $M_\smd(k)$ is performed in:
    \footnotesize
    \[
      \begin{array}{lrcl}
        M_\smd(k+1) \text{ (at step $0$):}
        & \begin{tpatternc}
          \textcolor{gray}{s_*} & 0 & 0^k & s_+ \\
          & & & \smd_2
          \CodeAfter \tikz \draw (1-2) circle (1.5mm);
        \end{tpatternc}
        & \TMtrans^* &
        \begin{tpatternc}
          \textcolor{gray}{s_*} & 0 & 0^k & s_+ \\
          & \smd_1 & &
          \CodeAfter \tikz \draw (1-2) circle (1.5mm);
        \end{tpatternc} \\
        M_\smb(k+1) \text{ (at step $f(k)+1$):}
        & \begin{tpatternc}
          s_+ & 0^k & 1 & \textcolor{gray}{s_*} \\
          & & \smd_2 &
          \CodeAfter \tikz \draw (1-1) circle (1.8mm); \tikz \draw (1-3) circle (1.5mm);
        \end{tpatternc}
        & \TMtrans^* &
        \begin{tpatternc}
          s_+ & 0^k & 1 & \textcolor{gray}{s_*} \\
          \smd_1 & & &
          \CodeAfter \tikz \draw (1-1) circle (1.8mm); \tikz \draw (1-3) circle (1.5mm);
        \end{tpatternc} \\
        M_\smq(k+1) \text{ (at step $2$):}
        & \begin{tpatternc}
          s_+ & 0^k & 2 & \textcolor{gray}{s_*} \\
          & & \smd_2 &
          \CodeAfter \tikz \draw (1-1) circle (1.8mm); \tikz \draw (1-3) circle (1.5mm);
        \end{tpatternc}
        & \TMtrans^* &
        \begin{tpatternc}
          s_+ & 0^k & 2 & \textcolor{gray}{s_*} \\
          \smd_2 & & &
          \CodeAfter \tikz \draw (1-1) circle (1.8mm); \tikz \draw (1-3) circle (1.5mm);
        \end{tpatternc}
      \end{array}
    \]
    \normalsize

    $M_\smp(k)$ is performed in:
    \footnotesize
    \[
      \begin{array}{lrcl}
        M_\smp(k+1) \text{ (at step $2f(k)+4$):}
        & \begin{tpatternc}
          \textcolor{gray}{s_*} & 0 & 0^k & s_+ \\
            & \smp_2 & &
            \CodeAfter \tikz \draw (1-2) circle (1.5mm);
        \end{tpatternc}
        & \TMtrans^* &
        \begin{tpatternc}
          \textcolor{gray}{s_*} & 0 & 0^k & s_+ \\
          & & & \smp_1
          \CodeAfter \tikz \draw (1-2) circle (1.5mm);
        \end{tpatternc} \\
        M_\smb(k+1) \text{ (at step $2f(k)+2$):}
        & \begin{tpatternc}
          s_+ & 0^k & 1 & \textcolor{gray}{s_*} \\
          \smp_2 & & &
          \CodeAfter \tikz \draw (1-1) circle (1.8mm); \tikz \draw (1-3) circle (1.5mm);
        \end{tpatternc}
        & \TMtrans^* &
        \begin{tpatternc}
          s_+ & 0^k & 1 & \textcolor{gray}{s_*} \\
          & & \smp_1 &
          \CodeAfter \tikz \draw (1-1) circle (1.8mm); \tikz \draw (1-3) circle (1.5mm);
        \end{tpatternc} \\
        M_\smq(k+1) \text{ (at step $f(k)+3$):}
        & \begin{tpatternc}
          s_+ & 0^k & 2 & \textcolor{gray}{s_*}\\
          \smp_2 & & &
          \CodeAfter \tikz \draw (1-1) circle (1.8mm); \tikz \draw (1-3) circle (1.5mm);
        \end{tpatternc}
        & \TMtrans^* &
        \begin{tpatternc}
          s_+ & 0^k & 2 & \textcolor{gray}{s_*}\\
          & & \smp_1 &
          \CodeAfter \tikz \draw (1-1) circle (1.8mm); \tikz \draw (1-3) circle (1.5mm);
        \end{tpatternc}
      \end{array}
    \]
    \normalsize

    $M_\smq(k)$ is performed in:
    \footnotesize
    \[
      \begin{array}{lrcl}
        M_\smq(k+1) \text{ (at step $2f(k)+4$):}
        & \begin{tpatternc}
          s_+ & 0^k & 0 & \textcolor{gray}{s_*} \\
          & & \smq_2 &
          \CodeAfter \tikz \draw (1-3) circle (1.5mm);
        \end{tpatternc}
        & \TMtrans^* &
        \begin{tpatternc}
          s_+ & 0^k & 0 & \textcolor{gray}{s_*} \\
          \smq_1 & & &
          \CodeAfter \tikz \draw (1-3) circle (1.5mm);
        \end{tpatternc} \\
        M_\smd(k+1) \text{ (at step $2f(k)+2$):}
        & \begin{tpatternc}
          \textcolor{gray}{s_*} & 1 & 0^k & s_+ \\
          & & & \smq_2
          \CodeAfter \tikz \draw (1-4) circle (1.8mm); \tikz \draw (1-2) circle (1.5mm);
        \end{tpatternc}
        & \TMtrans^* &
        \begin{tpatternc}
          \textcolor{gray}{s_*} & 1 & 0^k & s_+ \\
          & \smq_1 & &
          \CodeAfter \tikz \draw (1-4) circle (1.8mm); \tikz \draw (1-2) circle (1.5mm);
        \end{tpatternc} \\
        M_\smp(k+1) \text{ (at step $f(k)+3$):}
        & \begin{tpatternc}
          \textcolor{gray}{s_*} & 2 & 0^k & s_+ \\
          & & & \smq_2
          \CodeAfter \tikz \draw (1-4) circle (1.8mm); \tikz \draw (1-2) circle (1.5mm);
        \end{tpatternc}
        & \TMtrans^* &
        \begin{tpatternc}
          \textcolor{gray}{s_*} & 2 & 0^k & s_+ \\
          & \smq_1 & &
          \CodeAfter \tikz \draw (1-4) circle (1.8mm); \tikz \draw (1-2) circle (1.5mm);
        \end{tpatternc}
      \end{array}
    \]
    \normalsize
    \end{mdframed}

    \caption{Bottom-up analysis of SMART configurations: $k \to k+1$}
    \captionsetup{width=.82\linewidth,font=small}
    \caption*{Intermediate steps of moves of level $k+1$ (patterns of length $k+3$) corresponding to sub-moves of level $k$. We darken the part of the pattern of length $k+2$ performing the sub-move of level $k$. Circled ternary letters are not modified by the sub-move of level $k$, hence can be used to perform a case-analysis.}
    \label{fig:bottom-up-analysis-smart}
\end{figure}

\paragraph*{Inductive analysis $k \to k+1$ (for $k+1 \leq \ell-1$):}
Let $k \leq \ell-2$ be an integer, and $w$ be a non-initial cyclic configuration. If $w$ performs some computation of level $k$ (for $0 \leq k \leq \ell-2$), then it performs some computation of level $k+1$: indeed, Figure~\ref{fig:bottom-up-analysis-smart} shows that any computation of level $k$ belongs to some computation of level $k+1$, and that the latter is uniquely determined by considering the value of two cells (circled on the figure) which are left unmodified by the computation of level $k$.

Additionally, if we know that $w$ performs the $j\text{-th}$ step of some computation of level $k$ (for $0 \leq j \leq f(k)$), then the same case-analysis determines $j'$ ($0 \leq j' \leq f(k+1)$) such that $w$ performs the $j'\text{-th}$ step of its computation of level~$k+1$. 

\bigskip
\begin{example}
\label{ex:FindingStep1}
For example, consider the cyclic configuration of length $\ell = 6$:
  \footnotesize
  \[ 
    w = \begin{turingc}
    1 & 2 & 1 & 1 & 2 & 0 \\
    & & & \smb_2 & &
  \end{turingc}
  \]\normalsize

  First, $w$ performs step $0$ of $M_\smb(0)$ (by simply considering all the computations of level $0$). Indeed, the highlighted pattern inside the configuration below is the $0\text{-th}$ pattern of the move $M_\smb(0)$:

  \footnotesize
  \[ 
    w = \begin{turingc}
      \textcolor{gray}{1} & \textcolor{gray}{2} & \textcolor{gray}{1} & 1 & 2 & \textcolor{gray}{0} \\
      & & & \smb_2 & &
      \CodeAfter \tikz \draw [black,decorate,decoration = {brace,raise=0pt,amplitude=3pt,aspect=0.5}] ($(1-|4) + (0.05,0)$) -- ($(1-|6) + (-0.05,0)$);
  \end{turingc}
  \]\normalsize

  \smallskip
  Then, by looking at Figure~\ref{fig:bottom-up-analysis-smart} (four cases, with three subcases each), we deduce succesively that:
  \begin{enumerate}
    \item By considering Figure~\ref{fig:bottom-up-analysis-smart} (case $M_\smb(0)$, second subcase)
    \footnotesize
    \[ 
      w = \begin{turingc}
        \textcolor{gray}{1} & \textcolor{gray}{2} & \textcolor{gray}{1} & 1 & 2 & \textcolor{gray}{0} \\
        & & & \smb_2 & &
        \CodeAfter \tikz \draw (1-4) circle (1.5mm); \tikz \draw (1-5) circle (1.5mm);
    \end{turingc}
    \]\normalsize
    we deduce that $w$ performs step $2 = 0 + (f(0)+1)$ of $M_\smd(1)$. Additionally, the computation extends to the left, i.e.\ the move $M_\smd(1)$ appears in the following highlighted pattern:

    \footnotesize
    \[ 
      w = \begin{turingc}
        \textcolor{gray}{1} & \textcolor{gray}{2} & 1 & 1 & 2 & \textcolor{gray}{0} \\
        & & & \smb_2 & &
        \CodeAfter \tikz \draw [black,decorate,decoration = {brace,raise=0pt,amplitude=3pt,aspect=0.5}] ($(1-|3) + (0.05,0)$) -- ($(1-|6) + (-0.05,0)$);
    \end{turingc}
    \]\normalsize
    
    \item By considering Figure~\ref{fig:bottom-up-analysis-smart} (case $M_\smd(1)$, third subcase)
    \footnotesize
    \[ 
      w = \begin{turingc}
        \textcolor{gray}{1} & \textcolor{gray}{2} & 1 & 1 & 2 & \textcolor{gray}{0} \\
        & & & \smb_2 & &
        \CodeAfter \tikz \draw (1-3) circle (1.5mm); \tikz \draw (1-5) circle (1.5mm);
    \end{turingc}
    \]\normalsize
    we deduce that $w$ performs step $4 = 2 + 2$ of $M_\smq(2)$. Additionally, the computation extends to the right and the move $M_\smq(2)$ appears in the following highlighted pattern:

    \footnotesize
    \[ 
      w = \begin{turingc}
        \textcolor{gray}{1} & \textcolor{gray}{2} & 1 & 1 & 2 & 0 \\
        & & & \smb_2 & &
        \CodeAfter \tikz \draw [black,decorate,decoration = {brace,raise=0pt,amplitude=3pt,aspect=0.5}] ($(1-|3) + (0.05,0)$) -- (1-|7);
    \end{turingc}
    \]\normalsize

    \item By considering Figure~\ref{fig:bottom-up-analysis-smart} (case $M_\smq(2)$, first subcase)
    \footnotesize
    \[ 
      w = \begin{turingc}
        \textcolor{gray}{1} & \textcolor{gray}{2} & 1 & 1 & 2 & 0 \\
        & & & \smb_2 & &
        \CodeAfter \tikz \draw (1-3) circle (1.5mm); \tikz \draw (1-6) circle (1.5mm);
    \end{turingc}
    \]\normalsize
    we deduce that $w$ performs step $58 = 4 + (2f(2)+4)$ of $M_\smq(3)$. Additionally, the computation extends to the right and the move $M_\smq(3)$ appears in the following highlighted pattern (remember that configurations are cyclic)
    \footnotesize
    \[ 
      w = \begin{turingc}
        1 & \textcolor{gray}{2} & 1 & 1 & 2 & 0 \\
        & & & \smb_2 & &
        \CodeAfter 
        \tikz{\clip ($(1-|3) + (0,-0.3)$) rectangle ($(1-|7) + (0,0.3)$);
        \draw [black,decorate,decoration = {brace,raise=0pt,amplitude=3pt,aspect=0.5}] ($(1-|3) + (0.05,0)$) -- ($(1-|7) + (0.4,0)$);}
        \tikz{\clip ($(1-|1) + (0,-0.3)$) rectangle ($(1-|2) + (0,0.3)$);
        \draw [black,decorate,decoration = {brace,raise=0pt,amplitude=3pt,aspect=0.5}] ($(1-|1) + (-0.8,0)$) -- ($(1-|2) + (-0.05,0)$);}
    \end{turingc}
    \]\normalsize
    
    \item By considering Figure~\ref{fig:bottom-up-analysis-smart} (case $M_\smq(3)$, second subcase)
    \footnotesize
    \[ 
      w = \begin{turingc}
        1 & \textcolor{gray}{2} & 1 & 1 & 2 & 0 \\
        & & & \smb_2 & &
        \CodeAfter \tikz \draw (1-3) circle (1.5mm); \tikz \draw (1-1) circle (1.5mm);
    \end{turingc}
    \]\normalsize
    we deduce that $w$ performs step $218 = 58 + (2f(3)+2)$ of $M_\smd(4)$. Additionally, the computation extends to the left and the move $M_\smd(4)$ appears in the following highlighted pattern:

    \footnotesize
    \[ 
      w = \begin{turingc}
        1 & 2 & 1 & 1 & 2 & 0 \\
        & & & \smb_2 & &
        \CodeAfter 
        \tikz{\clip ($(1-|2) + (0,-0.3)$) rectangle ($(1-|7) + (0,0.3)$);
        \draw [black,decorate,decoration = {brace,raise=0pt,amplitude=3pt,aspect=0.5}] ($(1-|2) + (0.05,0)$) -- ($(1-|7) + (0.4,0)$);}
        \tikz{\clip ($(1-|1) + (0,-0.3)$) rectangle ($(1-|2) + (0,0.3)$);
        \draw [black,decorate,decoration = {brace,raise=0pt,amplitude=3pt,aspect=0.5}] ($(1-|1) + (-0.8,0)$) -- ($(1-|2) + (-0.05,0)$);}
    \end{turingc}
    \]\normalsize

    \item By considering Figure~\ref{fig:bottom-up-analysis-smart} (case $M_\smd(4)$, second subcase)
    \footnotesize
    \[ 
      w = \begin{turingc}
        1 & 2 & 1 & 1 & 2 & 0 \\
        & & & \smb_2 & &
        \CodeAfter \tikz \draw (1-2) circle (1.5mm); \tikz \draw (1-1) circle (1.5mm);
    \end{turingc}
    \]\normalsize
    we deduce that $w$ performs step $460 = 218 + (f(4)+1)$ of $M_\smb(5)$.

    On the final step, we should extend to the right; since we run out of cells on the tape, this means that we should interpret the circled $2$-cell as both the first and last cell of the $M_\smb(5)$-computation, i.e. the move $M_\smb(5)$ appears in the following highlighted self-overlapping pattern:

    \footnotesize
    \[ 
      w = \begin{turingc}
        1 & 2 & 1 & 1 & 2 & 0 \\
        & & & \smb_2 & &
        \CodeAfter 
        \tikz{\clip ($(1-|2) + (0,-0.3)$) rectangle ($(1-|7) + (0,0.3)$);
        \draw [black,decorate,decoration = {brace,raise=0pt,amplitude=3pt,aspect=0.5}] ($(1-|2) + (0.05,0)$) -- ($(1-|7) + (0.8,0)$);}
        \tikz{\clip ($(1-|1) + (0,-0.3)$) rectangle ($(1-|3) + (0,0.3)$);
        \draw [black,decorate,decoration = {brace,raise=0pt,amplitude=3pt,aspect=0.5}] ($(1-|1) + (-1.6,-0.01)$) -- ($(1-|3) + (-0.05,-0.01)$);}
    \end{turingc}
    \]\normalsize
  \end{enumerate}
\end{example}

So we claim that $w$ performs step $460 = 218 + (f(4)+1)$ of $M_\smb(5)$. And one can verify by a direct calculation (for example by computer) that

\[
      \begin{tpatternc}
          2 & 0 & 0 & 0 & 0 & 0 & {\color{gray}{s_*}} \\
          \smb_2 & & & & & &
      \end{tpatternc}
      \TMtrans^{\scriptscriptstyle{460}}
      \begin{tpatternc}
          2 & 1 & 1 & 2 & 0 & 1 & {\color{gray}{s_*}} \\
          & & \smb_2 & & & &
      \end{tpatternc}
\]
as we just deduced. \qee

\paragraph*{Conclusion:}
Let $w \in \ctape_\ell$ be a cyclic tape of length $\ell$ containing a head. By the claim above, there are three different cases.

Either $w$ is the shift of an initial configuration, in which case the state of $w$ determines which orbit $w$ belongs to, and counting the shift (and multiplying it by $2 \cdot 3^\ell$) is enough to know the position of the configuration $w$ in its orbit. A similar reasoning applies for configurations $w$ that contain a special pattern of level $\ell$.

\smallskip
On the other hand, assume $w$ contains some computation of level $0$. Then $w$ contains computations of every level $k$ for $0 \leq k \leq \ell-1$ by the previous induction. Similarly, if $w$ contains a special pattern of level $k$ for some $0 \leq k \leq \ell-1$, then $w$ corresponds to either the last two steps of $M_\smb(k)$ or $M_\smd(k)$, or the first two steps of $M_\smp(k)$ or $M_\smq(k)$. Then $w$ corresponds to computations of every level $k'$ for $k \leq k' \leq \ell-1$ by the previous induction.

Finally, the structure of each SMART cycle (detailed in the proof of Proposition~\ref{prop:smart-eventually-shifts}) enables to conclude about which orbit the configuration $w$ belong to, and its position in said orbit.

\bigskip
\begin{example}
\label{ex:FindingStep2}
  Consider once again the cyclic configuration of length $\ell = 6$

  \footnotesize
  \[ 
    w = \begin{turingc}
    1 & 2 & 1 & 1 & 2 & 0 \\
    & & & \smb_2 & &
  \end{turingc}
  \]\normalsize

  By the previous example, $w$ performs step $460$ of $M_\smb(5)$ on the following (self-overlapping) pattern:

  \footnotesize
  \[ 
    w = \begin{turingc}
      1 & 2 & 1 & 1 & 2 & 0 \\
      & & & \smb_2 & &
      \CodeAfter 
      \tikz{\clip ($(1-|2) + (0,-0.3)$) rectangle ($(1-|7) + (0,0.3)$);
      \draw [black,decorate,decoration = {brace,raise=0pt,amplitude=3pt,aspect=0.5}] ($(1-|2) + (0.05,0)$) -- ($(1-|7) + (0.8,0)$);}
      \tikz{\clip ($(1-|1) + (0,-0.3)$) rectangle ($(1-|3) + (0,0.3)$);
      \draw [black,decorate,decoration = {brace,raise=0pt,amplitude=3pt,aspect=0.5}] ($(1-|1) + (-1.6,-0.01)$) -- ($(1-|3) + (-0.05,-0.01)$);}
  \end{turingc}
  \]\normalsize

  Considering the proof of Proposition~\ref{prop:smart-eventually-shifts}, only the orbits of $C_\smp$ and $C_\smq$ have a cell containing tape-letter $2$ that is left unchanged during a computation of level $\ell-1$. Additionally, out of these two, only the orbit of $C_\smp$ contains the move $M_\smb(5)$.

  So $w$ belongs in the orbit of $C_\smp$. From the proof of Proposition~\ref{prop:smart-eventually-shifts}, we deduce that, modulo $2 \cdot 3^\ell$, the position of $w$ in said orbit is $1+460$. Finally, $2$ being the second cell of the tape, one shift already happened in the orbit: we conclude that the position of $w$ in the orbit of $C_\smp$ is $2\cdot 3^\ell + 461 = 1919$.

Indeed, one can verify by a direct calculation that
\[
      \begin{turingc}
          0 & 0 & 0 & 0 & 0 & 0 \\
          \smp_1 & & & & &
      \end{turingc}
      \TMtrans^{\scriptscriptstyle{1919}}
      \begin{turingc}
          1 & 2 & 1 & 1 & 2 & 0 \\
          & & & \smb_2 & &
      \end{turingc}
\]
as we just deduced. \qee
\end{example}

\subsection{Encoding cyclic configurations into their orbit positions in \texorpdfstring{$\ctape_\ell$}{\ctape\_{}l}}\label{sec:smart-encoding}

\smallskip
Denote $\smart$ the SMART machine introduced above. Recall the encoding map $\encoding_\ell : \ctape_\ell \to \ctape_\ell$ defined in Section~\ref{sec:smart-encoding-overall} as
\[ \encoding_\ell(w) = \sigma^{-\varepsilon \cdot a} \left( \begin{turing}
  c_1 & c_2 & \dots & c_\ell \\
  q_b & & &
\end{turing} \right) \]
if $w$ contains a head (in which case, there exists $n$ and $q \in \{\smb,\smd,\smp,\smq\}$ such that $w = T_{\ell,\smart}^n(C_q)$, and we set $b \cdot c \in \{1,2\} \cdot \{0,1,2\}^\ell$ to encode $n \bmod 2 \cdot 3^\ell$, and $a$ is the quotient of $n$ by $2 \cdot 3^\ell$); and if $w \in \ctape_\ell$ contains no head, then we set $\encoding_\ell(w) = w$.

\paragraph*{Inductive encoding}

\begin{figure}
  \begin{mdframed}[backgroundcolor=black!5!white]
    \footnotesize
    \[
      \begin{array}{lrcl}
        \text{Special $\smb_1$ higher level:}
        & \begin{tpattern}
          0 & s_* \\
          & \smb_1
        \end{tpattern}
        & \mapsto &
        \begin{tpattern}
          0 & s_* \\
          & \smb_1
        \end{tpattern} \\
        \text{Special $\smb_2$ higher level:}
        & \begin{tpattern}
          0 & s_* \\
          \smb_2 &
        \end{tpattern}
        & \mapsto &
        \begin{tpattern}
          0 & s_* \\
          & \smb_2
        \end{tpattern} \\
        M_\smb(0):
        & \begin{tpattern}
          1 & s_* \\
          \smb_2 &
        \end{tpattern}
        & \mapsto &
        \begin{tpattern}
          1 & s_* \\
          & \smb_1
        \end{tpattern} \\
        & \begin{tpattern}
          2 & s_* \\
          \smb_2 & 
        \end{tpattern}
        & \mapsto & 
        \begin{tpattern}
          1 & s_* \\
          & \smb_2
        \end{tpattern} \\
        & \begin{tpattern}
          1 & s_* \\
          & \smb_1
        \end{tpattern}
        & \mapsto &
        \begin{tpattern}
          2 & s_* \\
          & \smb_1
        \end{tpattern} \\
        & \begin{tpattern}
          2 & s_* \\
          & \smb_1
        \end{tpattern}
        & \mapsto & 
        \begin{tpattern}
          2 & s_* \\
          & \smb_2
        \end{tpattern}
      \end{array}
    \]
    \normalsize

    \medskip
    \footnotesize
    \[
      \begin{array}{lrcl}
        \text{Special $\smd_1$ higher level:}
        & \begin{tpattern}
          s_* & 0 \\
          \smd_1 &
        \end{tpattern}
        & \mapsto &
        \begin{tpattern}
          s_* & 0 \\
          \smd_1 & 
        \end{tpattern} \\
        \text{Special $\smd_2$ higher level:}
        & \begin{tpattern}
          s_* & 0 \\
          & \smd_2
        \end{tpattern}
        & \mapsto &
        \begin{tpattern}
          s_* & 0 \\
          \smd_2 &
        \end{tpattern} \\
        M_\smd(0):
        & \begin{tpattern}
          s_* & 1 \\
          & \smd_2
        \end{tpattern}
        & \mapsto &
        \begin{tpattern}
          s_* & 1 \\
          \smd_1 &
        \end{tpattern} \\
        & \begin{tpattern}
          s_* & 2 \\
          & \smd_2
        \end{tpattern}
        & \mapsto &
        \begin{tpattern}
          s_* & 1 \\
          \smd_2 &
        \end{tpattern} \\
        & \begin{tpattern}
          s_* & 1 \\
          \smd_1 &
        \end{tpattern}
        & \mapsto &
        \begin{tpattern}
          s_* & 2 \\
          \smd_1 &
        \end{tpattern} \\
        & \begin{tpattern}
          s_* & 2 \\
          \smd_1 &
        \end{tpattern}
        & \mapsto &
        \begin{tpattern}
          s_* & 2 \\
          \smd_2 &
        \end{tpattern}
      \end{array}
    \]
    \normalsize

    \medskip
    \footnotesize
    \[
      \begin{array}{lrcl}
        \text{Special $\smp_1$ higher level:}
        & \begin{tpattern}
          s_* & 0 \\
          & \smp_1
        \end{tpattern}
        & \mapsto &
        \begin{tpattern}
          s_* & 0 \\
          \smp_1 & 
        \end{tpattern} \\
        \text{Special $\smp_2$ higher level:}
        & \begin{tpattern}
          s_* & 0 \\
          \smp_2 &
        \end{tpattern}
        & \mapsto &
        \begin{tpattern}
          s_* & 0 \\
          \smp_2 &
        \end{tpattern} \\
        M_\smp(0):
        & \begin{tpattern}
          s_* & 1 \\
          \smp_2 &
        \end{tpattern}
        & \mapsto &
        \begin{tpattern}
          s_* & 1 \\
          \smp_1 &
        \end{tpattern} \\
        & \begin{tpattern}
          s_* & 2 \\
          \smp_2 & 
        \end{tpattern}
        & \mapsto &
        \begin{tpattern}
          s_* & 1 \\
          \smp_2 &
        \end{tpattern} \\
        & \begin{tpattern}
          s_* & 1 \\
          & \smp_1 
        \end{tpattern}
        & \mapsto &
        \begin{tpattern}
          s_* & 2 \\
          \smp_1 &
        \end{tpattern} \\
        & \begin{tpattern}
          s_* & 2 \\
          & \smp_1
        \end{tpattern}
        & \mapsto &
        \begin{tpattern}
          s_* & 2 \\
          \smp_2 &
        \end{tpattern}
      \end{array}
    \]
    \normalsize

    \medskip
    \footnotesize
    \[
      \begin{array}{lrcl}
        \text{Special $\smq_1$ higher level:}
        & \begin{tpattern}
          0 & s_* \\
          \smq_1 &
        \end{tpattern}
        & \mapsto &
        \begin{tpattern}
          0 & s_* \\
          & \smq_1
        \end{tpattern} \\
        \text{Special $\smq_2$ higher level:}
        & \begin{tpattern}
          0 & s_* \\
          & \smq_2
        \end{tpattern}
        & \mapsto &
        \begin{tpattern}
          0 & s_* \\
          & \smq_2
        \end{tpattern} \\
        M_\smq(0):
        & \begin{tpattern}
          1 & s_* \\
          & \smq_2
        \end{tpattern}
        & \mapsto &
        \begin{tpattern}
          1 & s_* \\
          & \smq_1
        \end{tpattern} \\
        & \begin{tpattern}
          2 & s_* \\
          & \smq_2
        \end{tpattern}
        & \mapsto & 
        \begin{tpattern}
          1 & s_* \\
          & \smq_2
        \end{tpattern} \\
        & \begin{tpattern}
          1 & s_* \\
          \smq_1 & 
        \end{tpattern}
        & \mapsto &
        \begin{tpattern}
          2 & s_* \\
          & \smq_1
        \end{tpattern} \\
        & \begin{tpattern}
          2 & s_* \\
          \smq_1 & 
        \end{tpattern}
        & \mapsto & 
        \begin{tpattern}
          2 & s_* \\
          & \smq_2
        \end{tpattern}
      \end{array}
    \]
    \normalsize
  \end{mdframed}

  \caption{Encoding SMART configurations: $\pencoding_{\mathrm{init}}$}
  \captionsetup{width=.82\linewidth,font=small}
  \caption*{$\pencoding_\mathrm{init}$ rewrites sub-patterns of length $2$ in cyclic SMART configurations.}
  \label{fig:encoding-initialization}
\end{figure}

\begin{figure}
  \vspace*{-2.5em}
  \begin{mdframed}[backgroundcolor=black!5!white]
    From $M_\smb(k)$ to $(k+1)$-level encoding:
    \footnotesize
    \[ 
      \begin{array}{lrcl}
        M_\smb(k+1) (\text{if } \val_3(c) \neq 0):
        & \begin{tpatternc}
          \textcolor{gray}{*} & c & 0 & \textcolor{gray}{*} \\
          & & \smb_1 &
        \end{tpatternc}
        & \mapsto &
        \begin{tpatternc}
          \textcolor{gray}{*} & \widt{\midwid}{[\val_3(c)+0]_{(3)}} & \textcolor{gray}{*} \\
          & & \smb_1
        \end{tpatternc} \\
        & \begin{tpatternc}
          \textcolor{gray}{*} & c & 0 & \textcolor{gray}{*} \\
          & & \smb_2 &
        \end{tpatternc}
        & \mapsto &
        \begin{tpatternc}
          \textcolor{gray}{*} & \widt{\midwid}{[\val_3(c)+0]_{(3)}} & \textcolor{gray}{*} \\
          & & \smb_2
        \end{tpatternc} \\
        M_\smd(k+1) (\text{if } \val_3(c) \neq 0):
        & \begin{tpatternc}
          \textcolor{gray}{*} & c & 1 & \textcolor{gray}{*} \\
          & & \smb_1 &
        \end{tpatternc}
        & \mapsto &
        \begin{tpatternc}
          \textcolor{gray}{*} & \widt{\midwid}{[\val_3(c) + f(k) + 1]_{(3)}} & \textcolor{gray}{*} \\
          \smd_1 & &
        \end{tpatternc}\\
        & \begin{tpatternc}
          \textcolor{gray}{*} & c & 2 & \textcolor{gray}{*} \\
          & & \smb_1 &
        \end{tpatternc}
        & \mapsto &
        \begin{tpatternc}
          \textcolor{gray}{*} & \widt{\midwid}{[\val_3(c) + f(k) + 1]_{(3)}} & \textcolor{gray}{*} \\
          \smd_2 & &
        \end{tpatternc} \\
        M_\smp(k+1) (\text{if } \val_3(c) \neq 0):
        & \begin{tpatternc}
          \textcolor{gray}{*} & c & 1 & \textcolor{gray}{*} \\
          & & \smb_2 &
        \end{tpatternc}
        & \mapsto &
        \begin{tpatternc}
          \textcolor{gray}{*} & \widt{\midwid}{[\val_3(c) + 2]_{(3)}} & \textcolor{gray}{*} \\
          \smp_1 & &
        \end{tpatternc} \\
        & \begin{tpatternc}
          \textcolor{gray}{*} & c & 2 & \textcolor{gray}{*} \\
          & & \smb_2 &
        \end{tpatternc}
        & \mapsto &
        \begin{tpatternc}
          \textcolor{gray}{*} & \widt{\midwid}{[\val_3(c) + 2]_{(3)}} & \textcolor{gray}{*} \\
          \smp_2 & &
        \end{tpatternc}
      \end{array}
    \]
    \normalsize
    From $M_\smd(k)$ to $(k+1)$-level encoding:
    \footnotesize
    \[
      \begin{array}{lrcl}
        M_\smd(k+1) (\text{if } \val_3(c) \neq 0):
        & \begin{tpatternc}
          \textcolor{gray}{*} & 0 & c & \textcolor{gray}{*} \\
          & \smd_1 & &
        \end{tpatternc}
        & \mapsto &
        \begin{tpatternc}
          \textcolor{gray}{*} & \widt{\midwid}{[\val_3(c)+0]_{(3)}} & \textcolor{gray}{*} \\
          \smd_1 & &
        \end{tpatternc} \\
        & \begin{tpatternc}
          \textcolor{gray}{*} & 0 & c & \textcolor{gray}{*} \\
          & \smd_2 & &
        \end{tpatternc}
        & \mapsto &
        \begin{tpatternc}
          \textcolor{gray}{*} & \widt{\midwid}{[\val_3(c)+0]_{(3)}} & \textcolor{gray}{*} \\
          \smd_2 & &
        \end{tpatternc} \\
        M_\smb(k+1) (\text{if } \val_3(c) \neq 0):
        & \begin{tpatternc}
          \textcolor{gray}{*} & 1 & c & \textcolor{gray}{*} \\
          & \smd_1 & &
        \end{tpatternc}
        & \mapsto &
        \begin{tpatternc}
          \textcolor{gray}{*} & \widt{\midwid}{[\val_3(c) + f(k)+1]_{(3)}} & \textcolor{gray}{*} \\
          & & \smb_1
        \end{tpatternc} \\
        & \begin{tpatternc}
          \textcolor{gray}{*} & 2 & c & \textcolor{gray}{*} \\
          & \smd_1 & &
        \end{tpatternc}
        & \mapsto &
        \begin{tpatternc}
          \textcolor{gray}{*} & \widt{\midwid}{[\val_3(c) + f(k)+1]_{(3)}} & \textcolor{gray}{*} \\
          & & \smb_2
        \end{tpatternc} \\
        M_\smq(k+1) (\text{if } \val_3(c) \neq 0):
        & \begin{tpatternc}
          \textcolor{gray}{*} & 1 & c & \textcolor{gray}{*} \\
          & \smd_2 & &
        \end{tpatternc}
        & \mapsto &
        \begin{tpatternc}
          \textcolor{gray}{*} & \widt{\midwid}{[\val_3(c) + 2]_{(3)}} & \textcolor{gray}{*} \\
          & & \smq_1
        \end{tpatternc} \\
        & \begin{tpatternc}
          \textcolor{gray}{*} & 2 & c & \textcolor{gray}{*} \\
          & \smd_2 & &
        \end{tpatternc}
        & \mapsto &
        \begin{tpatternc}
          \textcolor{gray}{*} & \widt{\midwid}{[\val_3(c) + 2]_{(3)}} & \textcolor{gray}{*} \\
          & & \smq_2
        \end{tpatternc}
      \end{array}
    \]
    \normalsize
    From $M_\smp(k)$ to $(k+1)$-level encoding:
    \footnotesize
    \[
      \begin{array}{lrcl}
        M_\smp(k+1) (\text{if } \val_3(c) \neq 0):
        & \begin{tpatternc}
          \textcolor{gray}{*} & 0 & c & \textcolor{gray}{*} \\
          & \smp_1 & &
        \end{tpatternc}
        & \mapsto &
        \begin{tpatternc}[]
          \textcolor{gray}{*} & \widt{\midwid}{[\val_3(c) + 2f(k) + 4]_{(3)}} & \textcolor{gray}{*} \\
          \smp_1 & &
        \end{tpatternc} \\
        & \begin{tpatternc}
          \textcolor{gray}{*} & 0 & c & \textcolor{gray}{*} \\
          & \smp_2 & &
        \end{tpatternc}
        & \mapsto &
        \begin{tpatternc}[]
          \textcolor{gray}{*} & \widt{\midwid}{[\val_3(c) + 2f(k) + 4]_{(3)}} & \textcolor{gray}{*} \\
          \smp_2 & &
        \end{tpatternc} \\
        M_\smb(k+1) (\text{if } 0 \;\; c \neq 0):
        & \begin{tpatternc}
          \textcolor{gray}{*} & 1 & c & \textcolor{gray}{*} \\
          & \smp_1 & &
        \end{tpatternc}
        & \mapsto &
        \begin{tpatternc}[]
          \textcolor{gray}{*} & \widt{\midwid}{[\val_3(c) + 2f(k)+2]_{(3)}} & \textcolor{gray}{*} \\
          & & \smb_1
        \end{tpatternc} \\
        & \begin{tpatternc}
          \textcolor{gray}{*} & 2 & c & \textcolor{gray}{*} \\
          & \smp_1 & &
        \end{tpatternc}
        & \mapsto &
        \begin{tpatternc}[]
          \textcolor{gray}{*} & \widt{\midwid}{[\val_3(c) + 2f(k)+2]_{(3)}} & \textcolor{gray}{*} \\
          & & \smb_2
        \end{tpatternc} \\
        M_\smq(k+1) (\text{if } \val_3(c) \neq 0):
        & \begin{tpatternc}
          \textcolor{gray}{*} & 1 & c & \textcolor{gray}{*} \\
          & \smp_2 & &
        \end{tpatternc}
        & \mapsto &
        \begin{tpatternc}[]
          \textcolor{gray}{*} & \widt{\midwid}{[\val_3(c) + f(k)+3]_{(3)}} & \textcolor{gray}{*} \\
          & & \smq_1
        \end{tpatternc} \\
        & \begin{tpatternc}
          \textcolor{gray}{*} & 2 & c & \textcolor{gray}{*} \\
          & \smp_2 & &
        \end{tpatternc}
        & \mapsto &
        \begin{tpatternc}[]
          \textcolor{gray}{*} & \widt{\midwid}{[\val_3(c) + f(k)+3]_{(3)}} & \textcolor{gray}{*} \\
          & & \smq_2
        \end{tpatternc}
      \end{array}
    \]
    \normalsize
    From $M_\smq(k)$ to $(k+1)$-level encoding:
    \footnotesize
    \[
      \begin{array}{lrcl}
        M_\smq(k+1) (\text{if } \val_3(c) \neq 0):
        & \begin{tpatternc}
          \textcolor{gray}{*} & c & 0 & \textcolor{gray}{*} \\
          & & \smq_1
        \end{tpatternc}
        & \mapsto &
        \begin{tpatternc}
          \textcolor{gray}{*} & \widt{\midwid}{[ \val_3(c) + 2f(k) + 4]_{(3)}} & \textcolor{gray}{*} \\
          & & \smq_1
        \end{tpatternc} \\
        & \begin{tpatternc}
          \textcolor{gray}{*} & c & 0 & \textcolor{gray}{*} \\
          & & \smq_2
        \end{tpatternc}
        & \mapsto &
        \begin{tpatternc}
          \textcolor{gray}{*} & \widt{\midwid}{[\val_3(c) + 2f(k) + 4]_{(3)}} & \textcolor{gray}{*} \\
          & & \smq_2
        \end{tpatternc} \\
        M_\smd(k+1) (\text{if }\val_3(c) \neq 0):
        & \begin{tpatternc}
          \textcolor{gray}{*} & c & 1 & \textcolor{gray}{*} \\
          & & \smq_1 &
        \end{tpatternc}
        & \mapsto &
        \begin{tpatternc}
          \textcolor{gray}{*} & \widt{\midwid}{[\val_3(c) + 2f(k) + 2]_{(3)}} & \textcolor{gray}{*} \\
          \smd_1 & &
        \end{tpatternc} \\
        & \begin{tpatternc}
          \textcolor{gray}{*} & c  & 2 & \textcolor{gray}{*} \\
          & & \smq_1
        \end{tpatternc}
        & \mapsto &
        \begin{tpatternc}
          \textcolor{gray}{*} & \widt{\midwid}{[\val_3(c) + 2f(k) + 2]_{(3)}} & \textcolor{gray}{*} \\
          \smd_2 &
        \end{tpatternc} \\
        M_\smp(k+1) (\text{if } \val_3(c) \neq 0):
        & \begin{tpatternc}
          \textcolor{gray}{*} & c & 1 & \textcolor{gray}{*} \\
          & & \smq_2 &
        \end{tpatternc}
        & \mapsto &
        \begin{tpatternc}
          \textcolor{gray}{*} & \widt{\midwid}{[\val_3(c) + f(k) + 3]_{(3)}} & \textcolor{gray}{*} \\
          \smp_1 & & 
        \end{tpatternc} \\
        & \begin{tpatternc}
          \textcolor{gray}{*} & c & 2 & \textcolor{gray}{*} \\
          & & \smq_2 &
        \end{tpatternc}
        & \mapsto &
        \begin{tpatternc}
          \textcolor{gray}{*} & \widt{\midwid}{[\val_3(c) + f(k) + 3]_{(3)}} & \textcolor{gray}{*} \\
          \smp_2 & &
        \end{tpatternc}
      \end{array}
    \]
    \normalsize
  \end{mdframed}

  \caption{Encoding SMART configurations: $\pencoding_{k \to k+1}$ (Part 1: $k \to k+1$)}
  \captionsetup{width=.95\linewidth,font=small}
  \caption*{First half of rewriting cases of $\pencoding_{k \to k+1}$, which rewrites sub-patterns of length $k+4$. Letters $\textcolor{gray}{*}$ are unmodified. $\pencoding_{k \to k+1}$ extends the word $c \in \{0,1,2\}^{k+1}$ with one additional letter and, considering $c$ as a counter, adds a number of steps in accordance with Lemma~\ref{prop:smart-moves}. \\Note that at $k+1 = \ell-2$ (resp. $k+1 = \ell-1$), the $\textcolor{gray}{*}$-cells overlap on each other (resp.\ the counter), because we reach the length of the cyclic tape.}
  \label{fig:bottom-up-encoding}
\end{figure}

\begin{figure}
  \vspace*{-2.5em}
  \begin{mdframed}[backgroundcolor=black!5!white]
    Encoding special $M_\smb(k+1)$:
    \footnotesize
    \[ 
      \begin{array}{lrcl}
        \text{Special } \smb_1:
        & \begin{tpattern}
          \textcolor{gray}{*} & 1 & 0^{k+1} & \textcolor{gray}{*} \\
          & & & \smb_1
        \end{tpattern}
        & \mapsto &
        \begin{tpattern}
          \textcolor{gray}{*} & \widt{\midwid}{[f(k+1)+1]_{(3)}} & \textcolor{gray}{*} \\
          & & \smb_1
        \end{tpattern} \\
        & \begin{tpattern}
          \textcolor{gray}{*} & 2 & 0^{k+1} & \textcolor{gray}{*} \\
          & & & \smb_1
        \end{tpattern}
        & \mapsto &
        \begin{tpattern}
          \textcolor{gray}{*} & \widt{\midwid}{[f(k+1)+1]_{(3)}} & \textcolor{gray}{*} \\
          & & \smb_2
        \end{tpattern} \\
        \text{Special } \smb_2:
        & \begin{tpattern}
          \textcolor{gray}{*} & 1 & 0^{k+1} & \textcolor{gray}{*} \\
          & & & \smb_2 
        \end{tpattern}
        & \mapsto &
        \begin{tpattern}
          \textcolor{gray}{*} & \widt{\midwid}{[f(k+1)]_{(3)}} & \textcolor{gray}{*} \\
          & & \smb_1
        \end{tpattern} \\
        & \begin{tpattern}
          \textcolor{gray}{*} & 2 & 0^{k+1} & \textcolor{gray}{*} \\
          & & & \smb_2
        \end{tpattern}
        & \mapsto &
        \begin{tpattern}
          \textcolor{gray}{*} & \widt{\midwid}{[f(k+1)]_{(3)}} & \textcolor{gray}{*} \\
          & & \smb_2
        \end{tpattern} \\
        \text{Special $\smb_1$ higher level}:
        & \begin{tpattern}
          \textcolor{gray}{*} & 0 & 0^{k+1} & \textcolor{gray}{*} \\
          & & & \smb_1
        \end{tpattern}
        & \mapsto & 
        \begin{tpattern}
          \textcolor{gray}{*} & \widt{\midwid}{0^{k+2}} & \textcolor{gray}{*} \\
          & & \smb_1
        \end{tpattern} \\
        \text{Special $\smb_2$ higher level}:
        & \begin{tpattern}
          \textcolor{gray}{*} & 0 & 0^{k+1} & \textcolor{gray}{*} \\
          & & & \smb_2
        \end{tpattern}
        & \mapsto & 
        \begin{tpattern}
          \textcolor{gray}{*} & \widt{\midwid}{0^{k+2}} & \textcolor{gray}{*} \\
          & & \smb_2
        \end{tpattern} \\
      \end{array}
    \]
    \normalsize
    Encoding special $M_\smd(k+1)$:
    \footnotesize
    \[
      \begin{array}{lrcl}
        \text{Special } \smd_1:
        & \begin{tpattern}
          \textcolor{gray}{*} & 0^{k+1} & 1 & \textcolor{gray}{*} \\
          \smd_1 & & &
        \end{tpattern}
        & \mapsto &
        \begin{tpattern}
          \textcolor{gray}{*} & \widt{\midwid}{[f(k+1)+1]_{(3)}} & \textcolor{gray}{*} \\
          \smd_1 & &
        \end{tpattern} \\
        & \begin{tpattern}
          \textcolor{gray}{*} & 0^{k+1} & 2 & \textcolor{gray}{*} \\
          \smd_1 & & &
        \end{tpattern}
        & \mapsto &
        \begin{tpattern}
          \textcolor{gray}{*} & \widt{\midwid}{[f(k+1)+1]_{(3)}} & \textcolor{gray}{*} \\
          \smd_2 & &
        \end{tpattern} \\
        \text{Special } \smd_2:
        & \begin{tpattern}
          \textcolor{gray}{*} & 0^{k+1} & 1 & \textcolor{gray}{*} \\
          \smd_2 & & &
        \end{tpattern}
        & \mapsto &
        \begin{tpattern}
          \textcolor{gray}{*} & \widt{\midwid}{[f(k+1)]_{(3)}} & \textcolor{gray}{*} \\
          \smd_1 & &
        \end{tpattern} \\
        & \begin{tpattern}
          \textcolor{gray}{*} & 0^{k+1} & 2 & \textcolor{gray}{*} \\
          \smd_2 & & &
        \end{tpattern}
        & \mapsto &
        \begin{tpattern}
          \textcolor{gray}{*} & \widt{\midwid}{[f(k+1)]_{(3)}} & \textcolor{gray}{*} \\
          \smd_2 & &
        \end{tpattern} \\
        \text{Special $\smd_1$ higher level}:
        & \begin{tpattern}
          \textcolor{gray}{*} & 0^{k+1} & 0 & \textcolor{gray}{*} \\
          \smd_1 & & &
        \end{tpattern}
        & \mapsto &
        \begin{tpattern}
          \textcolor{gray}{*} & \widt{\midwid}{0^{k+2}}  & \textcolor{gray}{*} \\
          \smd_1 & & 
        \end{tpattern} \\
        \text{Special $\smd_2$ higher level}:
        & \begin{tpattern}
          \textcolor{gray}{*} & 0^{k+1} & 0 & \textcolor{gray}{*} \\
          \smd_2 & & &
        \end{tpattern}
        & \mapsto &
        \begin{tpattern}
          \textcolor{gray}{*} & \widt{\midwid}{0^{k+2}} & \textcolor{gray}{*} \\
          \smd_2 & & 
        \end{tpattern}
      \end{array}
    \]
    \normalsize
    Encoding special $M_\smp(k+1)$:
    \footnotesize
    \[
      \begin{array}{lrcl}
        \text{Special } \smp_1:
        & \begin{tpattern}
          \textcolor{gray}{*} & 0^{k+1} & 1 & \textcolor{gray}{*} \\
          \smp_1 & & &
        \end{tpattern}
        & \mapsto &
        \begin{tpattern}
          \textcolor{gray}{*} & [2]_{(3)} & \textcolor{gray}{*} \\
          \smp_1 & &
        \end{tpattern} \\
        & \begin{tpattern}
          \textcolor{gray}{*} & 0^{k+1} & 2 & \textcolor{gray}{*} \\
          \smp_1 & & &
        \end{tpattern}
        & \mapsto &
        \begin{tpattern}
          \textcolor{gray}{*} & [2]_{(3)} & \textcolor{gray}{*} \\
          \smp_2 & &
        \end{tpattern} \\
        \text{Special } \smp_2:
        & \begin{tpattern}
          \textcolor{gray}{*} & 0^{k+1} & 1 & \textcolor{gray}{*} \\
          \smp_2 & & &
        \end{tpattern}
        & \mapsto &
        \begin{tpattern}
          \textcolor{gray}{*} & [1]_{(3)} & \textcolor{gray}{*} \\
          \smp_1 & &
        \end{tpattern} \\
        & \begin{tpattern}
          \textcolor{gray}{*} & 0^{k+1} & 2 & \textcolor{gray}{*} \\
          \smp_2 & & &
        \end{tpattern}
        & \mapsto &
        \begin{tpattern}
          \textcolor{gray}{*} & [1]_{(3)} & \textcolor{gray}{*} \\
          \smp_2 & &
        \end{tpattern} \\
        \text{Special $\smp_1$ higher level}:
        & \begin{tpattern}
          \textcolor{gray}{*} & 0^{k+1} & 0 & \textcolor{gray}{*} \\
          \smp_1 & & &
        \end{tpattern}
        & \mapsto &
        \begin{tpattern}
          \textcolor{gray}{*} & 0^{k+2} & \textcolor{gray}{*} \\
          \smp_1 & &
        \end{tpattern} \\
        \text{Special $\smp_1$ higher level}:
        & \begin{tpattern}
          \textcolor{gray}{*} & 0^{k+1} & 0 & \textcolor{gray}{*} \\
          \smp_2 & & &
        \end{tpattern}
        & \mapsto &
        \begin{tpattern}
          \textcolor{gray}{*} & 0^{k+2} & \textcolor{gray}{*} \\
          \smp_2 & &
        \end{tpattern}
      \end{array}
    \]
    \normalsize
    Encoding special $M_\smq(k+1)$:
    \footnotesize
    \[
      \begin{array}{lrcl}
        \text{Special } \smq_1:
        & \begin{tpattern}
          \textcolor{gray}{*} & 1 & 0^{k+1} & \textcolor{gray}{*} \\
          & & & \smq_1
        \end{tpattern}
        & \mapsto &
        \begin{tpattern}[]
          \textcolor{gray}{*} & [2]_{(3)} & \textcolor{gray}{*} \\
          & & \smq_1
        \end{tpattern} \\
        & \begin{tpattern}
          \textcolor{gray}{*} & 2 & 0^{k+1} & \textcolor{gray}{*} \\
          & & & \smq_1
        \end{tpattern}
        & \mapsto &
        \begin{tpattern}[]
          \textcolor{gray}{*} & [2]_{(3)} & \textcolor{gray}{*} \\
          & & \smq_2
        \end{tpattern} \\
        \text{Special } \smq_2:
        & \begin{tpattern}
          \textcolor{gray}{*} & 1 & 0^{k+1} & \textcolor{gray}{*} \\
          & & & \smq_2
        \end{tpattern}
        & \mapsto &
        \begin{tpattern}[]
          \textcolor{gray}{*} & [1]_{(3)} & \textcolor{gray}{*} \\
          & & \smq_1
        \end{tpattern} \\
        & \begin{tpattern}
          \textcolor{gray}{*} & 2 & 0^{k+1} & \textcolor{gray}{*} \\
          & & & \smq_2
        \end{tpattern}
        & \mapsto &
        \begin{tpattern}[]
          \textcolor{gray}{*} & [1]_{(3)} & \textcolor{gray}{*} \\
          & & \smq_2
        \end{tpattern} \\
        \text{Special $\smq_1$ higher level}:
        & \begin{tpattern}
          \textcolor{gray}{*} & 0 & 0^{k+1} & \textcolor{gray}{*} \\
          & & & \smq_1
        \end{tpattern}
        & \mapsto &
        \begin{tpattern}
          \textcolor{gray}{*} & 0^{k+2} & \textcolor{gray}{*} \\
          & & \smq_1
        \end{tpattern} \\
        \text{Special $\smq_2$ higher level}:
        & \begin{tpattern}
          \textcolor{gray}{*} & 0 & 0^{k+1} & \textcolor{gray}{*} \\
          & & & \smq_2
        \end{tpattern}
        & \mapsto &
        \begin{tpattern}
          \textcolor{gray}{*} & 0^{k+2} & \textcolor{gray}{*} \\
          & & \smq_2
        \end{tpattern}
      \end{array}
    \]
    \normalsize
  \end{mdframed}

  \caption{Encoding SMART configurations: $\pencoding_{k \to k+1}$ (Part 2: special $\to k+1$)}
  \captionsetup{width=.95\linewidth,font=small}
  \caption*{Second half of rewriting cases of $\pencoding_{k \to k+1}$, which rewrites sub-patterns of length $k+4$. Letters $\textcolor{gray}{*}$ are unmodified. Encodes special configurations of level $k+1$ by replacing the $k+2$ other letters by a counter of $\{0,1,2\}^{k+2}$ in accordance with Lemma~\ref{prop:smart-moves}, and preserve special configurations of level $> k+1$. \\Note that at $k+1 = \ell-2$ (resp. $k+1 = \ell-1$), the $\textcolor{gray}{*}$-cells overlap on each other (resp.\ the counter), because we reach the length of the cyclic tape.}
  \label{fig:bottom-up-encoding2}
\end{figure}

\begin{figure}
  \begin{mdframed}[backgroundcolor=black!5!white]
    \footnotesize
    \[ 
      \begin{array}{lrcl}
        & \begin{turing}
          0\;\;\;\;\;\; & \widt{\sidwid}{0^{\ell-2}} & 0 \\
          \smb_1 & &
        \end{turing}
        & \mapsto &
        \begin{turing}
          0 & \widt{\sidwid}{0^{\ell-2}} & 0 \\
          \smb_1 & &
        \end{turing} \\
        & \begin{turing}
          0\;\;\;\;\;\; & \widt{\sidwid}{0^{\ell-2}} & 0 \\
          & \smb_2 &
        \end{turing}
        & \mapsto &
        \begin{turing}
          2 & \widt{\sidwid}{2^{\ell-2}} & 2 \\
          \smb_2 & &
        \end{turing} \\
        & \begin{turing}
          c_0\;\;\;\; & \widt{1.1cm}{\cdots} & c_{\ell-1} \\
          \smb_1 & &
        \end{turing}
        & \mapsto &
        \begin{turing}[]
          c  \\
          \smd_1 \hspace*{1.5cm}
        \end{turing} \\
        & \begin{turing}
          c_0\;\;\;\; & \widt{1.1cm}{\cdots} & c_{\ell-1} \\
          \smb_2 & &
        \end{turing}
        & \mapsto &
        \begin{turing}[]
          c \\
          \smp_1 \hspace*{1.5cm}
        \end{turing} \\
        & & & \\
        & \begin{turing}
          0\;\;\;\;\;\; & \widt{\sidwid}{0^{\ell-2}} & 0 \\
          \smd_1 & &
        \end{turing}
        & \mapsto &
        \begin{turing}
          0 & \widt{\sidwid}{0^{\ell-2}} & 0 \\
          \smd_1 & & 
        \end{turing} \\
        & \begin{turing}
          0\;\;\;\;\;\; & \widt{\sidwid}{0^{\ell-2}} & 0 \\
          & & \smd_2
        \end{turing}
        & \mapsto &
        \begin{turing}
          2 & \widt{\sidwid}{2^{\ell-2}} & 2 \\
          \smd_2 & & 
        \end{turing} \\
        & \begin{turing}
          c_{\ell-1} & \widt{1.1cm}{c_0 \;\; \cdots} & c_{\ell-2} \\
          \smd_1 & &
        \end{turing}
        & \mapsto &
        \begin{turing}[]
          c \\
          \smb_1 \hspace*{1.5cm}
        \end{turing} \\
        & \begin{turing}
          c_{\ell-1} & \widt{1.1cm}{c_0 \;\; \cdots} & c_{\ell-2} \\
          \smd_2 & &
        \end{turing}
        & \mapsto &
        \begin{turing}[]
          c \\
          \smq_1 \hspace*{1.5cm}
        \end{turing} \\
        & & & \\
        & \begin{turing}
          0\;\;\;\;\;\; & \widt{\sidwid}{0^{\ell-2}} & 0 \\
          & & \smp_1
        \end{turing}
        & \mapsto &
        \begin{turing}
          0 & \widt{\sidwid}{0^{\ell-2}} & 0 \\
          \smp_1 & & 
        \end{turing} \\
        & \begin{turing}
          0\;\;\;\;\;\; & \widt{\sidwid}{0^{\ell-2}} & 0 \\
          \smp_2 & &
        \end{turing}
        & \mapsto &
        \begin{turing}
          2 & \widt{\sidwid}{2^{\ell-2}} & 2 \\
          \smp_2 & & 
        \end{turing} \\
        & \begin{turing}
          c_{\ell-1} & \widt{1.1cm}{c_0 \;\;\cdots} & c_{\ell-2} \\
          \smp_1 & &
        \end{turing}
        & \mapsto &
        \begin{turing}[]
          [\val_3(c)-1]_{(3)} \\
          \smb_2 \hspace*{1.48cm}
        \end{turing} \\
        & \begin{turing}
          c_{\ell-1} & \widt{1.1cm}{c_0 \;\;\cdots} & c_{\ell-2} \\
          \smp_2 & &
        \end{turing}
        & \mapsto &
        \begin{turing}[]
          [\val_3(c)-1]_{(3)} \\
          \smq_2 \hspace*{1.48cm}
        \end{turing} \\
        & & & \\
        & \begin{turing}
          0\;\;\;\;\;\; & \widt{\sidwid}{0^{\ell-2}} & 0\\
          & \smq_1 &
        \end{turing}
        & \mapsto &
        \begin{turing}
          0 & \widt{\sidwid}{0^{\ell-2}} & 0 \\
          \smq_1 & &
        \end{turing} \\
        & \begin{turing}
          0\;\;\;\;\;\; & \widt{\sidwid}{0^{\ell-2}} & 0 \\
          \smq_2 & &
        \end{turing}
        & \mapsto &
        \begin{turing}
          2 & \widt{\sidwid}{2^{\ell-2}} & 2\\
          \smq_2 & &
        \end{turing} \\
        & \begin{turing}
          c_0\;\;\;\; & \widt{1.1cm}{\cdots} & c_{\ell-1} \\
          \smq_1 & &
        \end{turing}
        & \mapsto &
        \begin{turing}[]
          [\val_3(c)-1]_{(3)} \\
          \smd_2 \hspace*{1.48cm}
        \end{turing} \\
        & \begin{turing}
          c_0\;\;\;\; & \widt{1.1cm}{\cdots} & c_{\ell-1} \\
          \smq_2 & &
        \end{turing}
        & \mapsto &
        \begin{turing}[]
          [\val_3(c)-1]_{(3)} \\
          \smp_2 \hspace*{1.48cm}
        \end{turing}
      \end{array}
    \]
    \normalsize
  \end{mdframed}

  \caption{Final encoding step of SMART configurations: $\pencoding_{\ell,\mathrm{final}}$}
  \captionsetup{width=.95\linewidth,font=small}
  \caption*{Rewrites complete cyclic configurations of length $\ell$. $\pencoding_{\ell,\mathrm{final}}$ acts according to the proof of Proposition~\ref{prop:smart-eventually-shifts}: it maps encodings of level $\ell-1$ to their final encodings, and ``corrects'' the position the head and shifts the counter when required (in the encodings of $M_\smd(\ell-1)$ and $M_\smp(\ell-1)$, or in initial configurations, whose heads were moved when applying $\pencoding_{\mathrm{init}}$).}
  \label{fig:bottom-up-encoding3}
\end{figure}

In this section, we use the analysis performed in Section~\ref{sec:smart-analysis} to provide a linear-time algorithm that computes this encoding $\encoding_\ell : \ctape_\ell \to \ctape_\ell$ inductively.

\smallskip
Figure~\ref{fig:encoding-initialization} describes a piecewise-defined bijection $\pencoding_{\mathrm{init}} : \ctape_\ell \to \ctape_\ell$: each case describes how a pattern of length $2$ (e.g. $ \begin{tpattern} 1 & s_* \\ \smb_2 & \end{tpattern}$) is bijectively replaced by another (in the previous example, by $\begin{tpattern}1 & s_* \\ & \smb_1 \end{tpattern}$). In other words, if a cyclic configuration $w$ contains a sub-pattern of length $2$ that matches with one case of Figure~\ref{fig:encoding-initialization}, then $\pencoding_\mathrm{init}$ replaces this sub-pattern in $w$ by its image in the figure.

\smallskip
Similarly, Figures~\ref{fig:bottom-up-encoding} and~\ref{fig:bottom-up-encoding2} together describe a piecewise-defined bijection $\pencoding_{k \to k+1} : \ctape_\ell \to \ctape_\ell$ that replaces sub-patterns of length $k+4$. Intuitively, Figure~\ref{fig:bottom-up-encoding} (defining the first half of $\pencoding_{k \to k+1}$) encodes moves of level $k+1$ into counters, while Figure~\ref{fig:bottom-up-encoding2} (the second half of $\pencoding_{k \to k+1}$) encodes special patterns of level $k+1$. 

\smallskip
Finally, Figure~\ref{fig:bottom-up-encoding3} describes a similar bijection $\pencoding_{\ell,\mathrm{final}} : \ctape_\ell \to \ctape_\ell$.

\medskip
We can then prove the following result:
\begin{lemma}\label{lem:encoding-is-correct}Let $w$ be a cyclic configuration of $\ctape_\ell$. Then:
  \[ \encoding_\ell(w) = \left[ \pencoding_{\ell,\mathrm{final}} \circ \prod_{k=0}^{\ell-2} \pencoding_{k \to k+1} \circ \pencoding_{\mathrm{init}} \right] (w). \]
\end{lemma}

\begin{proof}[Sketch of proof]
  Let $w$ be a configuration, and $k$ be some integer $k \leq \ell-1$. If $w$ is neither the shift of an initial configuration nor a special configuration of level $> k$, then there exists some $q \in \{\smb,\smd,\smp,\smq\}$ and a unique word $p = p_0 \dots p_{k+1}$ of length $k+2$ such that $p \sqsubseteq w$ and $p$ computes the $j\text{-th}$ step of $M_q(k)$ for some $0 \leq j \leq f(k) = 3^{k+2} - 2$.

  Then by induction on $k$, one sees that the partial composition
  \[ \prod_{k'=0}^{k-1} \pencoding_{k' \to k'+1} \circ \pencoding_\mathrm{init}, \]
  when applied on $w$, replaces $p$ in $w$ by another pattern $p'$ of the same length defined as follows:
  \[
    p' = \begin{cases}
      \begin{tpatternc}
        c_0 & \dots & c_k & p_{k+1} \\
        & & & \smb_b
      \end{tpatternc} & \text{(where $b = p_0 \in \{1,2\}$) if $p$ performs $M_\smb(k)$}\\
      \begin{tpatternc}
        p_0 & c_0 & \dots & c_k \\
        \smd_b
      \end{tpatternc} & \text{(where $b = p_{k+1} \in \{1,2\}$) if $p$ performs $M_\smd(k)$}\\
      \begin{tpatternc}
        p_0 & c_0 & \dots & c_k \\
        \smp_b & & &
      \end{tpatternc} & \text{(where $b = p_{k+1} \in \{1,2\}$) if $p$ performs $M_\smp(k)$}\\
      \begin{tpatternc}
        c_0 & \dots & c_k & p_{k+1} \\
        & & & \smq_b
      \end{tpatternc} & \text{(where $b = p_0 \in \{1,2\}$) if $p$ performs $M_\smq(k)$}
    \end{cases}
  \]
  where $c \in \{0,1,2\}^{k+1}$ is a ternary counter such that $\val_3(c) = j+1$. Notice that $0 \leq j \leq f(k) -2$, where $f(k) = 3^{k+2}-2$; so that $1 \leq j +1 \leq f(k)-1$ fits exactly in the space of non-zero counters. The zero counters are reserved for (shifts of) initial and special configurations.

  \medskip
  Finally, 
  \[\pencoding_{\ell,\mathrm{final}} \circ \prod_{k=0}^{\ell-2} \pencoding_{k \to k+1} \circ \pencoding_{\mathrm{init}} \]
  is equal to $\encoding_\ell$ by considering how $\pencoding_{\ell,\mathrm{final}}$ acts in accordance with the structure of the four disjoint cycles of SMART (detailed the proof of Proposition~\ref{prop:smart-eventually-shifts}).
\end{proof}

\bigskip
\begin{example}
  Consider once again the cyclic configuration of length $\ell = 6$ from Example~\ref{ex:FindingStep1}, and let us use the formulas above to encode it into a counter value. This process will roughly mirror Examples~\ref{ex:FindingStep1} and~\ref{ex:FindingStep2}, except that due to our coding convention, the counter value is one larger than the correct one until the very last step. First, we apply $F_{\init}$, which observes based on the highlighted cells

  \footnotesize
  \[ 
    w = \begin{turingc}
   \textcolor{gray}{1} & \textcolor{gray}{2} & \textcolor{gray}{1} & 1 & 2 & \textcolor{gray}{0} \\
    & & & \smb_2 & &
  \end{turingc}
  \]\normalsize
  \noindent that the first case of $M_\smb(0)$ applies, and rewrites this as

  \footnotesize
  \[ 
    w = \begin{turingc}
   \textcolor{gray}{1} & \textcolor{gray}{2} & \textcolor{gray}{1} & 1 & 2 & \textcolor{gray}{0} \\
    & & & & \smb_1 &
  \end{turingc}
  \]\normalsize

  Next, we apply $F_{0 \to 1}$. Based on the highlighted cells the fourth case of $M_\smb(0)$ we rewrite this as

  \vspace*{0.6em}
  \footnotesize
  \[ 
    w = \begin{turingc}
      \textcolor{gray}{1} & \textcolor{gray}{2} & 1 & 1 & 0 & \textcolor{gray}{0} \\
      & & \smd_2 & & &
      \CodeAfter \tikz{\draw [black,decorate,decoration = {brace,raise=0pt,amplitude=3pt,aspect=0.5}] ($(1-|4) + (0.05,0)$) -- ($(1-|6) + (-0.05,0)$) node[pos=0.5,above=0.5mm,black]{\footnotesize $\val_3(1) + f(0) + 1 = 3$};}
    \end{turingc}
  \]\normalsize
  
  Next, we apply $F_{1 \to 2}$. Based on the highlighted cells the fifth case of $M_\smd(1)$ applies and we rewrite this as

  \vspace*{0.6em}
  \footnotesize
  \[ 
    w = \begin{turingc}
      \textcolor{gray}{1} & \textcolor{gray}{2} & 0 & 1 & 2 & 0 \\
      & &  & & & \smq_1
      \CodeAfter
      \tikz{\draw [black,decorate,decoration = {brace,raise=0pt,amplitude=3pt,aspect=0.5}] ($(1-|3) + (0.05,0)$) -- ($(1-|6) + (-0.05,0)$) node[pos=0.5,above=0.5mm,black]{\footnotesize $\val_3(10) + 2 = 5$};}
    \end{turingc} 
  \]\normalsize
  
  Next, we apply $F_{2 \to 3}$. Based on the highlighted cells the first case of $M_\smq(2)$ applies and we rewrite this as

  \vspace*{0.6em}
  \footnotesize
  \[ 
    w = \begin{turingc}
      1 & \textcolor{gray}{2} & 2 & 0 & 1 & 2 \\
      \smq_1 & &  & & & 
      \CodeAfter
      \tikz{\draw [black,decorate,decoration = {brace,raise=0pt,amplitude=3pt,aspect=0.5}] ($(1-|3) + (0.05,0)$) -- ($(1-|7)$) node[pos=0.5,above=0.5mm,black]{\footnotesize $\val_3(012) + 2f(2) + 5 = 59$};}
    \end{turingc}
  \]\normalsize

  Next, we apply $F_{3 \to 4}$. Based on the highlighted cells the third case of $M_\smq(3)$ applies and we rewrite this as

  \vspace*{0.6em}
  \footnotesize
  \[ 
    w = \begin{turingc}
      0 & 2 & 2 & 2 & 0 & 1 \\
      & \smd_1 & & & & 
      \CodeAfter
      \tikz{\draw [draw=none] ($(1-|3) + (0.05,0)$) -- ($(1-|7) + (0.4,0)$) node[pos=0.5,above=0.5mm,black]{\footnotesize $\val_3(2012) + 2f(3) + 2 = 219$};}
      \tikz{\clip ($(1-|3) + (0,-0.3)$) rectangle ($(1-|7) + (0,0.6)$);
      \draw [black,decorate,decoration = {brace,raise=0pt,amplitude=3pt,aspect=0.5}] ($(1-|3) + (0.05,0)$) -- ($(1-|7) + (0.4,0)$);}
      \tikz{\clip ($(1-|1) + (0,-0.3)$) rectangle ($(1-|2) + (0,0.3)$);
      \draw [black,decorate,decoration = {brace,raise=0pt,amplitude=3pt,aspect=0.5}] ($(1-|1) + (-1.6,-0)$) -- ($(1-|2) + (-0.05,-0)$);}
    \end{turingc}
  \]\normalsize    
  
  Next, we apply $F_{4 \to 5}$. Based on the highlighted cells the fourth case of $M_\smd(4)$ applies and we rewrite this as

  \vspace*{0.6em}
  \footnotesize
  \[ 
    w = \begin{turingc}
      2 & 1 & 2 & 2 & 0 & 0 \\
      & \smb_2 & & & & 
      \CodeAfter
      \tikz{\draw [draw=none] ($(1-|2) + (0.05,0)$) -- ($(1-|7) + (0.4,0)$) node[pos=0.5,above=0.5mm,black]{\footnotesize $\val_3(22010) + f(4) + 1 = 461$};}
      \tikz{\clip ($(1-|2) + (0,-0.3)$) rectangle ($(1-|7) + (0,0.6)$);
      \draw [black,decorate,decoration = {brace,raise=0pt,amplitude=3pt,aspect=0.5}] ($(1-|2) + (0.05,0)$) -- ($(1-|7) + (0.4,0)$);}
      \tikz{\clip ($(1-|1) + (0,-0.3)$) rectangle ($(1-|2) + (0,0.3)$);
      \draw [black,decorate,decoration = {brace,raise=0pt,amplitude=3pt,aspect=0.5}] ($(1-|1) + (-1.6,-0)$) -- ($(1-|2) + (-0.05,-0)$);}
    \end{turingc} 
  \]\normalsize    
  
  Finally, we apply $F_{\final}$ (in Figure~\ref{fig:bottom-up-encoding3}). We are the fourth case, where we don't touch the counter and just change the state to $\smp_1$:

  \footnotesize
  \[ 
    w = \begin{turingc}
      2 & 1 & 2 & 2 & 0 & 0 \\
      & \smp_1 & & & & 
    \end{turingc} 
  \]\normalsize
  
 This gives the expected result: as $v_3(122002) = 461$, and the head is shifted once to the right in the tape, this configuration has position $1 \cdot (2 \cdot 3^\ell) + 461 = 1919$ in the orbit of the initial configuration corresponding to state $\smp_1$. \qee
\end{example}

\section{Finitary distortion for SMART}\label{sec:Finitary}
In this section, we first introduce the group $\tmgroup_{\ell}$ generated by Turing machines instructions (Section~\ref{sec:group-turing-machines-finitary-instructions}), and we slightly alter the SMART machine $\smart$ (Section~\ref{sec:decorated-smart}): we call it the \emph{decorated SMART} since we add some additional components to its states. The action of the SMART machine extends trivially to the new decorations -- new components are added as a Cartesian product, and the head simply carries its new components without modifying or reading them.

Denoting $T_{\ell,\smart} : \ctape_\ell \to \ctape_\ell$ the finite action of this decorated SMART on the cyclic tapes of $\ctape_\ell$, we then establish Lemma~\ref{lem:smart-distorted-on-cyclic-tapes}, an intermediary result about $T_{\ell,\smart}$. Namely: this automorphism is ``finitarily distorted'' in $\tmgroup_\ell$, in the sense that all its powers (including ones exponential in $\ell$) have word norm polynomial in $\ell$ under the fixed generators (which is exponentially lower than the order of the group would suggest). The rest of this section is dedicated to the proof of this lemma.

\medskip
Later in Section~\ref{sec:on-full-shift}, we prove that every non-trivial full shift contains a distortion element of infinite order. Technically, Section~\ref{sec:on-full-shift} focuses on transporting the finite actions $T_{\ell,\smart}$ of the decorated SMART into an infinite action on a full shift (and showing that this transposition preserves distortion). In other words, the ``distortion'' aspect of Lemma~\ref{lem:nice-implies-distortion-in-full-shift} entirely comes from Lemma~\ref{lem:smart-distorted-on-cyclic-tapes}, i.e.\ from the main result of this section.

\subsection{Context and results}

\subsubsection{Decorated SMART on \texorpdfstring{$\ctape_\ell$}{C\_{}l}}\label{sec:decorated-smart}

Let $\smart = (\origsmart{Q},\origsmart{\Gamma},\origsmart{\Delta})$ be the SMART machine (see Section~\ref{sec:smart}). Define the \emph{decorated version} of the SMART machine as $\smart_\dec = (Q,\Gamma,\Delta)$, with
\begin{align*}
  Q = &\ \origsmart{Q} \times \duckset \times \ghostset \\
  \Gamma = &\ \origsmart{\Gamma} \\
  \Delta =  & \bigcup_{(d,x) \in \duckset \times \ghostset} \left\{ \Big( (q,d,x),a,(q',d,x),b \Big) : (q,a,q',b) \in \origsmart{\Delta} \right\} \\
  & \cup \bigcup_{(d,x) \in \duckset \times \ghostset} \left\{ \Big((q,d,x),\delta,(q',d,x)\Big) : (q,\delta,q') \in \origsmart{\Delta} \right\}
\end{align*}
for $\duckset = \{\dright,\dleft\}$ and $\ghostset = \llbracket 0, 5 \rrbracket$, i.e.\ the states of SMART now carry a state $\origsmart{q} \in \origsmart{Q}$ of the original machine $\smart$, a special symbol $d \in \duckset$ called the \emph{duck}, and a \emph{ghost symbol} $x \in \ghostset$.  We have $|Q| = 96$.

Technically, the two SMART machines $\smart_\dec$ and $\smart$ are different: for one, they act on different sets of cyclic tapes (since they have different sets of states). However, they have very similar behaviors, as the decorated machine only carries its decoration unmodified in its state while acting on tapes. To refer to the original set of states of SMART, we will use $\origsmart{Q}$, and $Q$ will denote $Q = \origsmart{Q} \times \duckset \times \ghostset$.

Since the remainder of this article only uses the decorated version of SMART, in what follows $\smart$ will refer to the decorated version of the machine, despite lacking the ``$\dec$'' subscript. This should not cause confusion, as the machines act on different sets.

\medskip
The point of the \emph{ghost} $\llbracket 0, 5 \rrbracket$ is to allow us to condition the application of gates, and to build the permutations we perform in Section~\ref{sec:engineering-in-cyclic-tapes}. The \emph{duck} $d \in \{\dright,\dleft\}$ will be important during intermediate steps of computation in Section~\ref{sec:engineering-in-cyclic-tapes}, in order to realize piecewise defined functions.

\medskip
For $S \subseteq \ctape_\ell$ a subset of finite cyclic tapes, we denote $S[\dright]$ and $S[\dleft]$ the subsets of the tapes of $S$ containing a head, and whose ducks respectively are $d = \dright$ and $d = \dleft$. For $d \in \{\dright,\dleft\}$ and a function $f : S \to S$, we abuse notations and denote $\rduck{f}{d}$ the extended restriction $f\rfun_{S[d]}$ of $f$ to $S[d]$:
\[
  \rduck{f}{d}(w) = \begin{cases}
    f(w) & \text{if } w \in S[d] \\
    w & \text{otherwise }
  \end{cases}
\]

\subsubsection{Group of Turing machine instructions on finite cyclic tapes}\label{sec:group-turing-machines-finitary-instructions}

Recall that $\ctape_\ell$ is the set of finite cyclic tapes of length $\ell \geq 2$ (see Section~\ref{sec:defs}) with states $Q$ and tape-alphabet $\Gamma$, containing at most one head (i.e.\ a letter in $Q \times \Gamma)$. 

\medskip
Let $\tmgroup_\ell$ be the finitely generated subgroup of $\Sym(\ctape_\ell)$ generated by state-dependent moves, and the unary gates permuting heads. Formally, for $g \in \Sym(Q \times \Gamma)$, define the \emph{unary gate} $\pi_g \in \Sym(\ctape_\ell)$ as
\begin{align*}
  \pi_g(w)_j = \begin{cases}
    g(w_j) & \text{if } w_j \in (Q \times \Gamma) \\
    w_j & \text{if } w_j \in \Gamma
  \end{cases}
\end{align*}
for a cyclic tape $w \in \ctape_\ell$. In addition, for $q \in Q$, define the \emph{state-dependent right move} $\rho_q \in \Sym(\ctape_\ell)$ as:
\[ \rho_q(w)_j = \begin{cases}
  (q,w_j) & \text{if } w_{j-1 \bmod \ell} \in (\{q\} \times \Gamma) \\
  \pi_\Gamma(w_j) & \text{if } w_{j-1 \bmod \ell} \notin (\{q\} \times \Gamma) \wedge w_j \in (\{q\} \times \Gamma) \cup \Gamma  \\
  w_j & \text{if } w_j \in ((Q \setminus \{q\}) \times \Gamma)
  \end{cases} \]
for $w \in \ctape_\ell$ and $\pi_\Gamma : \Gamma \cup (Q \times \Gamma) \to \Gamma$ the natural projection.

\medskip
We then define the group $\tmgroup_\ell \leq \Sym(\ctape_\ell)$ generated by these permutations:
\[ \tmgroup_\ell = \langle \{\pi_g : g \in \Sym(Q \times \Gamma) \} \cup \{ \rho_q \mid q \in Q\} \rangle \]

We can see the group $\tmgroup_\ell$ as the group generated by the instructions of Turing-machines: moving heads based on their states, or permuting their values. To ease notations, we denote $\prod_{q \in Q} \rho_q$ by $\rho$. This (finite) group is equipped with a metric, that is the word norm given by the generators $\pi_g$ and $\rho_q$.

\medskip
It is easy to see that for any reversible Turing machine $\mathcal{M}$ of states $Q$ and tape-alphabet $\Gamma$, $T_{\ell,\mathcal{M}}$ is an element of $\tmgroup_\ell$. Indeed, a step of computation is the composition of a head permutation $\alpha$ of $Q \times \Gamma$, followed with state-dependent moves $\beta_{+1}$ and $\beta_{-1}$:
\begin{align*}
  \alpha(q,a) & = \begin{cases}
    (q',b) & \text{if } (q,a,q',b) \in \Delta \\
    (q',a) & \text{if } (q,\pm 1, q') \in \Delta
  \end{cases} \\
  \beta_{+1} & = \prod_{q' \mid \exists q, (q,+1,q') \in \Delta} \rho_{q'}  \\
  \beta_{-1} & = \prod_{q' \mid \exists q, (q,-1,q') \in \Delta} {\rho_{q'}}^{-1}
\end{align*}

Finally, we denote by $\delta(\ell,n)$ the word norm of $(T_{\ell,\mathcal{M}})^n$ in $\tmgroup_\ell$. In this chapter, we focus on proving that $\delta(\ell,n)$ is polynomial in $\ell$ for powers of the decorated SMART machine (even for powers exponential in $\ell$).

\begin{remark}
\label{rem:CharacterizationOfGlQGamma}
It can be shown that $\tmgroup_\ell$ is, for large enough $\ell$, $|Q|$ and $|\Gamma|$ (in particular for all versions of the SMART machine we consider and for $\ell \geq 2$), of bounded index in the automorphism group of $\ctape_\ell$ under the shift action of $\Z_\ell$. This is not particularly useful, however, as what we need it for is to provide a group where the SMART machine corresponds to an element of small word norm (far smaller than the radius of the group).
\end{remark}

\subsubsection{Main result: finitary distortion of the decorated SMART}

Recall that $m: \N \to \N$ is the \emph{movement function}, i.e.\ $m(n)$ is the maximal number of cells the machine $\smart$ can visit in $n$ steps; and that $\delta(\ell,n)$ is the word norm of $(T_{\ell,\smart})^n$ in $\tmgroup_\ell$.

\begin{lemma}\label{lem:smart-distorted-on-cyclic-tapes}
  Let $\smart$ be the (decorated) SMART machine.
  \begin{enumerate}
    \item $f_{\smart}$ has infinite order.
    \item There exist some $C,C' > 0$ such that $m(n) \leq C \log n + C'$.
    \item There exists some $p > 0$ such that $\delta(\ell,n) = O(\ell^p)$.
  \end{enumerate}
  In fact, for $\smart$, we can take $C = \ln(2)/\ln(3)$ and $p = 4$.
\end{lemma}

Any finite order $T$ satisfies the latter two items, and any non-trivial state-dependent shift satisfies the first and the third items. Achieving the first two items is already difficult, and to our knowledge these properties have only been explicitly shown (in the reversible case) for the SMART machine and the binary SMART machine~\cite{2020-CT}. We expect that the Kari-Ollinger construction in~\cite{2008-KO} can be used to produce more examples of machines satisfying these two properties (at least $m(n) = O(n / \log n)$ follows from general principles for all these machines~\cite{2017-GS}).

\subsubsection{Proof of Lemma \ref{lem:smart-distorted-on-cyclic-tapes}}

For the second item, the logarithmic speed of SMART is well known. To sketch a proof, consider the following computation: after less than 18 steps, the head of SMART is in state $\smp_1$ or $\smq_1$ reading a $0$ (ignoring the ghost and the duck). Then SMART is either at the left (for $\smp_1$) or right (for $\smq_1$) extremity of some word $0^m$ for some $m \geq \log_3(k)+2$, or by~\cite[Lemma 4]{2017-COT} it builds around this position some pattern in the set $C_m$ for $m \geq \log_3(k)+2$ (with the notations of~\cite[Lemma 4]{2017-COT}). Either way, after this point, $k$ steps of computations cannot read more than $\log_3(k)+2$ different cells.

\medskip
The proof of the third item is a matter of programming powers of the machine efficiently with the generators of $\tmgroup_\ell$, which can be considered as the primitive reversible instructions of Turing machines. To achieve this, we encode configurations into their orbit position with the automorphism $\encoding_\ell$ (defined in Section~\ref{sec:smart-encoding-overall}), perform an addition on these positions (as defined below), and decode back, as summarized in the commuting diagram below:

\begin{center}
  \begin{tikzpicture}
    \node[draw, text width=2.5cm, align=center] (A) at (0,0) {$(T_{\ell,\smart})^{n_0}(C_q)$};
    \node[draw, text width=2.5cm, align=center] (B) at (5.5,0) {$(T_{\ell,\smart})^{n_0+\textcolor{red}{n}}(C_q)$};
    \node[draw, text width=3.5cm, align=center] (B') at (0,-2) {$(q,n_0)$};
    \node[draw, text width=3.5cm, align=center] (B'') at (5.5,-2) {$(q,n_0 + \textcolor{red}{n} \bmod 2\ell \cdot 3^\ell)$};
    \draw[-to] (A) -- (B) node[midway,above] {$(T_{\ell,\smart})^{\textcolor{red}{n}}$};
    \draw[-to] (A) -- (B') node[midway,left] {$\encoding_\ell$};
    \draw[-to] (B'') -- (B) node[midway,right] {$\encoding_\ell^{-1}$};
    \draw[-to] (B') -- (B'') node[midway,below] {$+\textcolor{red}{n}$};
  \end{tikzpicture}
\end{center}

More precisely, we prove the two following lemmas:
\begin{restatable}{lemma}{encodingwithgarbage}\label{lem:encoding-with-garbage}
  Let $\encoding_\ell : \ctape_\ell \to \ctape_\ell$ the encoding map defined in Section~\ref{sec:smart-encoding-overall}. There exists some $\GRencoding_{\ell} : \ctape_\ell \to \ctape_\ell$ in $\tmgroup_\ell$ with word norm~$O(\ell^4)$ such that
  \[ w \in \ctape_\ell[\dright] \implies \GRencoding_\ell(w) = \encoding_\ell(w) \]
  where $\ctape_\ell[\dright]$ is the set of cyclic tapes $w \in \ctape_\ell$ having a head with duck $d = \dright$.
\end{restatable}

Note that we say nothing about the action of $\GRencoding_\ell$ on tapes having duck $d = \dleft$. This restriction comes from our use of the \emph{ducking trick} to build $\GRencoding_\ell$, which produces ``garbage'' (i.e.\ acts with no reasonable interpretation) on heads having duck $d = \dleft$. See Section~\ref{sec:engineering-ducking-trick} for more details.

\begin{restatable}{lemma}{additioninencodedbase}\label{lem:addition-in-encoded-base}
  Let $(\plusorbit{n})_\ell$ be the bijection of $\ctape_\ell$ that performs the addition of $n \in \N$ in base $2\ell \cdot 3^\ell$ on the orbits positions encoded by $\encoding_\ell$. Recall that $(\rduck{\plusorbit{n}}{\dright})_\ell : \ctape_\ell \to \ctape_\ell$ is defined as
  \[ (\rduck{\plusorbit{n}}{\dright})_\ell(w) = \begin{cases} (\plusorbit{n})_\ell(w) & \text{if } w \in \ctape_\ell[\dright] \\
    w & \text{otherwise} \end{cases} \]

  Then $(\rduck{\plusorbit{n}}{\dright})_\ell$ belongs to $\tmgroup_\ell$ with word norm $O(\ell^3)$.
\end{restatable}

We give a more detailed definition of $(\plusorbit{n})_\ell$ in Section~\ref{sec:proofs-additition-encoded-base}. Informally, $(\plusorbit{n})_\ell$ adds $n$ to the counter $\encoding_\ell(w)$ that encodes the orbit position of the SMART configuration $w$.

\medskip
Then, these two lemmas are enough to prove the third item of Lemma~\ref{lem:smart-distorted-on-cyclic-tapes} about the decorated SMART being distorted on cyclic tapes of length $\ell$ (more precisely, $\delta(\ell,n) = O(\ell^4)$). Indeed, combining the two previous results, we obtain:

\begin{lemma}\label{lem:smart-one-duck}
  Denoting again:
  \[ \rduck{T_{\ell,\smart}}{\dright}(w) =
  \begin{cases}
      T_{\ell,\smart}(w) & \text{if } w \in \ctape_\ell[\dright] \\
      w & \text{otherwise}
  \end{cases} \]
  Then for any $n \in \N$, $(\rduck{T_{\ell,\smart}}{\dright})^n$ belongs to $\tmgroup_\ell$ with word norm $O(\ell^4)$.
\end{lemma}
\begin{proof}
  Let $(\rduck{\plusorbit{n}}{\dright})_\ell$ be given by Lemma~\ref{lem:addition-in-encoded-base}. With Lemma~\ref{lem:encoding-with-garbage}, one can conjugate $(\rduck{\plusorbit{n}}{\dright})_\ell$ with $\GRencoding_\ell$ and obtain a bijection on $\ctape_\ell$ that maps configurations of $\ctape_\ell[\dright]$ to their $n$-th iterate by $\smart$, and is the identity on $\ctape_\ell[\dleft]$. In other words:

  \[ \left(\rduck{T_{\ell,\smart}}{\dright}\right)^n = \left((\rduck{\plusorbit{n}}{\dright})_\ell\right)^{\GRencoding_\ell} \]

  Indeed, the addition of $(\rduck{\plusorbit{n}}{\dright})_\ell$ is only performed on heads having duck $\dright$, so the garbage generated by $\GRencoding_\ell$ on ducks $d=\dleft$ is canceled in the conjugation; additionally, the shift of the tape (which happens when the addition modulo $2 \cdot 3^\ell$ overflows) is performed to the right (resp. to the left), exactly like $(T_{\ell,\smart})^{2 \cdot 3^\ell}$ acts on configurations $C_\smb$ and $C_\smp$ (resp. $C_\smd$ and $C_\smq$).
\end{proof}

And this lemma then leads to:
\begin{proof}[Proof of Lemma~\ref{lem:smart-distorted-on-cyclic-tapes}, 3\textsuperscript{rd} item]
  Let $d' = \dright \leftrightarrow \dleft \in \Sym(\duckset)$ be the involution that swaps ducks $\dright$ and $\dleft$, and $d = \ID \times d' \times \ID \times \ID \in \Sym(H)$ its lift to $H = Q \times \Gamma$ (where $Q = \origsmart{Q} \times \duckset \times \ghostset$). We have:
  \[ \left(T_{\ell,\smart}\right)^n = \left(\pi_d \circ \left(\rduck{T_{\ell,\smart}}{\dright}\right)^n \circ \pi_d \right) \circ \left(\rduck{T_{\ell,\smart}}{\dright}\right)^n \]
  because $(\rduck{T_{\ell,\smart}}{\dright})^{\pi_d} =  \rduck{T_{\ell,\smart}}{\dleft}$. 
\end{proof}

\subsubsection{Overview}

The following subsections deal with the proofs of Lemmas~\ref{lem:encoding-with-garbage} and \ref{lem:addition-in-encoded-base}, which respectively prove that the encoding and the addition can be implemented in $\tmgroup_\ell$ with respective word norms $O(\ell^4)$ and $O(\ell^3)$.

More precisely, Section~\ref{sec:engineering-in-cyclic-tapes} contains our main technical results, which define and implement conditional permutations. It also contains an exposition of the \emph{ducking trick}, a method we use to implement piecewise-defined bijections. Section~\ref{sec:proofs-cyclic-tapes} contains the proofs of Lemmas~\ref{lem:encoding-with-garbage} and \ref{lem:addition-in-encoded-base}.

\subsection{Permutation engineering in \texorpdfstring{$\ctape_\ell$}{C\_{}\{l,Q-dec,Gamma\}}}\label{sec:engineering-in-cyclic-tapes}

In this section, we develop two methods. In Section~\ref{sec:engineering-perm-condition}, we consider \emph{permutation conditioning}, which consists in building permutations $\pi_{g,C}$ (for some $g \in \Sym(Q \times \Gamma)$ a permutation of the head, and some $C \subseteq (Q \times \Gamma) \times \Gamma^{\ell-1}$) that apply the permutation $g$ if the condition $C$ holds. We prove (Lemma~\ref{lem:nc1-conditioning}) that if $C$ has a simple enough description, then $\pi_{g,C}$ has polynomial word norm in $\ell$. In Section~\ref{sec:engineering-ducking-trick} we consider the \emph{ducking trick}, which allows us to efficiently build piecewise-defined bijections.

\subsubsection{Permutation conditioning}\label{sec:engineering-perm-condition}

We call \emph{conditions} the subsets $C$ of $(Q \times \Gamma) \times \Gamma^{\ell-1}$. For $g \in \Sym(Q \times \Gamma)$, if $C$ (considered as a subset of $\ctape_\ell$) is a $\pi_g$-invariant subset, we can define the bijection $\pi_{g,C} : \ctape_\ell \to \ctape_\ell$ as $\pi_{g,C}(w) = w$ if $w \in \ctape_\ell$ contains no head; and if $w \in \ctape_\ell$ contains a head a position, say, $i_0 \in \Z/\ell\Z$, we set $\pi_{g,C}(w)$ to:
\[ \pi_{g,C}(w)_i = \begin{cases}
  w_i & \text{if } i \neq i_0 \\
  g(w_i) & \text{if } i=i_0 \text{ and } w_i \cdot w_{i+1} \dots w_{\ell-1} \cdot w_0 \dots w_{i-1} \in C \\
  w_i & \text{if } i=i_0 \text{ and } w_i \cdot w_{i+1} \dots w_{\ell-1} \cdot w_0 \dots w_{i-1} \notin C
\end{cases} \]
We then call $\pi_{g,C}$ the \emph{conditional application of $g$ under condition $C$}.

\medskip
Recall that the states of $Q$ have a ghost component $\ghostset$. We split $Q$ into two components: $Q = \wgh{Q} \times \ghostset$ (so, with the notation $Q = \origsmart{Q} \times \duckset \times \ghostset$, we have $\wgh{Q} = \origsmart{Q} \times \duckset$; but the exact structure of $\wgh{Q}$ has no importance in this section).

Then, the set $H = Q \times \Gamma$ also splits into $H = \wgh{H} \times \ghostset$, where $\wgh{H} = \wgh{Q} \times \Gamma$. Note that, for any permutation $g \in \Sym(\ghostset)$, the permutation $\ID_{\wgh{H}} \times g$ belongs to $\Sym(H)$; and similarly, if $g \in \Sym(\wgh{H})$, the permutation $g \times \ID_{\ghostset}$ belongs to $\Sym(H)$. Finally, all conditions we define below will be of the form $C = \wgh{C} \times \ghostset$, for $\wgh{C} \subseteq \wgh{H} \times \Gamma^{\ell-1}$.

\medskip
In this section, we prove Lemma~\ref{lem:gate-conditioning}: for gates $g \in \Sym(\wgh{H})$ and $\pi_{g \times \ID}$-invariant conditions $C = \wgh{C} \times \ghostset$ (for some $\wgh{C} \subseteq \wgh{H} \times \Gamma^{\ell-1}$), the conditioned gates $\pi_{g \times \ID, C}$ belong to $\tmgroup_\ell$. We also provide upper bounds on the word norm of $\pi_{g \times \ID, C}$ depending on $C$. This is essentially Barrington's theorem~\cite{1989-Barrington}.

\bigskip
As a first step, we consider the opposite case: instead of leaving the ghost-component of the head intact while permuting, we consider permutations that only permute the ghost-component $\ghostset$. As conditions $C = \wgh{C} \times \ghostset$ (for $\wgh{C} \subseteq \wgh{H} \times \Gamma^{\ell-1}$) are trivially $\pi_{\ID \times g}$-invariant for any $g \in \Sym(\ghostset)$, the conditioned gates $\pi_{\ID \times g, C}$ are always defined and:

\begin{lemma}\label{lem:ghost-conditioning}
  For any $g \in \Alt(\ghostset)$ and condition $C = \wgh{C} \times \ghostset$ (for some $\wgh{C} \subseteq \wgh{H} \times \Gamma^{\ell-1}$), the conditioned gate $\pi_{\ID \times g, C}$ belongs to $\tmgroup_\ell$. Let $T : \mathcal{P}(H \times \Gamma^{\ell-1}) \to \N$ be the optimal function such that $\lVert \pi_{\ID \times g,C} \rVert \leq T(C)$ for all $g \in \Alt(\ghostset)$. Then $T$ satifies the following inequalities: 
  \begin{align*}
    T([a]_j \times \ghostset) & \leq |\min(j, \ell-j)| \\
    T((\wgh{C} \cap \wgh{C'}) \times \ghostset) & \leq 2 \Big(T(\wgh{C} \times \ghostset) + T(\wgh{C'} \times \ghostset) \Big) \\
    T((\wgh{C} \cup \wgh{C'}) \times \ghostset) & \leq \begin{cases}
      T(\wgh{C} \times \ghostset) + T(\wgh{C'} \times \ghostset) & \text{if } \wgh{C} \cap \wgh{C'} = \emptyset \\
      2 \Big(T(\wgh{C} \times \ghostset) + T(\wgh{C'} \times \ghostset) \Big) + 5 & \text{otherwise}
    \end{cases} \\
    T((\wgh{C}^c) \times \ghostset) & \leq T(\wgh{C} \times \ghostset) + 1
  \end{align*}
\end{lemma}

\begin{proof}
  We prove by induction over $\wgh{C} \subseteq \wgh{H} \times \Gamma^{\ell-1}$ that, denoting $C = \wgh{C} \times \ghostset$, every $\pi_{\ID \times g,C}$ (for $g \in \Alt(\ghostset)$) has word norm that checks the aforementioned inequalities.

  \begin{enumerate}[label=\textbf{Case \arabic*.}]
    \item If $\wgh{C} = \wgh{B} \times \Gamma^{\ell-1}$ for some $\wgh{B} \subseteq \wgh{H}$, then any such $\pi_{\ID \times g,\wgh{B} \times \ghostset}$ already appears in the set of generators of $\tmgroup_\ell$. 
    \item If $\wgh{C} = [a]_j \triangleq \wgh{H} \times \Gamma^{j-1} \times \{a\} \times \Gamma^{\ell-j-1}$ for some $j \in \llbracket 1, \ell \rrbracket$ and $a \in \Gamma$, define $\wgh{B} = (\wgh{Q} \times \{a\}) \times \Gamma^{\ell-1}$. Then one can conjugate $\pi_{\ID \times g,\wgh{B} \times \ghostset}$ (which belongs to $\tmgroup_\ell$ by the first item) with either the right-move $\rho^{j}$ or the left-move $\rho^{-(\ell-j)}$: the resulting permutation applies $g$ on the ghost symbol if and only if $j$ cells away from the head, the content of the tape is $a$.
    \item If $\wgh{C} = (\wgh{C_1})^c$, then $\pi_{\ID \times g,\wgh{C} \times \ghostset} = \pi_{\ID \times g^{-1},\wgh{C_1} \times \ghostset} \circ \pi_{\ID \times g}$.
    \item If $\wgh{C} = \wgh{C_1} \cap \wgh{C_2}$, we use the ``commutator trick'': as $\ghostset$ has cardinality at least $5$, $g$ is a commutator by Ore's theorem~\cite[Theorem 7]{1951-Ore}, so there exist $g_1,g_2$ such that $g = [g_1,g_2]$. By the induction hypothesis and a straightforward calculation, we conclude that:
    \[ \pi_{\ID \times g, (\wgh{C_1} \cap \wgh{C_2}) \times \ghostset} = \Big[\pi_{\ID \times g_1,\wgh{C_1} \times \ghostset}, \pi_{\ID \times g_2, \wgh{C_2} \times \ghostset}\Big]\]
    \item If $C = (\wgh{C_1} \cup \wgh{C_2}) \times \ghostset$ with $\wgh{C_1} \cap \wgh{C_2} = \emptyset$, then 
    \[ \pi_{\ID \times g,(\wgh{C_1} \cup \wgh{C_2}) \times \ghostset} = \pi_{\ID \times g,\wgh{C_1} \times \ghostset} \circ \pi_{\ID \times g,\wgh{C_2} \times \ghostset}. \]
    \item If $C = (\wgh{C_1} \cup \wgh{C_2}) \times \ghostset$, then
    \[ \pi_{\ID \times g,(\wgh{C_1} \cup \wgh{C_2}) \times \ghostset} = \pi_{\ID \times g,(\wgh{C_1}^c \cap \wgh{C_2}^c)^c \times \ghostset}. \]
  \end{enumerate}

  We conclude that $\pi_{\ID \times g,C} \in \tmgroup_\ell$, and that the provided upper-bounds are correct.
\end{proof}

Note that for any $g \in \Sym(\wgh{H})$, $g \times \ID_{\ghostset} \in \Sym(H)$ is an even permutation since $|\ghostset|$ is even. Combining this with the previous lemma, we obtain the following result which allows the conditioning of gates depending on conditions of the form $\wgh{C} \times \ghostset$, and controls their word norm in $\tmgroup_\ell$:

\begin{lemma}\label{lem:gate-conditioning}
  Let $T : \mathcal{P}(H \times \Gamma^{\ell-1}) \to \N$ be given by the previous lemma. Assume that $g \in \Sym(\wgh{H})$ is a permutation and $C = \wgh{C} \times \ghostset$ is a $\pi_{g \times \ID}$-invariant condition, for some $\wgh{C} \subseteq \wgh{H} \times \Gamma^{\ell-1}$. Then the permutation $\pi_{g \times \ID,C}$ belongs to $\tmgroup_\ell$ with word norm $O(T(C))$.
\end{lemma}

We divide the proof of this lemma in three parts.

\begin{claim}
  For $g \in \Alt(\wgh{H} \times \ghostset)$ a 3-cycle and $C = \wgh{C} \times \ghostset$ a $\pi_{g}$-invariant condition, the permutation $\pi_{g, C}$ belongs to $\tmgroup_\ell$ with word norm $O(T(C))$.
\end{claim}
\begin{proof}
  Let $g = ((\wghintext{h_1},x_1),(\wghintext{h_2},x_2), (\wghintext{h_3},x_3))$ be a 3-cycle of $\Sym(\wgh{H} \times \ghostset)$ and $C = \wgh{C} \times \ghostset$ be a $\pi_g$-invariant condition for some $\wgh{C} \subseteq \wgh{H} \times \Gamma^{\ell-1}$.

  Consider the condition $B = \wgh{B} \times \ghostset$ where $\wgh{B} = \wgh{C} \cap [\wghintext{h_1}]_0$, and for $y_1,y_2,y_3 \in \ghostset$ three distinct elements, let us define the following permutations:
  \begin{align*}
    g'_{\ghostset} & = (y_1,y_2,y_3) \in \Alt(\ghostset) \\
    g' & = ((\wghintext{h_1},y_1),(\wghintext{h_1},y_2), (\wghintext{h_1},y_3)) \in \Alt(\wgh{H} \times \ghostset)
  \end{align*}

  By Lemma~\ref{lem:gate-conditioning}, $\pi_{\ID \times g'_{\ghostset}, B}$ belongs to $\tmgroup_\ell$ with word norm $O(T(B)) = O(T(C))$. But $\pi_{\ID \times g'_{\ghostset},B} = \pi_{g',C}$, since $\wgh{B} = \wgh{C} \cap [\wghintext{h_1}]_0$. This proves that $\pi_{g,C}$ belongs to $\tmgroup_\ell$ with word norm $O(T(C))$, since $g$ and $g'$ (hence $\pi_{g,C}$ and $\pi_{g',C}$) are conjugated by the involutions $(\wghintext{h_i},x_i) \leftrightarrow (\wghintext{h_1},y_i)$ of $\Sym(H)$.
\end{proof}

\begin{claim}
  For $g \in \Sym(\wgh{H})$ a cycle of support $S$ and $C = \wgh{C} \times \ghostset$ a $\pi_{g \times \ID}$-invariant condition, the condition $C$ is $\pi_{g'}$-invariant for every $g' \in \Sym(S \times \ghostset)$.
\end{claim}
\begin{proof}
  Let $g' \in \Sym(S \times \ghostset)$, $(\wghintext{h},x) \in \wgh{H} \times \ghostset$ and $\gamma \in \Gamma^{\ell-1}$. Let us denote $g'(\wghintext{h},x) = (\wghintext{h}',x')$. As $g$ is a cycle, there exists $k$ such that $g^k(\wghintext{h}) = \wghintext{h}'$. Then $\pi_{g'}((\wghintext{h},x)\cdot \gamma) = g'(\wghintext{h},x) \cdot \gamma = (g^k(\wghintext{h}),x') \cdot \gamma$, so that:
  \begin{align*}
    & \pi_{g'}((\wghintext{h},x) \cdot \gamma) \in \wgh{C} \times \ghostset \\
    \iff \quad & (g^k(\wghintext{h}),x') \cdot \gamma \in \wgh{C} \times \ghostset \\
    \iff \quad & \pi_{g}^k((\wghintext{h},x') \cdot \gamma) \in \wgh{C} \times \ghostset \\
    \iff \quad & (\wghintext{h},x') \cdot \gamma \in \wgh{C} \times \ghostset \qquad \text{as $\wgh{C} \times \ghostset$ is $\pi_g$-invariant}\\
    \iff \quad & (\wghintext{h},x) \cdot \gamma \in \wgh{C} \times \ghostset \qedhere
  \end{align*}
\end{proof}

We can now conclude the proof:

\begin{proof}[Proof of Lemma~\ref{lem:gate-conditioning}]
  Let $g \in \Sym(\wgh{H})$ and $C = \wgh{C} \times \ghostset$  be a $\pi_{g \times \ID}$-invariant condition for some $\wgh{C} \subseteq \wgh{H} \times \Gamma^{\ell-1}$. Without loss of generality, we can assume that $g$ is a cycle, whose support we denote $S$.

  Then $g \times \ID$ belongs to $\Alt(S \times \ghostset)$, since $\ghostset$ has even cardinality. Additionally, $\Alt(S \times \ghostset)$ is generated by its $3$-cycles since $|S \times \ghostset| \geq 3$. Let us write $g \times \ID = c_1 \circ \dots \circ c_k$ for $c_1,\dots,c_k$ 3-cycles of $\Alt(S \times \ghostset)$. Then $C$ is $\pi_{c_i}$-invariant for every $c_i$ by the second claim, so by the first claim each $\pi_{c_i,C}$ belongs to $\tmgroup_\ell$ with word norm $O(T(C))$.

  As $\Alt(S \times \ghostset) \leq \Alt(H)$ is finite, $k$ is bounded (with bound independent from $\ell$), and $\pi_{g \times \ID, C}$ is the composition of this bounded number of $\pi_{c_i,C}$. This concludes the proof.
\end{proof}

\medskip
We finally state an (optional) rephrasing of Lemma~\ref{lem:gate-conditioning}. Readers with a background in complexity theory will find the following version of the statement useful; it is immediate from the definition of the complexity class $\NC^1$.

\begin{lemma}\label{lem:nc1-conditioning}
  Let $L \subseteq \wgh{H} \times \Gamma^*$ be a language in $\NC^1$, and $L_n \subseteq \wgh{H} \times \Gamma^{n-1}$ its words of size $n$. For any $g \in \Alt(\wgh{H})$, the conditioned gates $f_{g \times \ID, L_n \times \ghostset}$ belong to $\tmgroup_\ell$ with polynomial word norm in $n$.
\end{lemma}

\subsubsection{Examples}

To clarify Lemma~\ref{lem:gate-conditioning}, we consider several examples: the permutation of adjacent letters, lexicographic comparisons, and (cyclic) ternary additions.

\paragraph*{Permuting two adjacent tape-letters}

\begin{lemma}\label{lem:permuting-letters}
  Consider the permutation $p \in \Sym(\ctape_\ell)$ of two adjacent tape-letters around the head, i.e.\ the shift-equivariant action on patterns of size $2$:
  \[ p : \begin{tpattern} w_0 & w_1 \\ q & \end{tpattern} \mapsto \begin{tpattern} w_1 & w_0 \\ q & \end{tpattern}\]
  for any $w = w_0 w_1 \in \Gamma^2$ and $q \in Q$.

  Then $p$ belongs to $\tmgroup_\ell$ with constant word norm.
\end{lemma}

\begin{proof}
  The permutation $p$ is the composition of finitely many commuting $p_{(a,b)}$ that only permutes adjacent tape-letters when they differ and belong to the set $\{a,b\} \subseteq \Gamma$, i.e.\ for $a \neq b \in \Gamma$, the involution $p_{(a,b)} \in \Sym(\ctape_\ell)$ is defined as
  \[ \qquad \begin{tpatternc} a & b \\ q & \end{tpatternc} \leftrightarrow \begin{tpatternc} b & a \\ q & \end{tpatternc}\]
  for any $q \in Q$. So we only need to prove that every $p_{(a,b)}$ belong to $\tmgroup_\ell$.

  \smallskip
  Let $a \neq b \in \Gamma$. Define $g_\Gamma \in \Sym(\Gamma)$ the involution $g_\Gamma = a \leftrightarrow b$. Then $\wgh{g} = \ID_{\wgh{Q}} \times g_\Gamma$ is a permutation of $\Sym(\wgh{H})$, and $g = \wgh{g} \times \ID_{\ghostset}$ is a permutation of $\Sym(H)$. By Lemma~\ref{lem:gate-conditioning}, both $\pi_{g,[a]_1 \times \ghostset}$ and $\pi_{g,[a]_{-1} \times \ghostset}$ belong to $\tmgroup_\ell$ with word norm independent of $\ell$, and we have $p_{(a,b)} = \Pi$, where
  \[ \Pi = (\rho^{-1} \circ \pi_{g,[a]_{-1} \times \ghostset} \circ \rho) \circ (\pi_{g,[a]_1 \times \ghostset}) \circ (\rho^{-1} \circ \pi_{g,[a]_{-1} \times \ghostset} \circ \rho) \]
  and $\rho$ denotes the right rotation of the head. Indeed,
  for any $q \in Q$ we calculate
  \begin{align*}
    \Pi\left( \begin{tpatternc} a & b \\ q & \end{tpatternc} \right) & = \begin{tpatternc} b & a \\ q & \end{tpatternc} \\
    \Pi\left( \begin{tpatternc} b & a \\ q & \end{tpatternc} \right) & = \begin{tpatternc} a & b \\ q & \end{tpatternc} \\
    \Pi\left( \begin{tpatternc} a & a \\ q & \end{tpatternc} \right) & = \begin{tpatternc} a & a \\ q & \end{tpatternc}
  \end{align*}
  and nothing of interest (other than the head moving back and forth) happens on other patterns.
\end{proof}

\begin{remark}\label{rem:permuting-letters}
  The previous permutation $p$ can be conditioned on any value of the state $\wgh{q} \in \wgh{Q}$ by simply replacing the conditions $[a]_1 \times \ghostset$ (resp. $[a]_{-1} \times \ghostset$) with $([a]_1 \cap [\{\wgh{q}\} \times \Gamma]_0) \times \ghostset$ (resp. $([a]_{-1} \cap [\{\wgh{q}\} \times \Gamma]_0) \times \ghostset$).
\end{remark}

\paragraph*{Lexicographic comparisons}\label{sec:engineering-additions}

Let us consider the case of conditions that lexicographically compare the content of the tape with arbitrary fixed words (word equalities and comparisons). By reducing the movement of the head, we obtain better upper bounds than what a strict reading of Lemma~\ref{lem:ghost-conditioning} would suggest.

\begin{lemma}\label{lem:gate-ordering-condition}
  Given a lexicographic equality/comparison ${\sim} \in \{{=},{<},{\leq},{>},{\geq}\}$ and a word $c \in \{0,1,2\}^l$ for some $l \leq \ell$, consider a condition $C \subseteq H \times \Gamma^{\ell-1}$ of the form $\sim c$, i.e.
  \[ \begin{tpatternc} a_0 & \dots & a_{l-1} & a_l & \dots & a_{\ell-1} \\ q \end{tpatternc} \in C \quad \iff \quad a_0 \cdots a_{l-1} \sim c_0 \cdots c_{l-1}. \]

  Then $T(C) = O(|c|^2)$.
\end{lemma}

\begin{proof}
  To prove this result, we use a divide and conquer approach: given $l \leq \ell$, any integer $l' \leq l$, and two words $w \in \{0,1,2\}^l$ and $c \in \{0,1,2\}^l$,
  \begin{align*}
    w = c & \iff \Big( w_{\llbracket 0,l'-1 \rrbracket} = c_{\llbracket 0,l'-1 \rrbracket} \Big) \wedge \Big( w_{\llbracket l',l-1 \rrbracket} = c_{\llbracket l',l-1 \rrbracket} \Big) \\
    w < c & \iff \Big( w_{\llbracket 0,l'-1 \rrbracket} < c_{\llbracket 0,l'-1 \rrbracket} \Big) \\
    & \hspace*{1cm} \vee \Big( \Big( w_{\llbracket 0,l'-1 \rrbracket} = c_{\llbracket 0,l'-1 \rrbracket} \Big) \wedge \Big( w_{\llbracket l',l-1 \rrbracket} < c_{\llbracket l',l-1 \rrbracket} \Big) \Big) 
  \end{align*}
  So if $C$ is the condition $= c$, and $g_1,g_2 \in \Alt(\ghostset)$, then:
  \[ \pi_{\ID \times [g_1,g_2],=c} = \Big [ \pi_{\ID \times g_1,= c_{\llbracket 0,l'-1 \rrbracket}}, \; \rho^{-l'} \circ \pi_{\ID \times g_2,= c_{\llbracket l',l-1 \rrbracket}} \circ \rho^{l'} \Big ]\]
  Taking $l' = \lfloor l/2 \rfloor$ and iterating, we obtain $T(C) = O(l^2) = O(|c|^2)$. Similarly, if $C$ is the condition $< c$, we obtain $T(C) = O(|c|^2)$. The other cases follow using elementary Boolean algebra on the operators of $\{{=},{<},{\leq},{>},{\geq}\}$.
\end{proof}

\begin{remark}
\label{rem:Wildcards}
  The previous lemma can still be used to compare words $w$ and $c$ at non-contiguous positions. Denote $\bullet$ a wildcard symbol that represents positions which will be ignored when performing the comparison $w \sim c$, i.e.\ for any two words $w_0 \cdots w_{l-1} \in \{0,1,2\}^l$ and $c_0 \dots c_{l-1} \in (\{0,1,2\} \cup \{\bullet\})^l$, if $i_0 < \dots < i_n$ are the non-wildcard positions in $c$ (i.e.\ $c_i \neq \bullet$ if and only if $i \in \{i_0,\dots,i_n\}$), we say that $w \sim c$ if the usual lexicographic comparison $w_{i_0} \cdots w_{i_n} \sim c_{i_0} \cdots c_{i_n}$ is true.

  Then the previous lemma still holds for comparisons $w \sim c$ for $w \in \{0,1,2\}^l$ and $c \in (\{0,1,2\} \cup \{\bullet\})^l$.
\end{remark}

\paragraph*{(Cyclic) ternary addition}
Finally, let us consider ternary additions, i.e.\ adding the number $v_3(k)$ for $k \in \{0,1,2\}^*$ to the $|k|$ letters on the right side of the head by considering them as a counter in base 3.

\begin{lemma}\label{lem:ternary-addition}
  For $k \in \{0,1,2\}^*$, $|k| \leq \ell$, let $\addc{k} \in \Sym(\ctape_\ell)$ be defined as the following shift-equivariant action:
  \[ \addc{k} : \begin{tpatternc}
    a_0 & \dots & a_{|k|-1} & a_{|k|} & \dots & a_{\ell-1} \\
    q & & & & &
  \end{tpatternc}
  \mapsto
  \begin{tpatternc}
    a_0' & \dots & a_{|k|-1}' & a_{|k|} & \dots & a_{\ell-1} \\
    q & & & & &
  \end{tpatternc}\] 
  where $v_3(a_0' \dots a_{|k|-1}') = v_3(a_0 \dots a_{|k|-1}) + v_3(k) \bmod 3^{|k|}$.

  Then $\addc{k}$ belongs to $\tmgroup_\ell$ with word norm $O(|k|^3)$.
\end{lemma}

\begin{proof}
  Let $r = \ID_Q \times (0,1,2) \in \Sym(Q \times \Gamma)$ denote the rotation of the tape-letter by the 3-cycle $(0,1,2) \in \Alt(\Gamma)$. 
  To perform additions modulo $3^{|k|}$, we apply the standard ``school algorithm'', whose main difficulty consists in performing carries.
  
  The application of a carry only needs to be performed when the addition of the previously added digits has overflowed, i.e.\ when the digits on which the addition has already been performed (from right to left) have a smaller value than the rightmost digits of $k$. This is exactly what we do.

  \medskip
  First, move the head to the right of the counter by applying $\rho^{|k|-1}$. Then, apply either $\pi_{r}$ if the last digit of $k$ is $1$ (i.e. $k_{|k|-1} = 1$), or $\pi_{r^2}$ if $k_{|k|-1} = 2$, or nothing if $k_{|k|-1} = 0$. 
  
  Then, for $j \in \llbracket 1,|k|-2 \rrbracket$, do:
  \begin{enumerate}
    \item Move the head to the left with $\rho^{-1}$.
    \item Apply either $\pi_{r}$ if $k_{|k|-j-1} = 1$, or $\pi_{r^2}$ if $k_{|k|-j-1} = 2$, or nothing if $k_{|k|-j-1} = 0$.
    \item Perform the carry: denoting $k' = k_{\llbracket |k|-j,|k|-1 \rrbracket}$, apply $\pi_{r, < \bullet \:\! k'}$ where $\bullet$ is the wildcard symbol, using the notations of Lemma~\ref{lem:gate-ordering-condition} and Remark~\ref{rem:Wildcards}.
  \end{enumerate}

  Then $\addc{k}$ is the composition of $O(|k|)$ permutations of $\tmgroup_\ell$, each having word norm $O(|k|^2)$ (according to Lemma~\ref{lem:gate-ordering-condition}). This concludes the proof.
\end{proof}

\begin{remark}\label{rem:ternary-addition}
  The addition $\addc{k}$ can be conditioned on any value of the state $\wgh{q} \in \wgh{Q}$ by simply intersecting all the conditions for applying $\pi_{r}$ (or $\pi_{r^2}$) with $[\{\wgh{q}\} \times \Gamma]_0 \times \ghostset$.
\end{remark}

\subsubsection{The ducking trick}\label{sec:engineering-ducking-trick}

Recall that the set $Q$ splits into $Q = \wgh{Q} \times \ghostset$, and that $\wgh{Q}$ itself splits into $\wgh{Q} = \origsmart{Q} \times \duckset$: the states $\origsmart{Q} = \{\smb,\smd,\smp,\smq\} \times \{1,2\}$ of the original SMART machine, and the duck $\duckset = \{\dright,\dleft\}$. Recall that, for $S \subseteq \ctape_\ell$, we denote by $S[\dright]$ (resp. $S[\dleft]$) the subset of tapes in $S$ containing a head whose duck is $d = \dright$ (resp $d = \dleft$).

In this section, we introduce the \emph{ducking trick}: this method uses this ducking component $\duckset$ to realize piecewise-defined bijections, in particular in the proof of Lemma~\ref{lem:encoding-with-garbage} that follows.

\medskip
Let $f = \origsmart{f} \times \ID_{\duckset \times \ghostset} : Q \times \Gamma^{\ell} \to Q \times \Gamma^{\ell}$ be a map defined on patterns of length $\ell$ by some $\origsmart{f} : \origsmart{Q} \times \Gamma^\ell \to \origsmart{Q} \times \Gamma^\ell$. Assume that $f$ is a piecewise-defined bijection, i.e.\ that there exist partitions $\sqcup_{i=1}^p \origsmart{\duckdomain_i}$ and $\sqcup_{i=1}^p \origsmart{\duckimage_i}$ of $\origsmart{Q} \times \Gamma^{\ell}$ and maps $\origsmart{f_i} : \origsmart{\duckdomain_i} \to \origsmart{\duckimage_i}$ such that $\origsmart{f} = \bigsqcup \origsmart{f_i}$ (as the union of graphs of functions).

In what follows, we denote $\duckdomain_i = \origsmart{\duckdomain_i} \times \duckset \times \ghostset \subseteq Q \times \Gamma^{\ell}$ and $\duckimage_i = \origsmart{\duckimage_i} \times \duckset \times \ghostset \subseteq  Q \times \Gamma^{\ell}$, and $f_i = \origsmart{f_i} \times \ID : \duckdomain_i \to \duckimage_i$, so that $f = \bigsqcup_i f_i$.

Assume that
\begin{enumerate}
  \item There exist maps $g_i \in \tmgroup_\ell$ with polynomial word norm in $\ell$ such that $g_i$ agrees with $f_i$ on $\duckdomain_i[\dleft]$ (i.e. $g_i|_{\duckdomain_i[\dleft]} = f_i$), and is the identity on $\duckdomain_i[\dright]$ (i.e. $g_i|_{\duckdomain_i[\dright]} = \ID$);
  \item Each $\duckdomain_i$ a simple description. Specifically, we should have a Boolean circuit of type $\NC^1$ for it.
\end{enumerate}
In our application, the descriptions of the $\duckdomain_i$ involve only numerical comparisons and direct letter comparisons, and the simplicity is encapsulated in the formulas in Section~\ref{sec:engineering-additions}. The maps $g_i$ differ from $f_i$ as the following: instead of having domain $\duckdomain_i$, they act on the whole set of patterns of length $\ell$, and are only required to act on $\duckdomain_i[\dleft]$ as $f_i$ does (and be the identity on $\duckdomain_i[\dright]$).

\begin{claim}[Ducking trick]
  Consider $F : Q \times \Gamma^\ell \to Q \times \Gamma^\ell$ defined as:
  \[ F(w) = \begin{cases} f(w) & \text{if }w \in Q \times \Gamma^{\ell}[\dright] \\
    f^{-1}(w) & \text{if } w \in Q \times \Gamma^{\ell}[\dleft] \\
    w & \text{otherwise}
  \end{cases}
  \]
  Under the assumptions above, $F$ belongs to $\tmgroup_\ell$ with polynomial word norm in $\ell$.
\end{claim}

\begin{proof}
  Denote $d' \in \Sym(\duckset)$ be the involution $d' = \dright \leftrightarrow \dleft$. Then $d = \ID_{\origsmart{Q}} \times d' \times \ID_{\ghostset} \times \ID_{\Gamma} \in \Sym(H)$ is an involution. As the sets $\duckdomain_i$ are $\pi_d$-invariant, by Lemma~\ref{lem:gate-conditioning} each $\pi_{d,\duckdomain_i}$ belongs to $\tmgroup_\ell$ with word norm $O(T(\duckdomain_i))$ (which is polynomial in $\ell$, because of the $\NC^1$-description).

  Now conjugate the permutations $\pi_{d,\duckdomain_i}$ by their respective $g_i^{-1}$,
  to get
  \[ F_i = (\pi_{d,\duckdomain_i})^{(g_i^{-1})} = g_i \circ \pi_{d,\duckdomain_i} \circ g_i^{-1}. \]
  This is an involution that acts like $f_i$ on $\duckdomain_i[\dright]$ (and flips the duck), like $f_i^{-1}$ on $\duckimage_i[\dleft]$ (and flips the duck), and is the identity on the rest of $Q \times \Gamma^{\ell}$.
  Indeed, $\pi_{d,\duckdomain_i}$ exchanges words $w \in \duckdomain_i[\dright]$ with $\pi_d(w) \in \duckdomain_i[\dleft]$. Conjugating it with $g_i^{-1}$ effectively means that we translate the domain of this permutation by $g_i$. More precisely, this is the involution
  \begin{align*}
    F_i : g_i(\duckdomain_i[\dright]) & \leftrightarrow g_i(\duckdomain_i[\dleft]) \\
    g_i(w) & \leftrightarrow g_i(\pi_d(w)) 
  \end{align*}
  As $g_i$ acts like $f_i$ on $\duckdomain_i[\dleft]$ and is the identity on $\duckdomain_i[\dright]$, we have in fact:
  \begin{align*}
    F_i : \duckdomain[\dright] & \leftrightarrow \duckimage_i[\dleft] \\
    w & \leftrightarrow f_i(\pi_d(w)) 
  \end{align*}
  And $F_i$ has polynomial word norm because $g_i$ and $\pi_{d,\duckdomain_i}$ have polynomial word norm. 
  
  Finally, each time $F_i$ acts non-trivially it also flips the duck, so all the $F_i$ have disjoint support. Composing them in any arbitrary order, we obtain that $F = \pi_d \circ \prod_i F_i$ belongs to $\tmgroup_\ell$ with polynomial word norm in $\ell$.
\end{proof}

\subsection{Implementing encodings and additions}\label{sec:proofs-cyclic-tapes}

\subsubsection{Proof of Lemma~\ref{lem:encoding-with-garbage}}

In this section, we use the previous lemmas to implement the inductive encoding of Section~\ref{sec:smart-encoding} in $\tmgroup_\ell$. More precisely, recall the statement of Lemma~\ref{lem:encoding-with-garbage}:

\encodingwithgarbage*

This section is dedicated to the proof of this lemma. Informally, we use the inductive decomposition of $\encoding_\ell$ into 
\[ \encoding_\ell = \pencoding_{\ell,\mathrm{final}} \circ \Big( \prod_{k=0}^{\ell-2} \pencoding_{k \to k+1} \Big) \circ \pencoding_{\mathrm{init}}\]
defined in Section~\ref{sec:smart-encoding}, and build every $\pencoding_{\mathrm{init}}, \pencoding_{k \to k+1}$ and $\pencoding_{\ell,\mathrm{final}}$ in the group~$\tmgroup_\ell$. Each of these steps being defined as piecewise-defined bijections, we use the \emph{ducking trick} mentioned in Section~\ref{sec:engineering-ducking-trick}: with the duck $\duckset = \{\dright,\dleft\}$, we write each bijection as a product of involutions that swap ducks~$\dright$ and~$\dleft$.

\medskip
Let $\pencoding$ be one of $\pencoding_{\mathrm{init}}$, $\pencoding_{k \to k+1}$ or $\pencoding_{\ell,\mathrm{final}}$. $\pencoding$ rewrites patterns as explained in Section~\ref{sec:smart-encoding}, and copies all other symbols unchanged. Let $\duckdomain \xmapsto{\pencoding} \duckimage$ be one of the cases of $\pencoding$, i.e.\ one of the rules of pattern rewriting that defines $\pencoding$. (Apart from the proof of Lemma~\ref{lem:encoding-case-in-group}, there is no need to actually know what $\duckdomain$, $\duckimage$ and $\pencoding$ precisely are in order to understand the following statements.)

Define
\begin{align*}
  \GRpencoding_{\duckdomain,\duckimage} : \duckdomain[\dright] & \leftrightarrow \duckimage[\dleft] \\
  w & \leftrightarrow \pi_d(\pencoding(w))
\end{align*}
as the involution of $\Sym(\ctape_\ell)$ that swaps words $w \in \duckdomain[\dright]$ with words $\pi_d (\pencoding(w)) \in \duckimage[\dleft]$, where $\pi_d$ is the duck-flipping involution. The involution $\GRpencoding_{\duckdomain,\duckimage}$ applies $\pencoding$ forward on tapes with duck $\dright$ and swaps the duck, and applies $\pencoding$ backwards on tapes with duck $\dleft$ and swaps the duck.

\begin{lemma}\label{lem:encoding-case-in-group}
  For any such case $\duckdomain \xmapsto{\pencoding} \duckimage$, the permutation $\GRpencoding_{\duckdomain,\duckimage}$ belongs to $\tmgroup_\ell$ with norm $O(\ell^3)$.
\end{lemma}

The proof relies on the structure of the different cases $\duckdomain \xmapsto{\pencoding} \duckimage$ defined in Section~\ref{sec:smart-encoding}.
\begin{proof}[Sketch of a proof.]
  Let $\duckdomain \xmapsto{\pencoding} \duckimage$ be such a piece.

  The condition $\duckdomain$ checks the state of the head, the value of a few distinct tape-letters, and a counter being non-zero (in the case of Figure~\ref{fig:bottom-up-encoding}) or full-zero (in the case of Figure~\ref{fig:bottom-up-encoding2}). Denoting by $d = \ID_{\origsmart{Q}} \times d' \times \ID_{\ghostset} \times \ID_\Gamma \in \Sym(H)$ for $d' : \dright \leftrightarrow \dleft \in \Sym(\duckset)$ the involution that swaps the duck, the gate $\pi_{d,\duckdomain}$ conditioned on $\duckdomain$ belongs to $\tmgroup_\ell$ with norm $O(\ell^2)$.

  By conjugating $\pi_{d,\duckdomain}$ by a sequence of permutations conditioned on the duck being $\dleft$, we can then build $\GRencoding_{\duckdomain,\duckimage}$ with norm $O(\ell^3)$ in $\tmgroup_\ell$.
\end{proof}

\begin{example}
  On an example, consider $\duckdomain \xmapsto{\pencoding} \duckimage$ to be the following transformation between levels $k$ and $k+1$ (this is the third rewrite in Figure~\ref{fig:bottom-up-encoding}):

  \footnotesize
  \[ \begin{tpatternc}
    \textcolor{gray}{*} & \hspace*{1.5em}c\hspace*{1.5em} & 1 & \textcolor{gray}{*} \\
    & & \smb_1 &
  \end{tpatternc}
  \quad \rightarrow \quad
  \begin{tpatternc}
    \textcolor{gray}{*} & [\val_3(c) + f(n) + 1]_{(3)} & \textcolor{gray}{*} \\
    \smd_1 & &
  \end{tpatternc} \] 
  \normalsize
  Then the condition $\duckdomain$ is the conjunction of the head being in state $\smb_1$, on top of the tape-letter $1$, and the $k+1$ tape-letters on its left being non-zero. By Lemma~\ref{lem:gate-ordering-condition}, $\pi_{d,\duckdomain}$ belongs to $\tmgroup_\ell$ with norm $O(k^2)$. 
  
  To obtain $\GRencoding_{\duckdomain,\duckimage}$, we then conjugate $\pi_{d,\duckdomain}$ with the sequence of \emph{inverses} of the following permutations, all of which are \emph{conditioned on the head having duck $\dleft$}:
  \begin{enumerate}
    \item At first, when $\pi_{d,{U}}$ is conjugated by nothing (i.e.\ the identity), we have:
    
    \footnotesize
    \[ \begin{tpatternc}
      \textcolor{gray}{*} & \hspace*{1.5em}c\hspace*{1.5em} & 1 & \textcolor{gray}{*} \\
      & & \smb_1^{\dright} &
    \end{tpatternc}
    \quad \leftrightarrow \quad
    \begin{tpatternc}
      \textcolor{gray}{*} & \hspace*{1.5em}c\hspace*{1.5em} & 1 & \textcolor{gray}{*} \\
      & & \smb_1^{\dleft} &
    \end{tpatternc} \] 
    \normalsize

    and, again, this permutation is only applied if $c$ is a non-zero counter. (Note that we see here the full support of the permutation.)
    \item Let $g' = (1,0) \in \Sym(\Gamma)$, and $g = \ID_{\origsmart{Q}} \times \{\dleft\} \times \ID_{\ghostset} \times g' \in \Sym(H)$. Apply the gate $\pi_g$, which has word norm $O(1)$ in $\tmgroup_\ell$. At the moment, we have the permutation:
    
    \footnotesize
    \[ \begin{tpatternc}
      \textcolor{gray}{*} & \hspace*{1.5em}c\hspace*{1.5em} & 1 & \textcolor{gray}{*} \\
      & & \smb_1^{\dright} &
    \end{tpatternc}
    \quad \leftrightarrow \quad
    \begin{tpatternc}
      \textcolor{gray}{*} & \hspace*{1.5em}c\hspace*{1.5em} & 0 & \textcolor{gray}{*} \\
      & & \smb_1^{\dleft} &
    \end{tpatternc} \] 
    \normalsize
    \item Using the permutation $p$ from Lemma~\ref{lem:permuting-letters} (that exchanges two adjacent letters) conditioned on the duck being $d=\dleft$ (see Remark~\ref{rem:permuting-letters}), and moving the head using the generator $\rho$ of $\tmgroup_\ell$, we move the counter $c$ one step to its right and the tape-letter $0$ under the head $k+1$ steps to its left. At the moment, we have the permutation:
    
    \footnotesize
    \[ \begin{tpatternc}
      \textcolor{gray}{*} & \hspace*{1.5em}c\hspace*{1.5em} & 1 & \textcolor{gray}{*} \\
      & & \smb_1^{\dright} &
    \end{tpatternc}
    \quad \leftrightarrow \quad
    \begin{tpatternc}
      \textcolor{gray}{*} & 0 & \hspace*{1.5em}c\hspace*{1.5em} & \textcolor{gray}{*} \\
      & \smb_1^{\dleft} & &
    \end{tpatternc} \] 
    \normalsize
    \item For $w \in \{0,1,2\}^{k+2}$ such that $v_3(w) = f(k)+1$, we apply $\addc{w}$ from Lemma~\ref{lem:ternary-addition} (namely, the permutation that adds $f(k)+1$ to the ternary number of length $k+2$ to the right of the head) conditioned on the duck being $d = \dleft$ (see Remark~\ref{rem:ternary-addition}). It has word norm $O(k^3)$. Then, apply $\rho^{-1}$ on ducks $d = \dleft$. At the moment, we have the permutation:
    
    \footnotesize
    \[ \begin{tpatternc}
      \textcolor{gray}{*} & \hspace*{1.5em}c\hspace*{1.5em} & 1 & \textcolor{gray}{*} \\
      & & \smb_1^{\dright} &
    \end{tpatternc}
    \quad \leftrightarrow \quad
    \begin{tpatternc}
      \textcolor{gray}{*} & [\val_3(c)+f(k)+1]_{(3)} & \textcolor{gray}{*} \\
      \smb_1^{\dleft} & &
    \end{tpatternc} \] 
    \normalsize
    \item Let $g' = \smb_1^{\dleft} \leftrightarrow \smd_1^{\dleft} \in \Sym(Q)$ and $g = g' \times \ID_\Gamma \in \Sym(H)$. We apply $\pi_g$ and finally obtain $\GRpencoding_{\duckdomain,\duckimage}$:
    
    \footnotesize
    \[ \begin{tpatternc}
      \textcolor{gray}{*} & \hspace*{1.5em}c\hspace*{1.5em} & 1 & \textcolor{gray}{*} \\
      & & \smb_1^{\dright} &
    \end{tpatternc}
    \quad \leftrightarrow \quad
    \begin{tpatternc}
      \textcolor{gray}{*} & [\val_3(c) + f(k) + 1]_{(3)} & \textcolor{gray}{*} \\
      \smd_1^{\dleft} & &
    \end{tpatternc} \] 
    \normalsize
  \end{enumerate}

  \vspace*{-1.8em}
  \qee
\end{example}

\medskip
Having built each case of the piecewise-defined function $\pencoding$, we complete the ``ducking'' process and obtain as an immediate corollary:
\begin{lemma}\label{lem:encoding-steps-in-group}
  For $\pencoding$ being either $\pencoding_{\mathrm{init}}$, $\pencoding_{k \to k+1}$ or $\pencoding_{\ell,\mathrm{final}}$, define
  \[ \GRpencoding(w) = \begin{cases}
    \pencoding(w) & \text{if } w \in \ctape_\ell[\dright] \\
    \pencoding^{-1}(w) & \text{if } w \in \ctape_\ell[\dleft] \\
    w & \text{otherwise}
  \end{cases} \]
  Then $\GRpencoding$ belong to $\tmgroup_\ell$ with word norm $O(\ell^3)$.
\end{lemma}
\begin{proof}
  Each of these transformations $\pencoding$ is composed of finitely many cases $\duckdomain \xmapsto{\pencoding} \duckimage$, and by the previous lemma, each $\GRpencoding_{\duckdomain,\duckimage}$ belong to $\tmgroup_\ell$ with word norm $O(\ell^3)$. Denote $d \in \Sym(H)$ the involution that swaps ducks $d=\dright$ and $d=\dleft$. As the $\GRpencoding_{\duckdomain,\duckimage}$ commute (since they have disjoint support), the involution $\GRpencoding$ can be written as the composition of finitely many $\GRpencoding_{\duckdomain,\duckimage}$ and $\pi_d$.
\end{proof}

We then conclude and obtain Lemma~\ref{lem:encoding-with-garbage}:
\encodingwithgarbage*

\begin{proof}[Proof of Lemma~\ref{lem:encoding-with-garbage}]
  Recall that the encoding map $\encoding_\ell$ from Section~\ref{sec:smart-encoding-overall} is defined as $\encoding_\ell = \pencoding_{\ell,\mathrm{final}} \circ \prod_{k=0}^{\ell-2} \pencoding_{k \to k+1} \circ \pencoding_{\mathrm{init}}$.

  By the previous lemma, there exists $\GRpencoding_{\mathrm{init}}$, $\GRpencoding_{k \to k+1}$ and $\GRpencoding_{\ell,\mathrm{final}}$ in $\tmgroup_\ell$ with word norm $O(\ell^3)$ that agree on $\ctape_\ell[\dright]$ with their respective $\pencoding_{\ell,\mathrm{final}}$, $\pencoding_{k \to k+1}$ and $\pencoding_{\mathrm{init}}$. As the set $\ctape_\ell[\dright]$ is stable by these permutations, we define
  \[ \GRencoding_\ell = \GRpencoding_{\ell,\mathrm{final}} \circ \prod_{k=0}^{\ell-2} \GRpencoding_{k \to k+1} \circ \GRpencoding_{\mathrm{init}}.\]
  Then the map $\GRencoding_\ell$ satisfies:
  \[ w \in \ctape_\ell[\dright] \implies \GRencoding_\ell(w) = \encoding_\ell(w) \in \ctape_\ell[\dright].\qedhere \]
\end{proof}

One should note that, on configurations $w \in \ctape_\ell[\dleft]$, $\GRencoding_\ell$ ``produces garbage'': we claim no meaningful interpretation for the image of such $w$. In particular, we note that while for $w \in \ctape_\ell[\dleft]$, we have $\GRpencoding_{\mathrm{init}}(w) = \pencoding_{\mathrm{init}}^{-1}(w)$, the composition $\pencoding_{\ell,\mathrm{final}}^{-1} \circ \prod_{k=0}^{\ell-2} \pencoding_{k \to k+1}^{-1} \circ \pencoding_{\mathrm{init}}^{-1}$ does \textbf{not} equal the inverse of $\encoding$.

\subsubsection{Proof of Lemma \ref{lem:addition-in-encoded-base}}\label{sec:proofs-additition-encoded-base}

Let us recall the statement of Lemma~\ref{lem:addition-in-encoded-base}:
\additioninencodedbase*

More precisely, denoting by $q_b$ any state in $[q_b] = \{q_b\} \times \duckset \times \ghostset$, the map $(\rduck{\plusorbit{n}}{\dright})_\ell$ bijectively acts on the following patterns of length $\ell$ as:
\[ 
  \rduck{\plusorbit{n}}{\dright}
  \begin{tpatternc}
  c_0 & \dots & c_{\ell-1} \\
  q_b & &
  \end{tpatternc}
  \to 
  \sigma^{- \varepsilon a} \left(
  \begin{tpatternc}
    c_0' & \dots & c_{\ell-1}' \\
    q_{b'}
  \end{tpatternc} \right)
\]
if $q_b \in \{q_b\} \times \{\dright\} \times \ghostset$, where $b' \cdot c' \in \{1,2\} \cdot \{0,1,2\}^\ell$ encodes $((b-1)\cdot 3^\ell + v_3(c)) + n \bmod 2 \cdot 3^\ell$, $a$ is the quotient of $((b-1)\cdot 3^\ell + v_3(c)) + n$ by $2 \cdot 3^\ell$, and $\epsilon = +1$ if $q \in \{\smb,\smp\}$, $\epsilon = -1$ if $q \in \{\smd,\smq\}$; and is the identity otherwise.

\medskip
We now prove Lemma~\ref{lem:addition-in-encoded-base}.

\begin{proof}
  Using the permutation $p$ of constant word norm from Lemma~\ref{lem:permuting-letters}, conditioned on the duck being $d = \dright$ and the state being either in $\{\smb,\smp\}$ or in $\{\smd,\smq\}$ (see Remark~\ref{rem:permuting-letters}), we can rotate the tape either right (if $\smb$ or $\smp$) or left (if $\smd,\smq$) at most $\ell$ times with word norm $O(\ell^2)$. We can now assume that $0 \leq n < 2 \cdot 3^\ell$.

  \medskip
  Let $b \cdot k \in \{1,2\} \cdot \{0,1,2\}^\ell$ encode $n$, i.e.\ $v_3(k) = n \bmod 3^\ell$, and $b-1$ be the quotient of $n$ by $3^\ell$. We first add $b \cdot k$ without rotating the tape.

  To do so, apply the permutation $\addc{k}$ of Lemma~\ref{lem:ternary-addition} to add $n \bmod 3^\ell$ to the cyclic tape conditioned on $d = \dright$ (see Remark~\ref{rem:ternary-addition}). Then, very similarly to the proof of the same lemma, we perform a carry in the component $\{1,2\}$ of the state (denoting $s = s' \times \{\dright\} \times \ID_\ghostset \times \ID_\Gamma \in \Sym(H)$, where $s' = (q,1) \leftrightarrow (q,2) \in \Sym(\origsmart{Q})$, we perform the conditioned gate $\pi_{s,<k}$ defined in Lemma~\ref{lem:gate-ordering-condition}), and depending on the bit $b$ we cyclically rotate the bit carried by state (if $b=2$, we apply $\pi_s$, and if $b=1$ we apply the identity).

  \medskip
  We are now left with performing a cyclic rotation of the whole tape if the addition of $n$ in the previous paragraph overflowed. Let $s_1 \in \Sym(H)$ and $s_2 \in \Sym(H)$ respectively flip filled arrows $\smb \leftrightarrow \smd$ and $\smp \leftrightarrow \smq$, i.e.\ $s_1 = s_1' \times \{\dright\} \times \ID \times \ID \in \Sym(H)$ (resp. $s_2 = s_2' \times \{\dright\} \times \ID \times \ID \in \Sym(H)$), for $s_1' = (\smb \leftrightarrow \smd) \times \ID \in \Sym(\origsmart{Q})$ (resp. $s_2' = (\smp \leftrightarrow \smq) \times \ID \in \Sym(\origsmart{Q})$). By Lemma~\ref{lem:gate-ordering-condition}, both $\pi_{s_1,C}$ and $\pi_{s_2,C}$ belong to $\tmgroup_\ell$ with word norm $O(\ell^2)$, for $C$ the overflowing condition $C =\; < (b \cdot k)$.\footnote{We slightly abuse notations here when writing $< (b \cdot k)$: these conditions were defined for comparisons of length $\leq \ell$ of tape-letters, here the comparison is of length $\ell+1$. Instead of just comparing on the tape, the bit $\{1,2\}$ carried in the state is also part of the comparison.}

  Let $r_\smb \in \Sym(\ctape_\ell)$ (resp. $r_\smp \in \Sym(\ctape_\ell)$) be the right cyclic shift of whole tapes having states in $\{\smb\} \times \{1,2\} \times \{\dright\} \times \ghostset$ (resp. $\{\smp\} \times \{1,2\} \times \{\dright\} \times \ghostset$). These permutations belong to $\tmgroup_\ell$ with word norm $O(\ell)$ by Lemma~\ref{lem:permuting-letters}.

  Then the commutator $[\pi_{s_1,C},r_\smb]$ (resp. $[\pi_{s_2,C},r_\smp]$) shifts the whole tape if and only if an overflow happened, and it either shifts to the right if the state is in $\{\smb\} \times \{1,2\} \times \{\dright\} \times \ghostset$ (resp. $\{\smp\} \times \{1,2\} \times \{\dright\} \times \ghostset$), or to the left if the state is in $\{\smd\} \times \{1,2\} \times \{\dright\} \times \ghostset$ (resp. $\{\smq\} \times \{1,2\} \times \{\dright\} \times \ghostset$). 

  \medskip
  We conclude that $\rduck{\plusorbit{n}}{\dright}$, being the composition of the previous paragraphs, belongs to $\tmgroup_\ell$ with word norm $O(\ell^3)$.
\end{proof}

\section{Distortion of automorphisms of the full shift}\label{sec:on-full-shift}
This chapter contains a proof of Theorem~\ref{thm:distortion-every-full-shift}:
\distortioneveryfullshift*

By~\cite{1990-KR}, the automorphism groups of full shifts with different (non-trivial) alphabets embed in each other, so we only need to prove that \emph{there exists $\Sigma$ and some $g \in \Aut(\Sigma^\Z)$ of infinite order such that $|g^n|_F = O(\log^4 n)$ for some finite set $F$.}

\medskip
We first introduce the conveyor belt construction, which allows to embbed a given Turing machine into a full shift. We then define a group $\tmgroup_*$, which contains every Turing machine on a given set of states and alphabet, along with some instructions to modify the conveyor belt structures in configurations. Then, we state Lemma~\ref{lem:nice-implies-distortion-in-full-shift}: every finitarily distorted Turing machine $\mathcal{M}$ gives rise to a distorted automorphism in $\tmgroup_*$. Overall, the proof consists in transporting the already existing (finitary) distortion into a full shift.

We finally add some various optimizations in the case of the SMART machine. All our results are summarized in Lemma~\ref{lem:better-bounds-for-smart}, which precisely defines a full shift $\Sigma^\Z$ and a distorted automorphism $f_\smart$ of infinite order which satisfies the aforementioned polylogarithmic word norm of degree four.

\subsection{Context and statements of results}\label{sec:alaala}

\subsubsection{Conveyor belts}

Recall that a Turing machine $\mathcal{M} = (Q,\Gamma,\Delta)$ acts on the related subshift $\bitape$ (see Section~\ref{sec:defs}), i.e.\ on the set of bi-infinite tapes containing at most one head (i.e.\ one symbol of $Q \times \Gamma)$.

To make a Turing machine $\mathcal{M} = (Q,\Gamma,\Delta)$ act on a full shift instead, we use the conveyor belt trick (see for example~\cite[Lemma 3]{2017-GS}). Let
\[ \Sigma = \Big(\Gamma^2 \times \{+1,-1\}\Big) \sqcup \Big( (Q \times \Gamma) \times \Gamma \Big) \sqcup \Big(\Gamma \times (Q \times \Gamma)\Big) \]
where we call \emph{conveyor bits} the bits of $\{+1,-1\}$ in $(\Gamma^2 \times \{+1,-1\})$, and define the action of $\mathcal{M}$ on $\Sigma^\Z$ as the following automorphism $f_\mathcal{M}$.

Any configuration $x \in \Sigma^\Z$ uniquely splits into $x = \dots w_{-2}w_{-1} w_0 w_1 w_2 \dots$ such that for every $i \in \Z$, we have either
\begin{align*}
  & w_i \in \Big(  \Gamma^2 \times \{+1\} \Big)^* 
    \Big( \Big((Q \times \Gamma) \times \Gamma \Big) \cup \Big(\Gamma \times (Q \times \Gamma)\Big)\Big)
    \Big ( \Gamma^2 \times \{-1\} \Big)^* \\
  \text{or } & w_i \in \Big(  \Gamma^2 \times \{+1\} \Big)^+ \Big(  \Gamma^2 \times \{-1\} \Big)^+
\end{align*}
with the exception that on some configurations, their might exist a leftmost or rightmost word with an infinite number of $+1$ or $-1$.\footnote{The corresponding decomposition claim in~\cite{2017-GS} has a misprint, as it uses Kleene stars also in the second form.} We describe the action of $f_\mathcal{M}$ on such finite words $w = w_i$ of these two forms: as configurations made of infinitely many finite $w$ are dense, and $f_\mathcal{M}$ will be uniformly continuous on these, $f_\mathcal{M}$ will uniquely extend to an automorphism on the full shift.

\begin{itemize}
    \item On words $w \in \Big(\Gamma^2 \times \{+1\} \Big)^+ \Big(\Gamma^2 \times \{-1\} \Big)^+$, we do nothing.
    \item On words $w \in \Big(\Gamma^2 \times \{+1\} \Big)^* \Big(\Big((Q \times \Gamma) \times \Gamma \Big) \cup \Big(\Gamma \times (Q \times \Gamma)\Big) \Big) \Big(\Gamma^2 \times \{-1\}\Big)^*$, let 
    \[ w' \in \Big(\Gamma^2\Big)^* \Big(\Big((Q \times \Gamma) \times \Gamma\Big) \cup \Big(\Gamma \times (Q \times \Gamma)\Big)\Big) \Big(\Gamma^2\Big)^* \]
    be the word obtained by erasing the conveyor bits $+1$ and $-1$ from $w$. We see $w'$ as a conveyor belt of length $2|w|$, that is the superimposition of a top word $u$ and a bottom word $v$, glued together at their borders as if the words were laid down on a conveyor belt.

    More precisely, for $\pi_1$ and $\pi_2$ the projections to the first and second component of the alphabet, let $u = \pi_1(w')$ and $v = \overline{\pi_2(w')}$ the reverse of~$\pi_2(w')$. One of these words is in $\Gamma^+$, and the other is in $\Gamma^* (Q \times \Gamma) \Gamma^*$. We make $\mathcal{M}$ act on $(uv)^\Z$ (despite it having infinitely many heads, it should be clear what this means, as all the heads move with the same transition), i.e.\ we define
    \[ u'v' = T_{\mathcal{M}} \left( (uv)^\Z \right)_{\llbracket 0,2|w|-1 \rrbracket} \]
    Note that $u'v'$ also contains exactly one head. We then rewrap $u'v'$ into a conveyor belt of $(\Gamma^2)^* (((Q \times \Gamma) \times \Gamma) \cup (\Gamma \times (Q \times \Gamma))) (\Gamma^2)^*$, and add conveyor bits $+1$ (resp. $-1$) to the cell symbols in $\Gamma^2$ to the left of the head (resp. right).

    This defines how $f_\mathcal{M}$ acts on such words $w$. This can be summarized as: $f_\mathcal{M}$ considers such words as a cyclic tape folded in the shape of a conveyor belt, and acts on the cyclic tape.
\end{itemize}

Note that $x$ and $f_\mathcal{M}(x)$ have the same decomposition into a product of conveyor belts, and that if $\mathcal{M}$ is reversible, then $f_\mathcal{M}$ is an automorphism of $\Aut(\Sigma^\Z)$.

\subsubsection{Group of Turing machines on conveyor belts}

Similar to the group of Turing machines $\tmgroup_\ell$ acting on the finite cyclic tapes of $\ctape_\ell$, let us define another point of view on Turing machines acting on full-shifts. For $g \in \Sym(Q \times \Gamma)$, define $f_{g}^{\aup}, f_{g}^{\adown}, f_g \in \Aut(\Sigma^\Z)$ as:
\begin{align*}
  f_{g}^{\aup}(x)_i & = \begin{cases}
    x_i & \text{if } x_i \in \Gamma^2 \times \{+1,-1\} \text{ or } x_i \in (\Gamma \times (Q \times \Gamma)) \\
    (g(a),b) & \text{if } x_i = (a,b) \in ((Q \times \Gamma) \times \Gamma )
  \end{cases} \\
  f_{g}^{\adown}(x)_i & = \begin{cases}
    x_i & \text{if } x_i \in \Gamma^2 \times \{+1,-1\} \text{ or } x_i \in ((Q \times \Gamma) \times \Gamma ) \\
    (a,g(b)) & \text{if } x_i = (a,b) \in (\Gamma \times (Q \times \Gamma))
  \end{cases} \\
  f_g(x)_i & = \begin{cases}
    x_i & \text{if } x_i \in \Gamma^2 \times \{+1,-1\} \\
    (g(a),b) & \text{if } x_i = (a,b) \in ((Q \times \Gamma) \times \Gamma ) \\
    (a,g(b)) & \text{if } x_i = (a,b) \in (\Gamma \times (Q \times \Gamma))
  \end{cases}
\end{align*}
for $x$ a configuration of $\Sigma^\Z$.

For every $q \in Q$, define $\rho_q$ as the right movement of heads in state $q$ \emph{inside their own conveyor belts}, and $\rho = \prod_{q \in Q} \rho_q$. Then define $\tmgroup$ the group they generate:

\[ \tmgroup = \langle \{ f_{g}^\aup, f_{g}^\adown, f_g \mid g \in \Sym(Q \times \Gamma) \} \cup \{ \rho_q \mid q \in Q \} \rangle \]

The generators of $\tmgroup_\ell$ introduced in Section~\ref{sec:Finitary} are in direct correspondence with the generators $\{ f_g \mid g \in \Sym(Q \times \Gamma) \}$ and $\{ \rho_q \mid q \in Q\}$ of $\tmgroup$, and the latter can also be seen as the basic instructions of Turing machines: moving heads based on their states, or permuting their values. 

\medskip
In Section~\ref{sec:Finitary}, we mentioned that every Turing machine $\mathcal{M} = (Q,\Gamma,\dots)$ belongs to the groups $\tmgroup_\ell$. Similarly, for any such machine, the automorphism $f_\mathcal{M}$ belongs to $\tmgroup$, since it is the composition $\beta_{-1} \circ \beta_{+1} \circ f_\alpha$:
\begin{align*}
  \alpha(q,a) & = \begin{cases}
    (q',b) & \text{if } (q,a,q',b) \in \Delta \\
    (q',a) & \text{if } (q,\pm 1, q') \in \Delta
  \end{cases} \\
  \beta_{+1} & = \prod_{q' \mid \exists q, (q,+1,q') \in \Delta} \rho_{q'}  \\
  \beta_{-1} & = \prod_{q' \mid \exists q, (q,-1,q') \in \Delta} {\rho_{q'}}^{-1}
\end{align*}

\subsubsection{Decorating Turing machines}

Given any two sets $\Gamma_0$ and $Q_0$, define:
\begin{align*}
  \Gamma & = \Gamma_0 \\
  Q & = Q_0 \times \duckset \times \ghostset \\
  \Sigma & = \Big( \Gamma^2 \times \{+1,-1\} \Big) \sqcup \Big( (Q \times \Gamma) \times \Gamma \Big) \sqcup \Big( \Gamma \times (Q \times \Gamma) \Big)
\end{align*}
where $\duckset = \{\dforw,\dback\}$ and $\ghostset = \llbracket 0, 5 \rrbracket$. As in Section~\ref{sec:decorated-smart}, $\duckset$ is called the \emph{duck} and $\ghostset$ is called the \emph{ghost}.

We say that $\Gamma$ and $Q$ are the decorated versions of $\Gamma_0$ and $Q_0$.

\subsubsection{Generalized group \texorpdfstring{$\tmgroup_*$}{G\_{}*} of Turing machines on conveyor belts}\label{sec:generalized-group-of-tm}
Let us now define an automorphism $\theta$ that intuitively moves all (decorated) heads to the right, \emph{disregarding the conveyor belt structure}, allowing heads to visit areas outside of their conveyor belts. This is completely \emph{ad-hoc}, and only used to build three specific automorphisms in Section~\ref{sec:creating-erasing-cb}.

\medskip
First, split $\ghostset = \llbracket 0,5 \rrbracket$ into $\ghostset \simeq \{+1,-1\} \times \llbracket 0, 2 \rrbracket$, where $\{+1,-1\}$ is its \emph{sign component}. Now every symbol of $\Sigma$ contains a sign $\{+1,-1\}$, either in the sign-component of the state, or as its conveyor bit. Let $\pi_\sign : \Sigma \to \{+1,-1\}$ return the sign of any symbol in $\Sigma$ bit, i.e.:
\[ \pi_\sign(z) = \begin{cases}
  \varepsilon & \text{if } z = ((a,b),\varepsilon) \in \Gamma^2 \times \{+1,-1\} \\
  \varepsilon & \text{if } z = ((q,a),b) \in (Q \times \Gamma) \times \Gamma$ or $z = (a,(q,b)) \in (\Gamma \times (Q \times \Gamma)) \\
  & \text{where } q = (q_0,d,x) \text{ and } x \simeq (\varepsilon,x') \in \ghostset
  \end{cases} \]

We can then define the automorphism $\theta$ that moves every head to the right while leaving this sign symbol intact. More precisely, let $\sigma_\sign : Q \times \{+1,-1\} \to Q$ rewrite the sign of the state:
\[ \sigma_\sign : ((q_0,d,(\_{},x')),\, \varepsilon) \to (q_0,d,(\varepsilon,x')) \]
i.e.\ $\pi_\sign(\sigma_\sign(q,\varepsilon)) = \varepsilon$, let $\pi_{\aup} : \Sigma \to \Gamma$ return the top tape-letter, $\pi_{\adown} : \Sigma \to \Gamma$ return the bottom tape-letter, and $\pi_Q : ((Q \times \Gamma) \times \Gamma) \cup (\Gamma \times (Q \times \Gamma)) \to Q$ return the state. Define $\theta : \Sigma^\Z \to \Sigma^\Z$ as:
\[ \theta(x)_i = \begin{cases}
  ((q',a),b) & \text{if } x_{i-1} \in (Q \times \Gamma) \times \Gamma \\
  (a,(q',b)) & \text{if }x_{i-1} \in \Gamma \times (Q \times \Gamma) \\
  ((a,b),\pi_\sign(x_i)) & \text{if } x_{i-1} \in \Gamma^2 \times \{+1,-1\}
\end{cases}\]
for $x \in \Sigma^\Z$, where $a = \pi_\aup(x_i)$, $b = \pi_\adown(x_i)$, and $q'$ is the state of $x_{i-1}$ with sign component $\pi_\sign(x_i)$, i.e.\ $q' = \sigma_\sign(\pi_Q(x_{i-1}), \pi_\sign(x_i))$.

Then $\theta$ is an automorphism of bi-radius $1$, and we define
\[  \tmgroup_* = \langle \tmgroup \cup \{ \theta \} \rangle \]
the finitely-generated group generated by $\theta$ and the Turing-machine instructions of $\tmgroup$. This finitely-generated group $\tmgroup_*$ is where we prove distortion in Lemma~\ref{lem:nice-implies-distortion-in-full-shift}.

\subsubsection{Main result: distortion in \texorpdfstring{$\Sigma^\Z$}{Sigma\^{}Z}}\label{sec:automorphism-main-result}

Let $\mathcal{M}_0 = (Q_0,\Gamma_0,\Delta_0)$ be an arbitrary Turing machine. Denoting $Q = Q_0 \times \duckset \times \ghostset$ and $\Gamma = \Gamma_0$ the decorated versions of $Q_0$ and $\Gamma_0$, we define the symmetrized Turing machine $\mathcal{M} = (Q,\Gamma,\Delta)$ that acts forward in time on ducks $\dforw$ and backward on ducks $\dback$:
\begin{align*}
  \Delta =\ & \bigcup_{x \in \ghostset} \left\{ \Big( (q,\dforw,x),a,(q',\dforw,x),b \Big) : (q,a,q',b) \in \Delta_0 \right\} \\
  & \qquad \cup \left\{ \Big((q,\dforw,x),\delta,(q',\dforw,x)\Big) : (q,\delta,q') \in \Delta_0 \right\} \\
  & \bigcup_{x \in \ghostset} \left\{ \Big( (q',\dback,x),b,(q,\dback,x),a \Big) : (q,a,q',b) \in \Delta_0 \right\} \\
  & \qquad \cup \left\{ \Big((q',\dback,x),\delta,(q,\dback,x)\Big) : (q,\delta,q') \in \Delta_0 \right\}
\end{align*}

We can now establish Lemma~\ref{lem:nice-implies-distortion-in-full-shift}, which is the main result of this section.

\begin{restatable}{lemma}{niceimpliesdistortioninfullshift}\label{lem:nice-implies-distortion-in-full-shift}
  Assume some Turing machine $\mathcal{M}_0$ satisfies the three properties of Lemma~\ref{lem:smart-distorted-on-cyclic-tapes}. Then, denoting $\mathcal{M}$ its symmetrized (and decorated) version, the automorphism $(f_{\mathcal{M}})^n$ of $\Aut(\Sigma^\Z)$ (i.e.\ the action of $\mathcal{M}$ on the conveyor belts of $\Sigma^\Z$) has word norm $O(\log^{p+1} n + \log^2 n)$ in $\tmgroup_*$.
\end{restatable}

\begin{remark}
  Note that the assumption in this lemma focuses on the original machine $\mathcal{M}_0$ (and not $\mathcal{M}$) verifying the properties of Lemma~\ref{lem:smart-distorted-on-cyclic-tapes}, while the conclusion of this lemma focuses on the symmetrized version $(f_{\mathcal{M}})^n$ (and not $(f_{\mathcal{M}_0})^n$).
\end{remark}

\subsubsection{Overview}

As the decorated SMART machine satisfies Lemma~\ref{lem:smart-distorted-on-cyclic-tapes}, combining it with Lemma~\ref{lem:nice-implies-distortion-in-full-shift} above we obtain a distortion element of infinite order in $\tmgroup_*$, whose powers have word norm $O(\log^5 n)$. However, this does not yet prove Theorem~\ref{thm:distortion-every-full-shift}: to obtain an upper bound $O(\log^4 n)$ on the word norm, we add a few additional tricks and optimizations in Section~\ref{sec:full-shift-optimize-degree}.

\medskip
This chapter focuses on the proof of Lemma~\ref{lem:nice-implies-distortion-in-full-shift} and Theorem~\ref{thm:distortion-every-full-shift}.

Intuitively, we prove Lemma~\ref{lem:nice-implies-distortion-in-full-shift} by applying the $\ell$-cyclic automorphism $T_{\ell,\mathcal{M}}$ in the conveyor belts of length $\ell$. The automorphisms $T_{\ell,\mathcal{M}}$ are finitarily distorted, and the distortion is transported from the cyclic automorphisms to $f_{\mathcal{M}}$.

Section~\ref{sec:engineering-full-shift} develops several tools for this proof. First, we build the tools to condition the application of automorphisms by the length of conveyor belts, as to apply the correct $T_{\ell,\mathcal{M}}$ to the conveyor belts of the correct length. Morally, $f_\mathcal{M}$ is then the infinite product of all the $T_{\ell,\mathcal{M}}$.

Then, we develop a method called the \emph{two-scale trick} to express $f_\mathcal{M}$ as a product of finitely many $T_{\ell,\mathcal{M}}$. It generates temporary conveyor belts of sufficient length so as to move the head (with finitary distortion) without it ``seeing'' the temporary borders, and then erase them.

Section~\ref{sec:full-shift-main-proof} then contains the proof of Lemma~\ref{lem:nice-implies-distortion-in-full-shift}. As mentioned above, Section~\ref{sec:full-shift-optimize-degree} covers various optimizations in the case of the SMART machine, in order to obtain the final degree four of polylogarithmic norm growth.

\begin{remark}
  At this point, the reader may think that conveyor belts are a restriction imposed by the context (as a way to embed Turing machines into a full-shift); and that proving distortion in a similar setting without conveyor belts, for example on the bi-infinite tapes of $\bitape$ (i.e.\ in the groups of generalized Turing machines $\RTM(n, k)$) would be easier.

  While the former is true, we think the latter is misleading. Indeed, the idea of temporary conveyor belts (from the two-scale trick) is a key component of our proof of distortion, even in the groups $\RTM(n, k)$. Without the ability to mark and erase temporary borders, we do not know how to prove the distortion of (what morally is) the SMART machine.
\end{remark}

\subsection{Permutation engineering in \texorpdfstring{$\Sigma^\Z$}{Sigma\^{}Z}}\label{sec:engineering-full-shift}

Similarly to Section~\ref{sec:engineering-perm-condition}, in this new setting the states of $Q$ once again have a ghost component $\ghostset$. We separate $Q$ into two components: $Q = \wgh{Q} \times \ghostset$ (so, with the notation $Q = Q_0 \times \duckset \times \ghostset$, we have $\wgh{Q} = \origsmart{Q} \times \duckset$).

With similar notations, the set of heads $H = Q \times \Gamma$ also split as $H = \wgh{H} \times \ghostset$, where $\wgh{H} = \wgh{Q} \times \Gamma$. Note that, for any $g \in \Sym(\wgh{H})$, the permutation $g \times \ID_{\ghostset}$ belongs to $\Sym(H)$.

In this section, we focus on building a new sort of gate-conditioning: conditioning on the structure of conveyor belts. To proceed, we somewhat follow the same ideas that we already developed in Section~\ref{sec:engineering-perm-condition}. Finally, we prove that we can modify the conveyor belt structure in Section~\ref{sec:creating-erasing-cb}.

\subsubsection{Structure conditioning}\label{sec:struct-conditioning}

Let us define $\cbstructcond$, the set of conditions that focus on the structure of conveyor belts.

Given a head and a conveyor belt, we denote by $\len \in \N$ the length of the conveyor belt (i.e.\ the length of the underlying finite cyclic tape). The integers $\rcbdist \in \N$ and $\lcbdist \in \N$ refer to the distance between a head and, respectively, the right and left border of the conveyor belt. In this context, $\cbstructcond$ is defined as the smallest set of conditions satisfying:
\begin{itemize}
  \item For every $n \in \N$, the condition $\len | n$ (i.e.\ the length of the conveyor belt divides $n$) belongs to $\cbstructcond$.
  \item The conditions $\rcbdist = 0$ and $\lcbdist = 0$ belong to $\cbstructcond$.
  \item For $n \in \Z$, the condition $\rho^n(C)$ belongs to $\cbstructcond$.
  \item If $C_1,C_2 \in \cbstructcond$, then $C_1 \wedge C_2 \in \cbstructcond$.
  \item If $C_1,C_2 \in \cbstructcond$, then $C_1 \vee C_2 \in \cbstructcond$.
  \item If $C \in \cbstructcond$, then $\neg C \in \cbstructcond$.
\end{itemize}
Here, for $C$ a condition, the condition $\rho^n(C)$ (for $n \in \Z$) holds if and only if $C$ holds after applying $\rho^n$.

\medskip
For any permutation $g \in \Sym(H)$ and any condition $C \in \cbstructcond$, define $f_{g,C}$ to be the conditioned gate that applies $g$ on a head if and only if this head belongs to a conveyor belts that verifies the condition $C$. 

\begin{lemma}\label{lem:condition-cb-struct}
  For any permutation $g \in \Sym(\wgh{H})$ and condition $C \in \cbstructcond$, the conditioned gate $f_{g \times \ID,C}$ belongs to $\tmgroup$. Additionally, if $T : \cbstructcond \to \N$ denotes the following function:
  \[ T(C) = \begin{cases}
    n & \text{if } C \text{ is } \len | n \\
    1 & \text{if } C \text{ is } \rcbdist = 0 \text{ or } \lcbdist = 0 \\
    T(C_1) + |n| & \text{if } C = \rho^n(C_1) \\
    2T(C_1) + 2T(C_2) & \text{if } C = C_1 \wedge C_2 \\
    2T(C_1) + 2T(C_2) + 5 & \text{if } C = C_1 \vee C_2 \\
    T(C_1) + 1 & \text{if } C = \neg C_1
  \end{cases}\]
  then the word norm of $f_{g \times \ID,C}$ in $\tmgroup$ is $O(T(C))$.
\end{lemma}

We divide the proof of this lemma into the following claims.
\begin{claim}\label{lem:condition-cb-divide-length}
  Let $g \in \Alt(\ghostset)$ and $\ell \in \N$. Then the conditioned gate $f_{\ID \times g,\len | \ell}$ (that applies $\ID_{\wgh{H}} \times g$ on conveyor belts whose length divides $\ell$) belongs to $\tmgroup$ with word norm $O(\ell)$.
\end{claim}

\begin{proof}
  Let $r' = (0,1,\dots,|\Gamma|-1) \in \Sym(\Gamma)$ and $r = \ID_Q \times r' \in \Sym(H)$ (recall $H = Q \times \Gamma)$ be the rotation of the tape-letter under the head. This proof relies on the following trivial observation: applying $f_r$ changes the letter at both position $0$ and position $\ell$ if and only if the length of the conveyor belt divides $\ell$ (since the tapes are cyclic).

  Let $C = \bigcup_{a \in \Gamma} \Big([Q \times \{a\}]_0 \cap [a]_{\ell}\Big)$ be the condition that checks whether cells at position $0$ and $\ell$ have the same tape-letter, and $g \in \Alt(\ghostset)$. Since $\ghostset$ has cardinality at least five, by Ore's theorem~\cite[Theorem 7]{1951-Ore}, there exist $g_1,g_2 \in \Alt(\ghostset)$ such that $g = [g_1,g_2]$; and according to Lemma~\ref{lem:ghost-conditioning}, both $f_{\ID \times g_1,C}$ and $f_{\ID \times g_2,C}$ belong to $\tmgroup$ with word norm $O(\ell)$. Then,
  \[ f_{\ID \times g, \len | \ell} = \left[ f_{\ID \times g_1, C}, f_r \circ f_{\ID \times g_2, C} \circ f_r^{-1} \right].\qedhere \]
\end{proof}

\begin{claim}\label{lem:condition-cb-on-borders}
  Let $g \in \Alt(\ghostset)$ and $C$ be a condition $\rcbdist = 0$ or $\lcbdist = 0$. Then the conditioned gate $f_{\ID \times g,C}$ (that applies $\ID_{\wgh{H}} \times g$ on heads that are respectively on the left or right border of their conveyor belt) belongs to $\tmgroup$ with word norm $O(\ell)$.
\end{claim}

\begin{proof}
  As $|\ghostset| \geq 5$, by Ore's theorem there exists $g_1,g_2 \in \Alt(\ghostset)$ such that $g = [g_1,g_2]$. Applying the commutator trick,
  \[ [f_{g_1}^{\aup}, \; \rho^{-1} \circ f_{g_2}^{\adown} \circ \rho ] \]
  applies $f_{[g_1,g_2]}$ on heads that are exactly in the top-right corner of their conveyor belts. Similar formulas exist for bottom-right, top-left and bottom-left corners of conveyor belts, so that one can condition any $f_g$ on being applied on heads that are in the left or right corners of their conveyor belts with word norm~$O(1)$.
\end{proof}

We can now proceed with the proof of Lemma~\ref{lem:condition-cb-struct}. This proof is very similar to the proofs of Lemmas~\ref{lem:ghost-conditioning} and~\ref{lem:gate-conditioning}.

\begin{proof}[Proof of Lemma~\ref{lem:condition-cb-struct}]
  By the proof of Lemma~\ref{lem:gate-conditioning}, we only need to prove that for every $g \in \Alt(\ghostset)$ and $C \in \cbstructcond$, the conditioned gate $f_{\ID \times g, C}$ belongs to $\tmgroup$.

  Let $C \in \cbstructcond$. If $C$ is $\len | n$, then for any $g \in \Alt(\ghostset)$, the conditioned gate $f_{g,C}$ belongs to $\tmgroup$ with word norm $O(n)$ by Claim~\ref{lem:condition-cb-divide-length}. If $C$ is $\rcbdist = 0$ or $\lcbdist = 0$, then $f_{g,C}$ belongs to $\tmgroup$ with word norm $O(1)$ by Claim~\ref{lem:condition-cb-on-borders}. If $C = \rho^n(C_1)$ for some $C_1 \in \cbstructcond$, then for any $g \in \Alt(\ghostset)$, we have $f_{g,C} = \rho^{-n} \circ f_{g,C_1} \circ \rho^n$.

  Finally, if $C = C_1 \wedge C_2$, $C = C_1 \vee C_2$ or $C = \neg C_1$, the proof of Lemma~\ref{lem:ghost-conditioning} applies \emph{mutatis mutandis}.
\end{proof}

\begin{remark}
  The very same proof shows that $f_{g \times \ID}^{\aup}$ and $f_{g \times \ID}^{\adown}$ can also be conditioned by $\cbstructcond$.
\end{remark}

\subsubsection{Corollary: conditioning on the length of conveyor belts}

For $\ell \in \N$ and ${\sim} \in \{<, \leq, =, \geq, >\}$, let us define conditions $\len \sim \ell$ (comparisons on the length of the conveyor belt).

\begin{lemma}\label{lem:condition-cb-length}
  Let $\ell \in \N$ and ${\sim} \in \{<,\leq,=,\geq,>\}$. If $C = \len \sim \ell$, then $C$ belongs to $\cbstructcond$ and $T(C) = O(\ell^2)$.

  In other words, for any $g \in \Sym(\wgh{H})$, the conditioned gates $f_{g \times \ID,C}$ belong to $\tmgroup$ with word norm $O(\ell^2)$.
\end{lemma}

\begin{proof}
  By the proof of Lemma~\ref{lem:gate-conditioning}, we only need to prove that for every $g \in \Alt(\ghostset)$ and condition $C = \len \sim \ell$, the conditioned gate $f_{\ID \times g, C}$ belongs to $\tmgroup$.

  Assume that $\sim$ is equality. By going through the divisors of $\ell$ in decreasing order, we can build any $f_{\ID \times g,\len = \ell}$ with word norm $O(\ell^2)$.\footnote{The number of divisors function $d$ satisfies $d(\ell) = o(\ell^\epsilon)$ for any $\epsilon > 0$, so we even get word norm $O(\ell^{1+\epsilon})$ for $f_{\ID \times g,\len = \ell}$.} For example, if $\ell = 6$,
  \[ f_{\ID \times g,\len=6} = f_{\ID \times g, \len \mid 1} \circ f_{\ID \times g^{-1}, \len \mid 2} \circ f_{\ID \times g^{-1}, \len \mid 3} \circ f_{\ID \times g,\len \mid 6}. \]
  (Our conveyor belts cannot actually have length $1$, so $f_{\ID \times g, \len \mid 1}$ may be dropped.)
  
  \medskip
  If $\sim$ is $\leq$, we similarly build $f_{\ID \times g,\len \leq \ell}$ with word norm $O(\ell^2)$ by going through the interval $\llbracket 1,\ell \rrbracket$ in decreasing order and picking suitable powers of $g$. For example, if $\ell = 6$, 
  \[ f_{\ID \times g,\len \leq 6} = f_{\ID \times g^{-2}, \len \mid 1} \circ f_{\ID \times g^{-1}, \len \mid 2} \circ f_{\ID \times g^0, \len \mid 3} \circ f_{\ID \times g, \len \mid 4} \circ f_{\ID \times g, \len \mid 5} \circ f_{\ID \times g,\len \mid 6}. \]
  So we obtain $f_{\ID \times g,\len \sim \ell}$ for the relations $\sim \ell$ with ${\sim} \in \{\leq,=\}$. From this, the automorphisms with ${\sim} \in \{<, \geq, >\}$ are easy to obtain with elementary boolean algebra: for example, $\len > \ell$ is the condition $\neg (\len \leq \ell)$.
\end{proof}

\begin{remark}
  One may view the above proof as an instance of M\"obius inversion. If $g$ has order $m$, take $K = \oplus_{\Z_+} \Z_m$ the commutative ring of infinitely many copies of $\Z_m$. We see $K$ as keeping track of how many times $g$ is applied at each conveyor belt length. Define functions $\iota, \gamma : \Z_+ \to K$ where $\iota(n)$ as the indicator function of $n$ (as an element of $K$), and $\gamma(n)$ the indicator function of the divisor poset of $n$. Then $\gamma(n) = \sum_{d | n} \iota(d)$ so by M\"obius inversion $\iota(n) = \sum_{d | n} \gamma(d) \mu(d, n)$ where $\mu$ is the M\"obius function of the divisibility poset; thus $\mu(d, n)$ tells us which power of $g$ we should use for each divisor to get $\iota(n)$. The values of $\iota$ are a basis of $K$, so we can get other conditional applications of $g$ with linear combinations.
\end{remark}

\subsubsection{Corollary: conditioning on the distance to borders}

Recall that $\rcbdist \in \N$ and $\lcbdist \in \N$ denote the distance between the head and the right (resp.\ left) border of the conveyor belt.
\begin{lemma}\label{lem:condition-cb-borders}
  For $t \in \N$ and ${\sim} \in \{\leq,<,=,>,\geq\}$, let $C$ be a condition of the form $\rcbdist \sim t$ (resp. $\lcbdist \sim t$).

  Then $C$ belongs to $\cbstructcond$ and $T(C) = O(t^2)$. In other words, for any $g \in \Sym(\wgh{H})$, the conditioned gate $f_{g \times \ID,C}$ belongs to $\tmgroup$ with word norm $O(t^2)$.
\end{lemma}

\begin{proof}
  Very similarly to the proof of Lemma~\ref{lem:gate-ordering-condition}, we rely on a divide and conquer approach. For example,
  \begin{align*}
    \rcbdist < t & \text{ is equivalent to } (\rcbdist < t/2+1) \; \vee \; \rho^{n/2}(\lcbdist < t/2) \\
    \rcbdist = t & \text{ is equivalent to } \neg(\rcbdist < t/2) \; \wedge \; \rho^{n/2}(\rcbdist = t/2) \\
    \rcbdist \leq t & \text{ is equivalent to } \rcbdist < t+1
  \end{align*}
  The rest of comparisons on $\rcbdist$ follows using basic Boolean algebra; and the case of $\lcbdist$ is symmetric.
\end{proof}

\subsubsection{Corollary: from cyclic tapes \texorpdfstring{$\ctape_\ell$}{C\_{}l} to conveyor belts in \texorpdfstring{$\Sigma^\Z$}{Sigma\^{}Z}}\label{sec:transport-gl-to-g}

For $T \in \tmgroup_\ell$, there exists by definition some $T_1, \dots, T_N \in \{\pi_g \mid g \in \Sym(Q \times \Gamma)\} \cup \{ \rho_q \mid q \in Q \}$ such that $T = T_N \circ \cdots \circ T_1$. Then, as each generator $\pi_g$ with $g \in \Sym(Q \times \Gamma)$ (resp.\ $\rho_q$ of $\tmgroup_\ell$) corresponds to a generator $f_g$ of $\tmgroup$ (resp. $\rho_q$ of $\tmgroup$), $T$~defines an automorphism $f_T$ of $\Aut(\Sigma^\Z)$ of the same word norm. By construction, for any $T \in \tmgroup_\ell$, the corresponding $f_T \in \tmgroup$ acts like $T$ on conveyor belts of length $\ell$ (and produces garbage on conveyor belts of length $\neq \ell$).

\begin{remark}
\label{rem:NotCanonical}
  The choice of $f_T$ is not canonical, in the sense that two different presentations $T_1,\dots,T_N$ and $T_1',\dots,T_{N'}'$ of $T$ may define two different automorphisms $f_T$ (they would define the same action on conveyor belts of length~$\ell$, but produce different garbage on conveyor belts of length~$\neq \ell$). This has no importance in what follows, but for the sake of cleanliness, fix an arbitrary order on the generators of $\tmgroup_\ell$ and define $f_T$ from the lexicographically minimal presentation (among the presentations of shortest length) of $T$.
\end{remark}

We now use the previous lemma to condition $f_T$ so that it acts only in conveyor belts of length $\ell$. Recall that $T\rfun_\dforw$ acts like $T$ on ducks $d = \dforw$, and is the identity otherwise, and denote $T_s$ the symmetrized version of $T$, i.e.
\[ T_s(w) = \begin{cases}
  T(w) & \text{if } w \in \ctape_\ell[\dforw] \\
  T^{-1}(w) & \text{if } w \in \ctape_\ell[\dback]\\
  w & \text{otherwise }
\end{cases} \]

\begin{lemma}\label{lem:from-cyclic-tapes-to-cb}
  Assume $T\rfun_\dforw$ belongs to $\tmgroup_\ell$. Then $f_{T_s,\len = \ell} \in \Aut(\Sigma^\Z)$ belongs to $\tmgroup$ with word norm $O(\lVert T\rfun_\dforw \rVert + \ell^{2})$.
\end{lemma}

\begin{proof}
  Let $d \in \Sym(Q \times \Gamma)$ be the involution that swaps $\dforw$ and $\dback$ on the head. By Lemma~\ref{lem:condition-cb-length}, $f_{d,\mathrm{len}=\ell}$ has word norm $O(\ell^{2})$ and:
  \[ f_{T_s,\mathrm{len} = \ell} = f_{d,\mathrm{len}=\ell} \circ (f_{T\rfun_\dforw})^{-1} \circ f_{d,\mathrm{len}=\ell} \circ (f_{T\rfun_\dforw}) \qedhere \]
\end{proof}

\begin{remark}
  Note that the assumption of this lemma focuses on $T\rfun_\dforw$, but the conclusion focuses on $T_s$. Note also that while $f_T$ is not canonical (see Remark~\ref{rem:NotCanonical}), $f_{T_s,\len = \ell}$ is.
\end{remark}

\subsubsection{Creating/erasing conveyor belts}\label{sec:creating-erasing-cb}

Let $\tau_\mathrm{cb} = \ID_{\wgh{H}} \times \tau' \in \Sym(H)$, where $\tau' \in \Sym(\ghostset)$ is the involution that permutes the sign $+1 \leftrightarrow -1$ in $\ghostset$ (when considering $\ghostset$ as $\ghostset \simeq \{+1,-1\} \times \llbracket 0, 2 \rrbracket$).

For $t \in \N$ define:
\begin{align*}
  f_{\tau_\mathrm{cb}, \rightarrow t} & = \theta^{-t} \circ f_{\tau_\mathrm{cb}} \circ \theta^t \\
  f_{\tau_\mathrm{cb}, t \leftarrow} & = \theta^{t} \circ f_{\tau_\mathrm{cb}} \circ \theta^{-t} \\
  f_{\tau_\mathrm{cb}, t \leftrightarrow t} & = f_{\tau_\mathrm{cb}, \rightarrow t} \circ f_{\tau_\mathrm{cb}, t \leftarrow}
\end{align*}
where $\theta \in \tmgroup_*$ (defined in Section~\ref{sec:generalized-group-of-tm}) is the right movement of heads ignoring the conveyor belt structure. These automorphisms all belong to $\tmgroup_*$ with word norm $O(t)$, and they modify the conveyor belt structure: they can create or erase conveyor belts.

In general, these modifications to the conveyor belt structure are quite unpredictable (and meaningless), so we will only apply them through conjugation (i.e.\ by conjugating automorphisms with ``small support'' by these automorphisms), so that they effectively only act in very specific situations.

\subsection{Proof of Lemma~\ref{lem:nice-implies-distortion-in-full-shift}}\label{sec:full-shift-main-proof}

Let us briefly recall the statement of Lemma~\ref{lem:nice-implies-distortion-in-full-shift}. For the precise context and notations, see Section~\ref{sec:automorphism-main-result}.

\niceimpliesdistortioninfullshift*

\begin{figure}
\begin{center}
  \begin{tikzpicture}[scale = 0.4]
  \draw (0,0) -- (12,0);
  \draw (0,1) -- (12,1);
  \node (qf1) at (6,1.5) {$q_{\dforw}$};
  \node () at (14, 0.5) {$\longleftrightarrow$};
  \node () at (14, 1.5) {$f$};
  \draw (16,0) -- (28,0);
  \draw (16,1) -- (28,1);
  \node (qb1) at (22,1.5) {$q_{\dback}$};
  
  \node () at (6, -0.5) {$i$};
  \node () at (22, -0.5) {$i$};
  \end{tikzpicture}
  
  \vspace{0.2cm}
  
  \begin{tikzpicture}[scale = 0.4]
  \draw (0,0) -- (1.5,0);
  \draw (0,1) -- (1.5,1);
  \draw (1.5,0) edge [out=0, in=0]  (1.5,1);
  \draw (2.5,0) edge [out=180, in=180]  (2.5,1);
  \draw (2.5,0) -- (9.5,0);
  \draw (2.5,1) -- (9.5,1);
  \draw (9.5,0) edge [out=0, in=0]  (9.5,1);
  \draw (10.5,0) edge [out=180, in=180]  (10.5,1);
  \draw (10.5,0) -- (12,0);
  \draw (10.5,1) -- (12,1);
  \node (qf1) at (6,1.5) {$q_{\dforw}$};
  \node () at (14, 0.5) {$\longleftrightarrow$};
  \node () at (14, 1.5) {$f^{f_1}$};
  \node () at (14, -0.5) {\shortstack{\tiny{chop at}\\\tiny{distance $2L$}}};

  \draw (16,0) -- (16+1.5,0);
  \draw (16,1) -- (16+1.5,1);
  \draw (16+1.5,0) edge [out=0, in=0]  (16+1.5,1);
  \draw (16+2.5,0) edge [out=180, in=180]  (16+2.5,1);
  \draw (16+2.5,0) -- (16+9.5,0);
  \draw (16+2.5,1) -- (16+9.5,1);
  \draw (16+9.5,0) edge [out=0, in=0]  (16+9.5,1);
  \draw (16+10.5,0) edge [out=180, in=180]  (16+10.5,1);
  \draw (16+10.5,0) -- (16+12,0);
  \draw (16+10.5,1) -- (16+12,1);
  \node (qb1) at (22,1.5) {$q_{\dback}$};
  
  \node () at (2, -0.5) {$i-2L$};
  \node () at (10, -0.5) {$i+2L$};
  \node () at (18, -0.5) {$i-2L$};
  \node () at (26, -0.5) {$i+2L$};
  \end{tikzpicture}
  
  \vspace{-0.2cm}
  
  \hspace*{0.45cm}\begin{tikzpicture}[scale = 0.4]
  
  \draw[->] (6,1+3) --  (6.25,1+2.8) -- (6.5,1+2.6) -- (6.75,1+2.4) -- (6.5,1+2.2) -- (6.25,1+2) --
  (6,1+1.8) -- (5.75,1+1.6) -- (5.5,1+1.4) -- (5.25,1+1.2); 
  \node[right, align=left] () at (6.5,3) {\shortstack{\tiny{Machine runs}\\\tiny{backward in time.}}};

  \draw[<-] (17+6,1+1.2) --  (17+6.25,1+1.4) -- (17+6.5,1+1.6) -- (17+6.75,1+1.8) -- (17+6.5,1+2) -- (17+6.25,1+2.2) --
  (17+6,1+2.4) -- (17+5.75,1+2.6) -- (17+5.5,1+2.8) -- (17+5.25,1+3); 
  \node[right, align=left] () at (17+7,2.8) {\shortstack{\tiny{Machine runs}\\\tiny{forward in time.}}};
  
  \draw (0,0) -- (1.5,0);
  \draw (0,1) -- (1.5,1);
  \draw (1.5,0) edge [out=0, in=0]  (1.5,1);
  \draw (2.5,0) edge [out=180, in=180]  (2.5,1);
  \draw (2.5,0) -- (9.5,0);
  \draw (2.5,1) -- (9.5,1);
  \draw (9.5,0) edge [out=0, in=0]  (9.5,1);
  \draw (10.5,0) edge [out=180, in=180]  (10.5,1);
  \draw (10.5,0) -- (12,0);
  \draw (10.5,1) -- (12,1);
  \node (qf1) at (5,1.65) {$q^-_{\dforw}$};
  \node () at (14, 0.5) {$\longleftrightarrow$};
  \node () at (14, -0.5) {\shortstack{\tiny{move}}};
  \node () at (14, 1.5) {$f^{f_2}$};
  \draw (16,0) -- (16+1.5,0);
  \draw (16,1) -- (16+1.5,1);
  \draw (16+1.5,0) edge [out=0, in=0]  (16+1.5,1);
  \draw (16+2.5,0) edge [out=180, in=180]  (16+2.5,1);
  \draw (16+2.5,0) -- (16+9.5,0);
  \draw (16+2.5,1) -- (16+9.5,1);
  \draw (16+9.5,0) edge [out=0, in=0]  (16+9.5,1);
  \draw (16+10.5,0) edge [out=180, in=180]  (16+10.5,1);
  \draw (16+10.5,0) -- (16+12,0);
  \draw (16+10.5,1) -- (16+12,1);
  \node (qb1) at (23,1.65) {$q^+_{\dback}$};
  
  \node () at (5, -0.5) {$i-t$};
  \node () at (23, -0.5) {$i+t'$}; 
  \end{tikzpicture}
    \vspace{0.2cm}
  
  \begin{tikzpicture}[scale = 0.4]
  
  \draw (0,0) -- (1.5,0);
  \draw (0,1) -- (1.5,1);
  \draw (1.5,0) edge [out=0, in=0]  (1.5,1);
  \draw (2.5,0) edge [out=180, in=180]  (2.5,1);
  \draw (2.5,0) edge [out=0, in=0]  (2.5,1);
  \draw (3.5,0) edge [out=180, in=180]  (3.5,1);
  \draw (3.5,0) -- (6.5,0);
  \draw (3.5,1) -- (6.5,1);
  \draw (6.5,0) edge [out=0, in=0]  (6.5,1);
  \draw (7.5,0) edge [out=180, in=180]  (7.5,1);
  \draw (7.5,0) -- (9.5,0);
  \draw (7.5,1) -- (9.5,1);
  \draw (9.5,0) edge [out=0, in=0]  (9.5,1);
  \draw (10.5,0) edge [out=180, in=180]  (10.5,1);
  \draw (10.5,0) -- (12,0);
  \draw (10.5,1) -- (12,1);
  \node (qf1) at (5,1.65) {$q^-_{\dforw}$};
  \node () at (14, 0.5) {$\longleftrightarrow$};
  \node () at (14, 1.5) {$f^{f_3}$};
  \node () at (14, -0.5) {\shortstack{\tiny{chop at}\\\tiny{distance $L$}}};
  \draw (16,0) -- (16+1.5,0);
  \draw (16+0,1) -- (16+1.5,1);
  \draw (16+1.5,0) edge [out=0, in=0]  (16+1.5,1);
  \draw (16+2.5,0) edge [out=180, in=180]  (16+2.5,1);
  \draw (16+2.5,0) -- (16+4.5,0);
  \draw (16+2.5,1) -- (16+4.5,1);
  \draw (16+4.5,0) edge [out=0, in=0]  (16+4.5,1);
  \draw (16+5.5,0) edge [out=180, in=180]  (16+5.5,1);
  \draw (16+5.5,0) -- (16+8.5,0);
  \draw (16+5.5,1) -- (16+8.5,1);
  \draw (16+8.5,0) edge [out=0, in=0]  (16+8.5,1);
  \draw (16+9.5,0) edge [out=180, in=180]  (16+9.5,1);
  \draw (16+9.5,0) edge [out=0, in=0]  (16+9.5,1);
  \draw (16+10.5,0) edge [out=180, in=180]  (16+10.5,1);
  \draw (16+10.5,0) -- (16+12,0);
  \draw (16+10.5,1) -- (16+12,1);
  \node (qb1) at (22.75,1.65) {$q^+_{\dback}$};
  
  \node () at (3, -0.5) {$i-t-L$};
  \node () at (7, -0.5) {$i-t+L$};
  \node () at (21, -0.5) {$i+t'-L$};
  \node () at (25, -0.5) {$i+t'+L$};
  \end{tikzpicture}
  
  \hspace*{0.35cm}\begin{tikzpicture}[scale = 0.4]
  \draw[<-] (6,1+1.2) --  (6.25,1+1.4) -- (6.5,1+1.6) -- (6.75,1+1.8) -- (6.5,1+2) -- (6.25,1+2.2) --
  (6,1+2.4) -- (5.75,1+2.6) -- (5.5,1+2.8) -- (5.25,1+3); 
  \node[right, align=left] () at (7,2.8) {\shortstack{\tiny{Machine runs}\\\tiny{forward in time.}}};
  \draw[->] (17+6,1+3) --  (17+6.25,1+2.8) -- (17+6.5,1+2.6) -- (17+6.75,1+2.4) -- (17+6.5,1+2.2) -- (17+6.25,1+2) --
  (17+6,1+1.8) -- (17+5.75,1+1.6) -- (17+5.5,1+1.4) -- (17+5.25,1+1.2); 
  \node[right, align=left] () at (17+6.5,3) {\shortstack{\tiny{Machine runs}\\\tiny{backward in time.}}};
  
  \draw (0,0) -- (1.5,0);
  \draw (0,1) -- (1.5,1);
  \draw (1.5,0) edge [out=0, in=0]  (1.5,1);
  \draw (2.5,0) edge [out=180, in=180]  (2.5,1);
  \draw (2.5,0) edge [out=0, in=0]  (2.5,1);
  \draw (3.5,0) edge [out=180, in=180]  (3.5,1);
  \draw (3.5,0) -- (6.5,0);
  \draw (3.5,1) -- (6.5,1);
  \draw (6.5,0) edge [out=0, in=0]  (6.5,1);
  \draw (7.5,0) edge [out=180, in=180]  (7.5,1);
  \draw (7.5,0) -- (9.5,0);
  \draw (7.5,1) -- (9.5,1);
  \draw (9.5,0) edge [out=0, in=0]  (9.5,1);
  \draw (10.5,0) edge [out=180, in=180]  (10.5,1);
  \draw (10.5,0) -- (12,0);
  \draw (10.5,1) -- (12,1);
  \node (qf1) at (6,1.5) {$q_{\dforw}$};
  \node () at (14, 0.5) {$\longleftrightarrow$};
  \node () at (14, 1.5) {$f^{f_4}$};
  \node () at (14, -0.5) {\shortstack{\tiny{move}}};
  
  \draw (16,0) -- (16+1.5,0);
  \draw (16+0,1) -- (16+1.5,1);
  \draw (16+1.5,0) edge [out=0, in=0]  (16+1.5,1);
  \draw (16+2.5,0) edge [out=180, in=180]  (16+2.5,1);
  \draw (16+2.5,0) -- (16+4.5,0);
  \draw (16+2.5,1) -- (16+4.5,1);
  \draw (16+4.5,0) edge [out=0, in=0]  (16+4.5,1);
  \draw (16+5.5,0) edge [out=180, in=180]  (16+5.5,1);
  \draw (16+5.5,0) -- (16+8.5,0);
  \draw (16+5.5,1) -- (16+8.5,1);
  \draw (16+8.5,0) edge [out=0, in=0]  (16+8.5,1);
  \draw (16+9.5,0) edge [out=180, in=180]  (16+9.5,1);
  \draw (16+9.5,0) edge [out=0, in=0]  (16+9.5,1);
  \draw (16+10.5,0) edge [out=180, in=180]  (16+10.5,1);
  \draw (16+10.5,0) -- (16+12,0);
  \draw (16+10.5,1) -- (16+12,1);
  \node (qb1) at (22,1.5) {$q_{\dback}$};
  \end{tikzpicture}
  
  \vspace{0.2cm}
  \hspace*{0.14cm}\begin{tikzpicture}[scale = 0.4]  
  \draw (0,0) -- (2.5,0);
  \draw (0,1) -- (2.5,1);
  \draw (2.5,0) edge [out=0, in=0]  (2.5,1);
  \draw (3.5,0) edge [out=180, in=180]  (3.5,1);
  \draw (3.5,0) -- (6.5,0);
  \draw (3.5,1) -- (6.5,1);
  \draw (6.5,0) edge [out=0, in=0]  (6.5,1);
  \draw (7.5,0) edge [out=180, in=180]  (7.5,1);
  \draw (7.5,0) -- (12,0);
  \draw (7.5,1) -- (12,1);
  \node (qf1) at (6,1.5) {$q_{\dforw}$};
  \node () at (14, 0.5) {$\longleftrightarrow$};
  \node () at (14, 1.5) {$f^{f_5}$};
  \node () at (14, -0.5) {\shortstack{\tiny{glue at}\\\tiny{distance $2L$}}};
  \draw (16,0) -- (16+4.5,0);
  \draw (16+0,1) -- (16+4.5,1);
  \draw (16+4.5,0) edge [out=0, in=0]  (16+4.5,1);
  \draw (16+5.5,0) edge [out=180, in=180]  (16+5.5,1);
  \draw (16+5.5,0) -- (16+8.5,0);
  \draw (16+5.5,1) -- (16+8.5,1);
  \draw (16+8.5,0) edge [out=0, in=0]  (16+8.5,1);
  \draw (16+9.5,0) edge [out=180, in=180]  (16+9.5,1);
  \draw (16+9.5,0) -- (16+12,0);
  \draw (16+9.5,1) -- (16+12,1);
  \node (qb1) at (22,1.5) {$q_{\dback}$};
  \end{tikzpicture}
  \vspace{-0.2cm}
  
  \hspace*{0.35cm}\begin{tikzpicture}[scale = 0.4]  
  \draw[->] (6,1+3) --  (6.25,1+2.8) -- (6.5,1+2.6) -- (6.75,1+2.4) -- (6.5,1+2.2) -- (6.25,1+2) --
  (6,1+1.8) -- (5.75,1+1.6) -- (5.5,1+1.4) -- (5.25,1+1.2); 
  \node[right, align=left] () at (6.5,3) {\shortstack{\tiny{Machine runs}\\\tiny{backward in time.}}};
  \draw[<-] (17+6,1+1.2) --  (17+6.25,1+1.4) -- (17+6.5,1+1.6) -- (17+6.75,1+1.8) -- (17+6.5,1+2) -- (17+6.25,1+2.2) --
  (17+6,1+2.4) -- (17+5.75,1+2.6) -- (17+5.5,1+2.8) -- (17+5.25,1+3); 
  \node[right, align=left] () at (17+7,2.8) {\shortstack{\tiny{Machine runs}\\\tiny{forward in time.}}};
  \draw (0,0) -- (2.5,0);
  \draw (0,1) -- (2.5,1);
  \draw (2.5,0) edge [out=0, in=0]  (2.5,1);
  \draw (3.5,0) edge [out=180, in=180]  (3.5,1);
  \draw (3.5,0) -- (6.5,0);
  \draw (3.5,1) -- (6.5,1);
  \draw (6.5,0) edge [out=0, in=0]  (6.5,1);
  \draw (7.5,0) edge [out=180, in=180]  (7.5,1);
  \draw (7.5,0) -- (12,0);
  \draw (7.5,1) -- (12,1);
  \node (qf1) at (5,1.65) {$q^-_{\dforw}$};
  \node () at (14, 0.5) {$\longleftrightarrow$};
  \node () at (14, 1.5) {$f^{f_6}$};
  \node () at (14, -0.5) {\shortstack{\tiny{move}}};
  \draw (16,0) -- (16+4.5,0);
  \draw (16+0,1) -- (16+4.5,1);
  \draw (16+4.5,0) edge [out=0, in=0]  (16+4.5,1);
  \draw (16+5.5,0) edge [out=180, in=180]  (16+5.5,1);
  \draw (16+5.5,0) -- (16+8.5,0);
  \draw (16+5.5,1) -- (16+8.5,1);
  \draw (16+8.5,0) edge [out=0, in=0]  (16+8.5,1);
  \draw (16+9.5,0) edge [out=180, in=180]  (16+9.5,1);
  \draw (16+9.5,0) -- (16+12,0);
  \draw (16+9.5,1) -- (16+12,1);
  \node (qb1) at (23,1.65) {$q^+_{\dback}$};
  
  \node () at (5, -0.5) {$i-t$};
  \node () at (23, -0.5) {$i+t'$};
  \end{tikzpicture}
  
  \vspace{0.2cm}
  
  \begin{tikzpicture}[scale = 0.4]
  \draw (0,0) -- (12,0);
  \draw (0,1) -- (12,1);
  \node (qf1) at (5,1.65) {$q^-_{\dforw}$};
  \node () at (14, 0.5) {$\longleftrightarrow$};
  \node () at (14, 1.5) {$f^{f_7}$};
  \node () at (14, -0.5) {\shortstack{\tiny{glue at}\\\tiny{distance $L$}}};
  
  \draw (16,0) -- (28,0);
  \draw (16+0,1) -- (28,1);
  \node (qb1) at (23,1.65) {$q^+_{\dback}$};
  
  \node () at (5, -0.5) {$i-t$};
  \node () at (23, -0.5) {$i+t'$};
  \end{tikzpicture}
  \end{center}
  
  \caption{The ``two-scale trick''}
  \captionsetup{width=.95\linewidth,font=small}
  \caption*{Illustration of the two-scale trick. We show the conveyor belts as lines (without tape letters). The head may be on either track. The head is in position $i$ in state $q$ initially, was in position $i-t$ in state $q^-$ at time $-n$, and will be in position $i+t'$ in state $q^+$ at time $+n$. The automorphism $f$ acts trivially unless we are in the situation of the first line, so the conjugated automorphisms also act non-trivially only in the shown situation. In particular $f^{f_7}$ behaves as expected.}
  \label{fig:two-scale-trick}
\end{figure}

Before going into the precise math, here is an overview of the proof. The main idea consists in building $(f_{\mathcal{M}})^n$ (the action of $\mathcal{M}$ on the conveyor belts of~$\Sigma^\Z$) from the automorphisms $(T_{\ell,\mathcal{M}_0})^n$ (each action of $\mathcal{M}_0$ on each length $\ell$ of finite cyclic tapes~$\ctape_\ell$), the latter being finitarily distorted by hypothesis (i.e.\ having small word norm $O(\ell^p)$). Conditioning each $(T_{\ell,\mathcal{M}_0})^n$ on the correct length of conveyor belt with Lemma~\ref{lem:from-cyclic-tapes-to-cb}, we can use these automorphisms and the distortion carries from finite cyclic tapes to conveyor belts.

Informally, $(f_\mathcal{M})^n$ then becomes the product of infinitely many $(T_{\ell,\mathcal{M}})^n$. In order to write $(f_\mathcal{M})^n$ as a finite product, consider the movement of $\mathcal{M}$ and notice that $m(n)$ -- the number of cells visited by $\mathcal{M}$ in $n$ steps -- is assumed to be $O(\log n$). As a consequence, a head at distance more than $m(n)$ from the borders of its conveyor belts will not see the borders in question. In such a case, we can create temporary conveyor belts of size $O(\log n)$, apply the corresponding $(T_{\ell,\mathcal{M}})^n$, and erase the temporary borders: the applied operation coincides with $(f_{\mathcal{M}})^n$ in the original large conveyor belt.

\bigskip
Let us formalize these ideas. In particular, generating temporary conveyor belts is not a reversible operation: to solve this issue, we introduce the \emph{two-scale trick} (which introduces not one, but two temporary conveyor belts of different sizes, hence its name).

\begin{proof}[Proof of Lemma~\ref{lem:nice-implies-distortion-in-full-shift}]
  Fix an integer, which is the power of $f_{\mathcal{M}}$ that we want to build. Without any loss of generality, we assume that it is even. Indeed, if some $2n+1$ is odd, then $2n$ is even, and $\lVert (f_{\mathcal{M}})^{2n+1} \rVert = \lVert (f_{\mathcal{M}})^{2n} \rVert + O(1)$. We denote this even integer by $2n$. With notations from Lemma~\ref{lem:smart-distorted-on-cyclic-tapes}, let $L = C \cdot \log n + C'$. By hypothesis, every $(T_{\ell,\mathcal{M}_0})^{n}$ has word norm $O(\ell^p)$ in $\tmgroup_{\ell}$.
  
  \medskip
  First, we use Lemma~\ref{lem:from-cyclic-tapes-to-cb} to condition each $(T_{\ell,\mathcal{M}_0})^{2n}$ to apply only on conveyor belts of length $\ell$. In other words, since $\mathcal{M}$ is the symmetrized version of $\mathcal{M}_0$, we obtain that every automorphism 
  \[ (f_{(T_{\ell,\mathcal{M}}),\len = \ell})^{2n} = f_{(T_{\ell,\mathcal{M}})^{2n}, \len = \ell}  \]
  belongs to $\tmgroup$ with word norm $O(\ell^p)$, so that we can manage all conveyor belts of length $< 12 L$ with word norm $O(L \cdot L^p)$:
  \[ (f_{\mathcal{M},\len < 12L})^{2n} = \prod_{\ell=1}^{6L-1} ({f_{(T_{2\ell,\mathcal{M}}),\len = 2\ell}})^{2n} \]
  (recall that conveyor belts in $\Sigma^\Z$ are cyclic tapes of even length).

  \medskip
  Then, to manage larger conveyor belts, we use what we call the \emph{two-scale trick}. In the introductory paragraphs, we mentioned the idea of introducing temporary conveyor belts to move heads that originally belonged in very large conveyor belts, and then removing the temporary borders. The difficulty lies in properly removing the temporary conveyor belts once the machine has been applied. To solve this, we will actually use temporary conveyor belts twice, with different sizes. We give a visual explanation of this trick in Figure~\ref{fig:two-scale-trick}. 

  Define $L_1 = 4L - 2$, $L_2 = 8L-2$, $L_3 = 12L-2$. Note that $L_1$ (resp.\ $L_2$) is the length of a conveyor belt constructed by $f_{\tau_\mathrm{cb},L \leftrightarrow L}$ (resp.\ $f_{\tau_\mathrm{cb},2L \leftrightarrow 2L}$) defined in Section~\ref{sec:creating-erasing-cb}. Then, let:
  \[ \begin{split}
    \lambda_{n} = \Big(f_{\tau_\mathrm{cb},2L \leftrightarrow 2L} \circ \Big(f_{(T_{L_2,\mathcal{M}})^{n},\len = L_2}\Big) \circ f_{\tau_\mathrm{cb},L \leftrightarrow L} \circ \Big(f_{(T_{L_1,\mathcal{M}})^{n},\len = L_1}\Big)^{-1} \\
    \circ f_{\tau_\mathrm{cb},2L \leftrightarrow 2L} \circ \Big(f_{(T_{L_1,\mathcal{M}})^{n},\len = L_1}\Big) \circ f_{\tau_\mathrm{cb},L \leftrightarrow L} \Big).
  \end{split} \]
  Let $f_i$ denote the composition of the first $i$ automorphisms on this list, i.e.\ $f_1 = f_{\mathrm{cb},2L \leftrightarrow 2L}, \ldots, f_7 = \lambda_{n}$. The actions of the inverses of the automorphisms $f_i$ are illustrated in Figure~\ref{fig:two-scale-trick} (on a certain subset of configurations).
 
  Then, denote $d \in \Sym(H)$ the ducking involution, i.e.\ the involution that flips ducks $\dforw$ and $\dback$ in $Q = Q_0 \times \duckset \times \ghostset$. Let $C \in \cbstructcond$ be the structural condition $(\len \geq 12L) \wedge (\lcbdist \geq 3L) \wedge (\rcbdist \geq 3L)$ (see Section~\ref{sec:struct-conditioning}). By Lemmas~\ref{lem:condition-cb-length}, \ref{lem:condition-cb-borders} and \ref{lem:condition-cb-struct}, the automorphism $f_{d,C}$ belongs to $\tmgroup$ with word norm $O(L^2)$. Defining:
  \[ f_{\mathcal{M},2n,C} = \left({\lambda_{n}}\right)^{-1} \circ \big( f_{d,C} \big) \circ \left({\lambda_{n}}\right), \]
  we see from Figure~\ref{fig:two-scale-trick}, reading the successive partial conjugations ${f_{d,C}}^{f_i}$ top-down, that $f_{d,C}$ gets conjugated to a map that applies the machine $(\mathcal{M})^{n}$ twice if it is on a conveyor belt of length $\geq 12L$ and the head is sufficiently far (i.e.\ at distance at least $3L$) from both its left and right borders; and flips the duck as a side product.
  
  \medskip
  Denoting $C_\mathrm{r}$ the condition $(\len \geq 12L) \wedge (\lcbdist \geq 3L) \wedge (\rcbdist < 3L)$ and $C_\mathrm{l}$ the condition $(\len \geq 12L) \wedge (\lcbdist < 3L) \wedge (\rcbdist \geq 3L)$, we can build automorphisms $f_{\mathcal{M},2n,C_\mathrm{r}}$ (resp.\ $f_{\mathcal{M},2n,C_\mathrm{l}}$) that manage heads in conveyor belts of size $\geq 12L$ containing heads at distance less than $3L$ from their right (resp.\ left) border.
  
  By a very similar reasoning, $f_{\mathcal{M},2n,C_\mathrm{r}}$ and $f_{\mathcal{M},2n,C_\mathrm{l}}$ belong to $\tmgroup_*$ with word norm $O(L^{p+1} + L^2)$. We need to alter this reasonning only twice: first, instead of using $f_{\tau_\mathrm{cb}, \ell \leftrightarrow \ell}$ to create/erase borders both left and right, we use respectively $f_{\tau_\mathrm{cb}, \ell \leftarrow}$ and $f_{\tau_\mathrm{cb}, \rightarrow \ell}$ to create/erase borders only in one direction; second, as the size of the conveyor belt is no longer precisely $L_1$ or $L_2$, but only bounded by $L_1$ or $L_2$, we have to replace occurrences of $(f_{(T_{\ell,\mathcal{M}})^{n},\len = \ell})$ (for $\ell = L_1$ or $\ell = L_2$) in the previous formulas with 
  \[ \prod_{j=1}^\ell f_{(T_{j,\mathcal{M}})^{n},\len = j}. \]
  
  \medskip
  In any case, these permutations have disjoint support (because the distance to a conveyor belt border is checked ``after $n$ steps of computation'' in the conjugations ${\lambda_{n}}^{-1} \circ \big( f_{d,C} \big) \circ {\lambda_{n}}$) and word norm $O(L^{p+1} + L^2)$. We obtain that:
  \[ (f_{\mathcal{M}})^{2n} = \Big(f_{d,\len \geq L_3} \circ f_{\mathcal{M},2n,C_\mathrm{r}} \circ f_{\mathcal{M},2n,C_\mathrm{l}} \circ f_{\mathcal{M},2n,C} \Big) \circ \left(\prod_{n = 1}^{6L-1} ({f_{{T_{\ell,\mathcal{M}}},\mathrm{len} = 2l}})^{2n} \right) \]
  which concludes the proof.
\end{proof}

\subsection{Improving the upper bound on SMART}\label{sec:full-shift-optimize-degree}

Lemma~\ref{lem:nice-implies-distortion-in-full-shift} was a general result telling how finitely distorted machines give rise to distorted automorphisms on full-shifts. In the context of the SMART machine, we provide some additional improvements in order to get the upper bound from $O(\log^5 n)$ to $O(\log^4 n)$.

These improvements are based on the following idea: the computation of the $(T_{\ell,\smart})^n$ is ``uniform'' in $\ell$, as the $k$-th step of the encoding (from Lemma~\ref{lem:encoding-with-garbage}) and the $k$-th step of the addition (from Lemma~\ref{lem:addition-in-encoded-base}) are the same on all conveyor belts of length~$\geq k$. Thus, we can compute the action of SMART on all conveyor belts in parallel, improving the word norm as mentioned. We also make minor optimizations to the alphabet, by joining the decorations used for the decorated SMART and the decorated automorphism.

All in all, Lemma~\ref{lem:better-bounds-for-smart} concludes the proof of Theorem~\ref{thm:distortion-every-full-shift}. It also provides a self-contained result about distortion.

\begin{lemma}\label{lem:better-bounds-for-smart}
  Let $\smart = (\origsmart{\Gamma},\origsmart{Q},\origsmart{\Delta})$ be the original SMART machine introduced in Section~\ref{sec:smart}, and define its decorated symmetrized version $\smart = (\Gamma,Q,\Delta)$ as follows:
  \begin{align*}
    \Gamma =\ & \origsmart{\Gamma} \\
    Q =\ & \origsmart{Q} \times \duckset \times \ghostset \\
    \Delta =\ & \bigcup_{x \in \ghostset} \left\{ \Big( (q,\dforw,x),a,(q',\dforw,x),b \Big) : (q,a,q',b) \in \origsmart{\Delta} \right\} \\
      & \qquad \cup \left\{ \Big((q,\dforw,x),\delta,(q',\dforw,x)\Big) : (q,\delta,q') \in \origsmart{\Delta} \right\} \\
      & \bigcup_{x \in \ghostset} \left\{ \Big( (q',\dback,x),b,(q,\dback,x),a \Big) : (q,a,q',b) \in \origsmart{\Delta} \right\} \\
      & \qquad \cup \left\{ \Big((q',\dback,x),\delta,(q,\dback,x)\Big) : (q,\delta,q') \in \origsmart{\Delta} \right\}
  \end{align*}
  where $\duckset = \{\dforw,\dback\}$ and $\ghostset = \llbracket 0,5 \rrbracket \simeq \{+1,-1\} \times \llbracket 0,2 \rrbracket$; and denote
  \[ \Sigma = \Big( \Gamma^2 \times \{+1,-1\} \Big) \sqcup \Big( (Q \times \Gamma) \times \Gamma \Big) \sqcup \Big( \Gamma \times (Q \times \Gamma) \Big) \]
  Then $(f_{\smart})^n$ has word norm $O(\log^4 n)$ in a finitely-generated subgroup $\tmgroup_*$ of $\Aut(\Sigma^\Z)$.
\end{lemma}

\begin{proof}
  Once again, we prove that $(f_\smart)^n$ has word norm $O(\log^4 n)$ in the following subgroup of $\Aut(\Sigma^\Z)$
  \begin{align*}
  \tmgroup_* = \langle & \{ f_{g}^\aup, f_{g}^\adown, f_g : g \in \Sym(Q \times \Gamma) \} \\
  & \cup \{ \rho_q \mid q \in Q \} \cup \{ \theta \} \rangle
  \end{align*}
  
  We start by explaining the new alphabet. We have dropped the duck $\{\dright,\dleft\}$ to fuse it with $\{\dforw,\dback\}$. Our assumption in the previous section is that we can efficiently apply the Turing machine when the duck is $\dforw$, while doing nothing on ducks $\dback$, and this is exactly how the duck was used in Section~\ref{sec:Finitary} (see Lemma~\ref{lem:smart-one-duck}).

  We also have fused the two ghosts: as both ghosts were only used temporarily in our construction, always returning to their original value after each step of computation, it is safe to fuse them.

  \medskip
  Now we explain the optimization; we only give a high-level explanation, as this is completely analogous to what was done in the previous section. We only amend the proof above by proving that both $(f_{\mathcal{M},\len < L_3})^{2n}$ and $(f_{\mathcal{M},\len \leq L_3})^{2n}$ have word norm $O(L^4)$, as they are sufficient to manage both small conveyor belts of length $< L_3$ and large conveyor belts in the two-scale trick.

  \begin{enumerate}
    \item The encoding (word norm $O(L^4)$). Let $\GRpencoding$ being either $\GRpencoding_{\mathrm{init}}$, $\GRpencoding_{k \to k+1}$ or $\GRpencoding_{\ell,\final}$, the steps of encoding of $\smart$ defined in Lemma~\ref{lem:encoding-steps-in-group}. Recall that
    \[ \GRpencoding(w) = \begin{cases}
      \pencoding(w) & \text{if } w \in \ctape_\ell[\dforw] \\
      \pencoding^{-1}(w) & \text{if } w \in \ctape_\ell[\dback] \\
      w & \text{otherwise}
    \end{cases} \]

    As explained in Section~\ref{sec:transport-gl-to-g}, the generators of $\tmgroup_\ell$ correspond to generators of $\tmgroup_*$, so that $\GRpencoding$ can also be considered as an element $f_{\GRpencoding}$ of $\tmgroup_*$, which acts like $\GRpencoding$ on conveyor belts of length $\ell$, and produces garbage on conveyor-belts of length $\neq \ell$.

    The key point of this proof consists in noticing that $\GRpencoding_{k \to k+1}$ behaves correctly on every conveyor belt of length $\geq k+2$, and produces garbage on conveyor belts of length $\leq k+1$.

    \medskip
    Recall that each $\GRpencoding$ was a piecewise-defined bijection defined as a product of finitely many $\GRpencoding_{\duckdomain,\duckimage}$. Each $\GRpencoding_{\duckdomain,\duckimage}$ was built in Lemma~\ref{lem:encoding-case-in-group} by conjugating some conditioned gate $\pi_{d,\duckdomain}$ ($d \in \Sym(H)$ be the ducking involution, which swaps ducks $\dforw$ and $\dback$) with some automorphism $g \in \tmgroup_\ell$. 

    Combining Lemmas~\ref{lem:gate-conditioning} and~\ref{sec:struct-conditioning}, we now condition simultaneously on the content of the tape (i.e.\ the condition $\duckdomain$ in the proof of Lemma~\ref{lem:encoding-steps-in-group}) and the structure of conveyor belts. In other words, for any structural condition $C \in \cbstructcond$, the automorphism $f_{d,\duckdomain \wedge C}$ belongs to $\tmgroup_*$ and we can define
    \[ f_{\GRpencoding_{\duckdomain,\duckimage},C} = f_{g^{-1}} \circ f_{d,\duckdomain \wedge C} \circ f_{g} \]
    (where $f_g \in \tmgroup_*$ denotes the automorphism $g \in \tmgroup_\ell$ considered as an element of $\tmgroup_*$).
    
    Then $f_{\GRpencoding_{\duckdomain,\duckimage},C}$ acts like $\GRpencoding_{\duckdomain,\duckimage}$ on conveyor belts which respect the structural condition $C$. And since the word norm of $\GRpencoding_{\duckdomain,\duckimage}$ was already the word norm of $g \in \tmgroup_\ell$, i.e.\ $O(\ell^3)$, we obtain that $f_{\GRpencoding_{\duckdomain,\duckimage},C}$ also has word norm $O(\ell^3 + T(C))$. Composing finitely many of those, and adding a final $f_{d,C}$, we obtain automorphisms $f_{\GRpencoding,C}$ for $\GRpencoding$ being either $\GRpencoding_{\mathrm{init}}$, $\GRpencoding_{k \to k+1}$ or $\GRpencoding_{\ell,\final}$.

    \medskip
    Now, we consider the following automorphisms (for $0 \leq k < L_3 - 2$):
    \[
      f_{\GRpencoding_{\mathrm{init}},\mathrm{\len} < L_3}, \quad 
      f_{\GRpencoding_{k \to k+1},k+2 \leq \mathrm{\len} < L_3}, \quad
      f_{\GRpencoding_{\ell,\mathrm{final}}, \mathrm{\len} = \ell}
    \]
    Each of these elements has word norm $O(L_3^2 + L_3^3)$. Then, define:
    \[ f_{\encode} = \left( \prod_{\ell=1}^{L_3-1} f_{\GRpencoding_{\ell,\mathrm{final}}, \mathrm{\len} = \ell} \right) \circ \left( \prod_{k=1}^{L_3-3} f_{\GRpencoding_{k \to k+1},k+2 \leq \mathrm{\len} < L_3} \right) \circ f_{\GRpencoding_{\mathrm{init}},\mathrm{\len} < L_3} \]

    The automorphism $f_{\encode}$, which acts on conveyor belts of length $< L_3$ and encodes SMART configurations with ducks $\dforw$ into their correct encoding, and produces garbage on ducks $\dback$, has word norm $O(L^4)$. (A similar automorphism obviously exists for conveyor belts of length $\leq L_3$).
  
    \item The addition of $2n$ on ducks $\dforw$ (word norm $O(L^3)$). We follow the proof of Lemma~\ref{lem:addition-in-encoded-base}, which applied the ``school algorithm'' for additions. The key idea of this proof is that the $j$-th digit of $2n$ can be applied with the same automorphism to every conveyor belt of length $\geq j$.
    
    \medskip
    As in the proof of Lemma~\ref{lem:addition-in-encoded-base}, we first manage the addition of $2n$ modulo $2 \cdot 3^\ell$ in every conveyor belt of length $\ell < L_3 = 12L-2$ with the school-like algorithm. To do so, we rely on conditions from Lemma~\ref{lem:condition-cb-length}. To fix notations, for every $\ell < L_3$, let $k[\ell] \in \{1,2\} \times \{0,1,2\}^\ell$ encode $2n \bmod 2 \cdot 3^\ell$.
    
    Let $\rho_{\dforw} = \prod_{q \in \origsmart{Q} \times \{\dforw\} \times \ghostset} \rho_q$ move heads with duck $\dforw$ right. Still denoting $d \in \Sym(H)$ the ducking involution and $C$ a condition, the automorphism $f_{d,C} \circ \rho_{\dforw}^{-1} \circ f_{d,C} \circ \rho_{\dforw}$ moves head in conveyor belts that verify condition $C$, and heads with duck $\dforw$ move to the right (while heads with duck $\dback$ move left). Denote this automorphism $\rho_{\dforw,C}$.
    
    We now proceed as follows. First, apply $(\rho_{\dforw,\len < L_3})^{-1}$ to move heads with duck $\dforw$ left, i.e.\ as to go to the last digit of the counter to which we want to add $2n$. Denote $r' = (0,1,2) \in \Sym(\Gamma)$ be the addition of a single digit, and $r_\dforw = \ID \times \{\dforw\} \times \ID \times r' \in \Sym(H)$. Then, for $0 \leq \ell < L_3$, we do the following:
    \begin{enumerate}
      \item Add the second digit of $k[\ell]$ on the tape in conveyor belts of length $\ell \leq \len < L_3$: apply $f_{r_\dforw,\ell \leq \len < L_3}$ if $k[\ell]_1 = 1$, or $f_{r_\dforw,\ell \leq \len < L_3}$ if $k[\ell]_1 = 2$.
      \item Perform the carry in conveyor belts of length $\ell \leq \len < L_3$. Let $k' = k[\ell]_{\llbracket 2,|k|-1 \rrbracket}$ be the digits of $k$ we want to check an overflow for, and apply $f_{r_\dforw,(\ell \leq \len < L_3) \wedge (< \bullet k')}$: this applies the carry if the addition on the cells on the right overflowed.
      \item Add the first digit of $k[\ell]$ (i.e.\ $k[\ell]_0 \in \{1,2\}$) in conveyor belts of length $\len = \ell-1$. Recall that $\origsmart{Q} = \{\smb,\smd,\smp,\smq\} \times \{1,2\}$: denoting $b' = (q,1) \leftrightarrow (q,2) \in \Sym(\origsmart{Q})$ and $b_\dforw = b' \times \{\dforw\} \times \ID \times \ID \in \Sym(H)$, apply $f_{b_\dforw,\len = \ell-1}$ if $k[\ell]_0 = 2$, and the identity otherwise.
      \item Perform the carry in conveyor belts of length $\len = \ell-1$: apply $f_{b_\dforw,(\len = \ell-1) \wedge (< k[\ell]_{\llbracket 1, |k[\ell]|-1 \rrbracket})}$, which applies the carry in the state if the addition on the tape overflowed.
      \item Apply $(\rho_{\dforw,\ell \leq \len < L_3})^{-1}$ to move heads with duck $\dforw$ left (and heads with duck $\dback$ right).
    \end{enumerate}

    These steps apply the correct addition on heads having duck $d = \dforw$. On duck $d = \dback$, the head simply moves to the right at each step, not modifying anything, until it gets back to its initial positions. All these steps together have word norm $O(L^3)$.
    
    We are left with shifting the tape $a$ times (left or right, depending on the state $\{\smb,\smd,\smp,\smq\}$), where $a = \lfloor 2n / (2 \cdot 3^\ell) \rfloor$; and apply a final shift if the addition in the previous paragraph overflowed. We do so for each length $0 \leq \ell < L_3$ independently.
    
    Let $s \in \Sym(H)$ be the involution that exchanges states $\smb \leftrightarrow \smd$ and $\smp \leftrightarrow \smq$. Recall that $\rho_\dforw$ moves every head with duck $\dforw$ to the right, and that $p_\dforw$ exchanges two adjacent letters if the head has duck $\dforw$: these let us define $\sigma_\dforw = \prod_{i=1}^\ell \rho_\dforw \circ p_\dforw$, the left shift of a whole cyclic tape of length $\ell$ for states in $\{\smb,\smp\} \times \{1,2\} \times \{\dforw\} \times \ghostset$. The automorphism $\sigma_\dforw$ has word norm $O(\ell)$.
    
    Denoting $C_1$ the condition $\len = \ell$ and $C_2$ the overflowing condition (i.e.\ the lexicographic comparison $< k[\ell]$), we successively apply $[f_{s,C_1},(\sigma_\dforw)^{-a}]$ (of word norm $O(\ell^2)$) and $[f_{s,C_1 \wedge C_2},(\sigma_\dforw)^{-1}]$ (of word norm $O(\ell^2)$).

    \medskip
    Let $f_{+2n}$ be the composition of all these steps. Then $f_{+2n}$ performs the addition of $2n$ in all the conveyor belts of length $< L_3$ with ducks $\dforw$ simultaneously, is the identity on ducks $\dback$, and has word norm $O(L^3)$.
  \end{enumerate}
  Then, conjugating $f_{+2n}$ by $f_{\encode}$ performs $(f_{\smart})^{2n}$ on heads with duck $\dforw$ on conveyor belts of length $< L_3$, and is the identity otherwise. Adding the same automorphism conjugated with $f_d$ and composing (for $d \in \Sym(H)$ the ducking involution), we obtain $(f_{\mathcal{M},\len < L_3})^{2n}$ with word norm $O(L^4)$.
  
  A similar formula exists for $(f_{\mathcal{M},\len \leq L_3})^{2n}$.
\end{proof}

\begin{remark}
  The word norm of this implementation of $(f_{\smart})^n$ is $O(\log n \cdot (\omega_{\leq}(\log n) + \omega_{+}(\log n)))$, where $\omega_{\leq}(N)$ is the complexity of the lexicographic inequalities on words of length $N$, and $\omega_{+}(N)$ is the complexity of the ternary addition on words of length $N$. While we could not find a way to perform $\omega_{+}(N)$ with complexity less than $O(N^3)$, it would be interesting to optimize this specific operation by itself.
\end{remark}

\section{Corollaries}
\label{sec:Corollaries}

We prove the other theorems listed in the introduction, all of which are straightforward corollaries of Theorem~\ref{thm:distortion-every-full-shift} (and its proof).

\subsection{Distortion in other subshifts}

\Sofic*

\begin{proof}
If $X$ is uncountable, then $\Aut(A^\Z) \leq \Aut(X)$~\cite{2018-Salo-NoteSAFullS,1990-KR}.
If $X$ is countable, then the proof of Proposition~2 in~\cite{SaTo12} shows that every automorphism $f \in \Aut(X)$
is either periodic or admits a \emph{spaceship}, namely a configuration of the form $x = ...uuuuvwww...$
which is not spatially periodic, and $f^n(x) = \sigma^m(x)$ for some $m \neq 0$. Clearly this prevents
distortion.
\end{proof}

Recall that the lower entropy dimension~\cite{2011-Meyerovitch} is 
\[ \underline{D}(X) = \liminf_{k \to \infty} \frac{\log(\log N_k(X))}{\log k} \]
We recall and prove Lemma~\ref{lem:LowComplexity} (which was used to prove Theorem~\ref{thm:LowerEntropyDim}).


\LowComplexity*

\begin{proof}
Suppose we have $|f^n| = O(\log^d n)$ for large $n$. Then the radius of $f^n$ is also $O(\log^d n)$. It follows that the trace subshift of $f$ has complexity function at most $n \mapsto N_{\lfloor C \log^d n \rfloor}(X)$ for some constant $C$. If $f$ is not of finite order, by the Morse-Hedlund theorem we must have $N_{\lfloor C \log^d n \rfloor}(X) > n$ for all $n$. Substituting $\lfloor e^{\sqrt[d]{n/C}} \rfloor$ for $n$ we get
$N_n(X) \geq e^{\sqrt[d]{an}}$ for some constant $a > 1$.
Substituting this lower bound into the definition of lower entropy dimension, we get $\underline{D}(X) \geq \frac1d$.
\end{proof}

\subsection{Distortion in the group of Turing machines}

We recall the definition of the group of Turing machines from~\cite{BaKaSa16}.

\begin{definition}
Let $n \geq 2$ and $k \geq 1$. Let $Y_n$ be the full shift on $n$ letters, and $X_k = \{x \in \{0,1,...,k\}^\Z \;|\; 0 \notin \{x_i, x_j\} \implies i = j\}$. Then
\[ \RTM(n, k) = \{f \in \Aut(Y_n \times X_k) \;|\; f|_{Y_n \times \{0^\Z\}} = \ID|_{Y_n \times \{0^\Z\}}\}. \]
\end{definition}


\TuringMachines*

\begin{proof}
We show that it immediately follows from the main theorem that $\RTM(18, 96)$ has a distortion element. We then explain how to conclude this for all $\RTM(n, k)$.

Recall that our automorphisms use the alphabet
  \[ \Sigma = \Big( \Gamma^2 \times \{+1,-1\} \Big) \sqcup \Big( (Q \times \Gamma) \times \Gamma \Big) \sqcup \Big( \Gamma \times (Q \times \Gamma) \Big) \]
  where $\Gamma = \origsmart{\Gamma}$ and $Q = \origsmart{Q} \times \{\dforw,\dback\} \times \{+1,-1\} \times \llbracket 0, 2 \rrbracket$.
  
  We may instead view this as
  \[ (\Gamma^2 \times \{+1, -1\}) \sqcup (Q \times \{\uparrow, \downarrow\} \times \Gamma^2), \]
  by grouping $((Q \times \Gamma) \times \Gamma)$ and $(\Gamma \times (Q \times \Gamma))$ together and replacing the choice with an arrow from $\{\uparrow, \downarrow\}$. Next, moving $\{\uparrow, \downarrow\}$ to the state and dropping $\{+1, -1\}$ out of it, we may view this as
  \[ (\Gamma^2 \times \{+1, -1\}) \sqcup (Q' \times \{+1, -1\} \times \Gamma^2), \]
for a certain set of states $Q'$ with $|Q'| = |Q| = 96$.

Consider the sofic subshift $Z$ where a symbol of $(Q' \times \{+1, -1\} \times \Gamma^2)$ can appear at most once. We clearly have a conjugacy $Z \cong X_{96} \times Y_{18}$, since $|\Gamma^2 \times \{+1,-1\}| = 18$.

It is easy to see that all of the generators $F$ defined in Lemma~\ref{lem:better-bounds-for-smart} fix $Z$. Furthermore, our generators only act near the head, so by definition this restricted action makes them elements of $\RTM(18, 96)$. The element $f_{\smart}$ coming from the SMART machine clearly has infinite order, since it acts as the SMART machine on infinite configurations. The word norm of $f_{\smart}$ w.r.t.\ $F$ of course cannot grow faster after restricting these elements to an invariant subspace, so we obtain that the subgroup of $\RTM(18, 96)$ generated by $F$ still has a distortion element, and the distortion is at least as bad as on the full shift.

\medskip
Now, we describe some minor modifications to the main construction that allow to conclude the result for $\RTM(n, k)$. In the construction of the main theorem, in place of the alphabet recalled above, take any finite set $S$ and use instead
\[ ((\Gamma^2 \times \{+1, -1\}) \sqcup S) \sqcup (Q' \times ((\Gamma^2 \times \{+1, 1\}) \sqcup S)). \]
Imagining elements of $S$ as new empty conveyor belts of size $1$, it is clear how most generators of $F$ should act, as their action is defined by how they act on finite conveyor belts. The element $\theta$ does not respect the conveyor belts, but it is also clear how it should act (now that we have moved $\{+1, -1\}$ out of the state onto the tape) -- it simply moves all heads.

Now recall that the only use of $\theta$ was to make sure that the automorphisms $f_{\tau_\mathrm{cb},\rightarrow t}$ (resp.\ $f_{\tau_\mathrm{cb},t \leftarrow}$) are in our group. These automorphisms apply the involution $(+1) \leftrightarrow (-1)$ on the sign carried either by $\Gamma^2 \times \{+1,-1\}$ or by a head, at distance $t$ to the right of the heads (resp.\ left of the heads, resp.\ both left and right of the heads). The correct extension of these is simply that the flip $(+1) \leftrightarrow (-1)$ does nothing on symbols in $S$. Then $\theta$ allows the implementation of natural analogs of the automorphisms $f_{\tau_\mathrm{cb},\rightarrow t}$ and $f_{\tau_\mathrm{cb},t \leftarrow}$ (with the exact same description).

Next, we recall that the automorphisms $f_{\tau_\mathrm{cb},\rightarrow t}$ are only used ``through conjugation'', i.e.\ they are used during the two-scale trick in very specific situations where we already know the head is on a large conveyor belt, in particular there are no $S$-symbols in the affected part. Thus the proof goes through without any modifications.

The introduction of $S$ with $|S| = t$ changes the group of Turing machines from $\RTM(18, 96)$ to $\RTM(18 + t, 96)$, and $\RTM(18 + t, 96)$ clearly embeds in $\RTM(18 + t , 96 + \ell)$ for any $\ell \geq 0$ (by behaving as the identity when the head is in one of the $\ell$ many new states). In particular for large enough $m$, we can pick $t, \ell$ such that $18 + t = n^m$ and $96 + \ell = k n^m$, to get a distortion element in a subgroup of $\RTM(n^m, k n^m)$. Finally, there is an embedding of $\RTM(n^m, k n^m)$ into $\RTM(n^m, k)$, by considering $m$-blocks of cells as individual cells, and considering the word on the tape at the origin as part of the state; and then $\RTM(n^m, k)$ embeds into $\RTM(n,k)$ by moving by $m$ steps at once.
\end{proof}

\subsection{Distortion in the Brin-Thompson group \texorpdfstring{$mV$}{mV}}

It was shown by Belk and Bleak that classical reversible Turing machines embed in the Brin-Thompson group $2V$. More generally, the group of Turing machines embeds in $2V$, and indeed this embedding is entirely transparent. For this, we recall the \emph{moving-tape model} of Turing machines.

\begin{definition}
Write $\RTM_{\mathrm{fix}}(n, k)$ for the family of homeomorphisms $f : \llbracket k \rrbracket \times \llbracket n \rrbracket^\Z \to \llbracket k \rrbracket \times \llbracket n \rrbracket^\Z $ such that for some radius $r \geq 1$ and local rule $f_{\mathrm{loc}} : \{0,1\}^r \times \{0,1\}^r \times \llbracket k \rrbracket \to \{0,1\}^* \times \{0,1\}^* \times \llbracket k \rrbracket$ we have
\[ f(xu.vy, a) = (xu'.v'y, b) \mbox{ whenever } f_{\mathrm{loc}}(u, v, a) = (u', v', b) \]
and for all $u, v$, $f_{\mathrm{loc}}(u, v) = (u', v', n)$ satisfies $|u'| + |v'| = 2r$.
\end{definition}

A proof of the following easy fact was outlined in~\cite{BaKaSa16}; one simply translates tape shifts into head movement into the opposite direction.

\begin{lemma}
The family of homeomorphisms $\RTM_{\mathrm{fix}}(n, k)$ forms a group under composition, and there is a canonical group isomorphism $\RTM_{\mathrm{fix}}(n, k) \cong \RTM(n, k)$.
\end{lemma}

\begin{lemma}
\label{lem:TMin2V}
The group $\RTM(n, k)$ embeds in the Brin-Thompson group $mV$ for all $m \geq 2, n \geq 2, k \geq 1$.
\end{lemma}

\begin{proof}
The group $2V$ embeds in $mV$, so it is enough to show this for $m = 2$. This is very similar to the proof in Belk-Bleak~\cite{BeBl17}, and was also essentially stated in~\cite{BaKaSa16}, so we only outline the proof. First, it is enough to embed $\RTM(n, 1)$, since $\RTM(n, k)$ embeds in $\RTM(n, k + \ell)$ for all $\ell \geq 0$, thus in particular in $\RTM(n, n^m)$ for sufficiently large $m$, and this group is isomorphic to $\RTM(n, 1)$ (see the end of the proof of Theorem~\ref{thm:TuringMachines}).

Now pick a complete suffix code $C \subset \{0,1\}^*$ of cardinality $n$, and a complete prefix code $D \in \{0,1\}^*$ of cardinality $n$. One can uniquely parse any $x.y \in \{0,1\}^\Z$ as $\cdots u_{-2} u_{-1}. v_0 v_1 v_2 \cdots$ with $u_{-i} \in C, v_i \in D$ for all applicable $i$, which gives a homeomorphism $\phi : \{0,1\}^\Z \to \llbracket n \rrbracket^\Z$. For $g \in \RTM(n, 1)$, the map $g^\phi$ is easily seen to be in $2V$, so this gives a group-theoretic embedding of $\RTM(n, 1)$ into $2V$.
\end{proof}

Dynamically, the proof gives a topological conjugacy between the natural action of $\RTM(n, 1)$ and the natural action of the subgroup of $2V$ that respects the encoding. 


\TwoV*

\begin{proof}
The embedding of the group $\RTM(n, k)$ in particular embeds the group where we constructed a distortion element. Adding the finite generating set of $mV$ clearly cannot make the element less distorted.
\end{proof}

\section{Open questions}
\label{sec:OpenQuestions}

\begin{question}
Are there ever distortion elements in $\Aut(X)$ for $X$ a minimal subshift? What about $X$ of zero entropy?
\end{question}

Minimal subshifts are interesting, because at present we do not know any restrictions on their automorphism groups, yet all known examples are locally virtually abelian. Zero entropy is interesting because on the one hand there are many known examples of interesting behaviors in their automorphism groups, but \cite{2018-CFKP} shows that exponential distortion is impossible.

Next, it seems worth recalling the remaining parts of~\cite[Questions~5.1--3]{2018-CFKP}.

\begin{question}
Are there more natural (in terms of the group structure) subgroups having distortion elements in $\Aut(A^\Z)$, or even in $\Aut(X)$, where $X \subset A^G$ is an arbitrary subshift on an abelian group $G$? For example, can we embed the Heisenberg group (or more generally $\mathrm{SL}(3, \Z)$), or the Baumslag-Solitar group $\mathrm{BS}(1,n)$?
\end{question}

The Heisenberg group was originally asked about in~\cite{1990-KR} (though not explicitly due to distortion concerns). One important note about this group is that every (infinite f.g.\ torsion-free nonabelian) nilpotent group contains a copy of it. Nilpotent groups are considered some of the simplest (in the non-technical sense) kinds of infinite groups after abelian groups; in the case of automorphism groups of subshifts we can implement a wide variety of behaviors, yet embeddability of nilpotent groups remains a mystery.

Embedding the Baumslag-Solitar group is the same as finding an element of infinite order that is conjugate to a higher power of itself. We believe the SMART machine does not have this property (before or after an embedding into $\Aut(A^\Z)$), but we have not proved this.

A slightly more abstract question of interest is whether there exists an amenable subgroup of the automorphism group of a subshift which has distortion elements. One thing amenability rules out is groups that are too large, e.g.\ f.g.-universal subgroups. The Heisenberg group and $\mathrm{BS}(1,n)$ are of course amenable (even solvable).

\begin{question}
Can a one-sided subshift have distortion elements in its automorphism group? Does $\Aut(A^\N)$ have distortion elements?
\end{question}

Note that $\Aut(A^\N)$ is simply the subgroup of $\Aut(A^\Z)$ consisting of automorphisms $f$ such that both $f$ and $f^{-1}$ have ``one-sided radius'', i.e.\ $f(x)_i$, $f^{-1}(x)_i$ depend only on $x_{\llbracket i, i+r \rrbracket}$ for some $r$. We do not even know whether one-sided automorphisms of subshifts can have sublinear radius growth.

\medskip

As mentioned in the introduction, the true word norm growth of our automorphism is between $\Omega(\log n)$ and $O(\log^4 n)$. It would be of great interest to pinpoint the growth up to a multiplicative constant for our automorphism, or for any other automorphism with sublinear growth.

\begin{question}
What are the distortion functions (word norm growth rates) of elements of $\Aut(A^\Z)$ (or $\Aut(X)$ for more general subshifts)?
\end{question}

Of course, in the case of a non-finitely generated group like $\Aut(A^\Z)$, the distortion function depends on the finite generating set chosen. While it is of interest to implement distortion functions with respect to subgroups, a more natural object to consider is the directed set of distortion functions with increasing generating sets, and especially the eventual behavior as the generating set increases.

Similar questions can be asked about groups of Turing machines and the Brin-Thompson $2V$, where we also exhibit elements whose word norm grows polylogarithmically, but have no further control on the distortion.

A natural idea for getting control over the distortion function would be to use, in place of SMART, a general-purpose Turing machine, which is made to have sublinear movement using the reversible Hooper trick from~\cite{2008-KO} (and finally embedded in some natural way into the automorphism group of a subshift). It is known that this construction always produces Turing machines with zero Lyapunov exponents, i.e.\ with sublinear movement~\cite{2017-GS,Je14}.

\begin{question}
Does the Kari-Ollinger construction in~\cite{2008-KO} always produce distortion elements?
\end{question}



\bibliographystyle{plain}
\bibliography{bib}{}

\end{document}